\documentclass[12pt, reqno, a4paper]{amsart}

\usepackage{amssymb}
\usepackage{hyperref}

\usepackage[english]{babel}

\theoremstyle{plain}

\newtheorem*{corollary}{Corollary}
\newtheorem{lemma}{Lemma}
\newtheorem{theorem}{Theorem}
\newtheorem*{conjecture}{Conjecture}

\theoremstyle{remark}
\newtheorem*{remark}{Remark}

\theoremstyle{definition}

\newtheorem{example}{Example}

\DeclareMathOperator{\Id}{Id}

\DeclareMathOperator{\id}{id}
\DeclareMathOperator{\ch}{char}

\DeclareMathOperator{\ad}{ad}
\DeclareMathOperator{\GL}{GL}
\DeclareMathOperator{\tr}{tr}
\DeclareMathOperator{\height}{ht}
\DeclareMathOperator{\End}{End}
\DeclareMathOperator{\Aut}{Aut}

\DeclareMathOperator{\length}{length}
\DeclareMathOperator{\sign}{sign}
\DeclareMathOperator{\Alt}{Alt}
\DeclareMathOperator{\Ann}{Ann}
\DeclareMathOperator{\Hom}{Hom}
\DeclareMathOperator{\PIexp}{PIexp}

\begin{document}

\title{Amitsur's conjecture for polynomial $H$-identities of $H$-module Lie algebras}

\author{A.\,S.~Gordienko}

\address{Memorial University of Newfoundland, St. John's, NL, Canada}
\email{asgordienko@mun.ca}

\begin{abstract}
Consider a finite dimensional $H$-module Lie algebra $L$ over a field of characteristic~$0$
where $H$ is a Hopf algebra.
We prove the analog of Amitsur's conjecture
on asymptotic behavior for codimensions of polynomial $H$-identities of $L$
under some assumptions on $H$.
In particular, the conjecture holds when $H$ is finite dimensional semisimple.
As a consequence, we obtain the analog of Amitsur's conjecture
for graded codimensions of any finite
dimensional Lie algebra graded by an arbitrary group
and for $G$-codimensions of any finite
dimensional Lie algebra with a rational action of a reductive affine algebraic group $G$
by automorphisms and anti-automorphisms.
\end{abstract}


\thanks{
Supported by post doctoral fellowship
from Atlantic Association for Research
in Mathematical Sciences (AARMS), Atlantic Algebra Centre (AAC),
Memorial University of Newfoundland (MUN), and
Natural Sciences and Engineering Research Council of Canada (NSERC)}

\maketitle

\section{Introduction} 

In the 1980's, a conjecture about the asymptotic behaviour
of codimensions of ordinary polynomial identities was made
by S.A.~Amitsur. Amitsur's conjecture was proved in 1999 by
A.~Giambruno and M.V.~Zaicev~\cite[Theorem~6.5.2]{ZaiGia} for associative algebras, in 2002 by M.V.~Zaicev~\cite{ZaiLie}
 for finite dimensional Lie algebras, and in 2011 by A.~Giambruno,
 I.P.~Shestakov, M.V. Zaicev for finite dimensional Jordan and alternative
 algebras~\cite{GiaSheZai}. In 2011 the author proved its analog
 for polynomial identities of finite dimensional representations of Lie
 algebras~\cite{ASGordienko}.
 
 Alongside with ordinary polynomial
identities of algebras, graded polynomial identities, $G$- and
$H$-identities are
important too~\cite{BahtGiaZai, BahturinLinchenko, BahtZaiGraded, BahtZaiGradedExp, BahtZaiSehgal, 
 BereleHopf, Linchenko}.
 Usually, to find such identities is easier
  than to find the ordinary ones. Furthermore, each of these types of identities completely determines the ordinary polynomial identities. 
Therefore the question arises whether the conjecture
holds for graded codimensions, $G$- and $H$-codimensions.
The analog of Amitsur's conjecture
for codimensions of graded identities was proved in 2010--2011 by
E.~Aljadeff,  A.~Giambruno, and D.~La~Mattina~\cite{AljaGia, AljaGiaLa, GiaLa}
  for all associative PI-algebras graded by a finite group.
   As a consequence, they proved the analog  of the conjecture for $G$-codimensions
   for any associative PI-algebra with an action of a finite Abelian group $G$ by automorphisms.
The case when $G=\mathbb Z_2$
acts on a finite dimensional associative algebra by automorphisms and anti-automorphisms
(i.e. polynomial identities with involution)
 was considered by A.~Giambruno and
 M.V.~Zaicev~\cite[Theorem~10.8.4]{ZaiGia}
 in 1999. In 2012 the author~\cite{ASGordienko3} proved
 the analog of Amitsur's conjecture for polynomial
 $H$-identities of finite dimensional associative algebras
 with a generalized $H$-action. As a consequence, 
 the analog of Amitsur's conjecture was proved for
 $G$-codimensions of finite dimensional associative algebras with
 an action of an arbitrary finite group $G$ by automorphisms and anti-automorphisms,
 and for $H$-codimensions of finite dimensional $H$-module associative algebras for a
 finite dimensional semisimple Hopf algebra~$H$.
     
      In 2011 the author~\cite{ASGordienko2} proved 
  the analog of Amitsur's conjecture for graded polynomial identities
  of finite dimensional
Lie algebras graded by a finite Abelian group
and  for $G$-identities
of finite dimensional
Lie algebras with an action of any finite group (not necessarily Abelian).

This article is concerned with the analog of Amitsur's conjecture for 
graded codimensions of Lie algebras graded by an arbitrary group (Subsection~\ref{SubsectionGraded}),
$G$-codimensions of Lie algebras with a rational action of a reductive affine algebraic group
$G$ by automorphisms and anti-automorphisms (Subsection~\ref{SubsectionG}), and
  $H$-codimensions of $H$-module Lie algebras
(Subsections~\ref{SubsectionHopf}--\ref{SubsectionHopfPIexp}).
The results of Subsections~\ref{SubsectionGraded} and \ref{SubsectionG} (including the results of~\cite{ASGordienko2} as a special case)
are derived  in Section~\ref{SectionAppl} from the results of Subsection~\ref{SubsectionHnice}.

 In Subsection~\ref{SubsectionHopfPIexp} we provide an explicit formula
for the Hopf PI-exponent that is a natural generalization of the formula
for the ordinary PI-exponent~\cite[Definition~2]{ZaiLie}. Of course, this formula can be used for
graded codimensions and $G$-codimensions as well. The formula has immediate applications.
In particular, in Section~\ref{SectionExamples} we apply it 
to calculate  the Hopf PI-exponent for several classes of algebras and to provide a criterion of $H$-simplicity. 

The results obtained provide a useful tool to study polynomial identities in Lie algebras and hence to study Lie algebras themselves.

\subsection{Graded polynomial identities and their codimensions}\label{SubsectionGraded}

Let $G$ be a group.
 Denote by 
$F\lbrace X^{\mathrm{gr}} \rbrace$ the absolutely free nonassociative algebra on the countable set $X^{\mathrm{gr}}=\bigcup_{g \in G}X^{(g)}$, $X^{(g)} = \{ x^{(g)}_1,
x^{(g)}_2, \ldots \}$,
over a field $F$,  i.e. the algebra of nonassociative noncommutative polynomials
 in variables from $X^{\mathrm{gr}}$. 
 
The algebra $F\lbrace X^{\mathrm{gr}} \rbrace$ has the following natural $G$-grading. 
 The indeterminates from $X^{(g)}$ are said to be homogeneous of degree
$g$. The $G$-degree of a monomial $x^{(g_1)}_{i_1} \dots x^{(g_t)}_{i_t} \in F\lbrace X^{\mathrm{gr}} \rbrace$   (arbitrary arrangement of brackets) is defined to
be $g_1 g_2 \dots g_t$, as opposed to its total degree, which is defined to be $t$. Denote by
$F\lbrace X^{\mathrm{gr}} \rbrace^{(g)}$ the subspace of the algebra $F\lbrace X^{\mathrm{gr}} \rbrace$ spanned by all the monomials having
$G$-degree $g$. Notice that $F\lbrace X^{\mathrm{gr}} \rbrace^{(g)} F\lbrace X^{\mathrm{gr}} \rbrace^{(h)} \subseteq F\lbrace X^{\mathrm{gr}} \rbrace^{(gh)}$, for every $g, h \in G$. It follows that
$$F\lbrace X^{\mathrm{gr}} \rbrace=\bigoplus_{g\in G} F\lbrace X^{\mathrm{gr}} \rbrace^{(g)}$$ is a $G$-grading.

Consider the intersection $I$ of all graded ideals of $F\lbrace X^{\mathrm{gr}} \rbrace$
containing the set \begin{equation}\label{EqSetOfGrGen}
\bigl\lbrace u(vw)+v(wu)+w(uv) \mid u,v,w \in  F\lbrace X^{\mathrm{gr}} \rbrace\bigr\rbrace \cup\bigl\lbrace u^2 \mid u \in  F\lbrace X^{\mathrm{gr}} \rbrace\bigr\rbrace.
\end{equation}
 Then $L(X^{\mathrm{gr}}) := F\lbrace X^{\mathrm{gr}} \rbrace/I$
is \textit{the free $G$-graded Lie algebra}
on $X^{\mathrm{gr}}$, i.e. for any $G$-graded Lie algebra $L=\bigoplus_{g\in G} L^{(g)}$ 
and a map $\psi \colon X^{\mathrm{gr}} \to L$ such that $\psi(X^{(g)}) \subseteq L^{(g)}$, there exists a unique homomorphism $\bar\psi \colon L(X^{\mathrm{gr}}) \to L$
of graded algebras such that $\bar\psi\bigr|_{X^{\mathrm{gr}}} =\psi$. 


We use the commutator notation for the multiplication in $L(X^{\mathrm{gr}})$
and other Lie algebras. In this notation, $L=\bigoplus_{g\in G} L^{(g)}$
is \textit{graded} if $[L^{(g)}, L^{(h)}]\subseteq L^{(gh)}$.

\begin{remark} If $G$ is Abelian, then $L(X^{\mathrm{gr}})$ is the ordinary
free Lie algebra with free generators from $X^{\mathrm{gr}}$ since the ordinary ideal of 
$F\lbrace X^{\mathrm{gr}} \rbrace$ generated by~(\ref{EqSetOfGrGen})
is already graded.
However, if $gh \ne hg$ for some $g, h \in G$,  then $[x^{(g)}_i, x^{(h)}_j]=0$
in $L(X^{\mathrm{gr}})$ for all  $i,j \in \mathbb N$.
\end{remark}

  Let $f=f(x^{(g_1)}_{i_1}, \dots, x^{(g_t)}_{i_t}) \in L (X^{\mathrm{gr}})$.
We say that $f$ is
a \textit{graded polynomial identity} of
 a $G$-graded Lie algebra $L=\bigoplus_{g\in G}
L^{(g)}$
and write $f\equiv 0$
if $f(a^{(g_1)}_{i_1}, \dots, a^{(g_t)}_{i_t})=0$
for all $a^{(g_j)}_{i_j} \in L^{(g_j)}$, $1 \leqslant j \leqslant t$.
In other words, $f$ is
a graded polynomial identity if for any homomorphism
$\psi\colon L(X^{\mathrm{gr}}) \to L$ of graded algebras
we have $\psi(f)=0$. 
  The set $\Id^{\mathrm{gr}}(L)$ of graded polynomial identities of
   $L$ is
a graded ideal of $L(X^{\mathrm{gr}})$.
The case of ordinary polynomial identities is included
for the trivial group $G=\lbrace e \rbrace$.

\begin{example}\label{ExampleIdGr}
 Let $G=\mathbb Z_2 = \lbrace \bar 0, \bar 1 \rbrace$,
$\mathfrak{gl}_2(F)=\mathfrak{gl}_2(F)^{(\bar 0)}\oplus \mathfrak{gl}_2(F)^{(\bar 1)}$
where $\mathfrak{gl}_2(F)^{(\bar 0)}=\left(
\begin{array}{cc}
F & 0 \\
0 & F
\end{array}
 \right)$ and $\mathfrak{gl}_2(F)^{(\bar 1)}=\left(
\begin{array}{cc}
0 & F \\
F & 0
\end{array}
 \right)$. Then  $[x^{(\bar 0)}, y^{(\bar 0)}]
\in \Id^{\mathrm{gr}}(\mathfrak{gl}_2(F))$.
\end{example}

\begin{example}[\cite{PaReZai}]\label{ExampleIdGr2}
Consider $L = \left\lbrace\left(\begin{array}{cc} \mathfrak{gl}_2(F) & 0 \\ 0 & \mathfrak{gl}_2(F) \end{array}\right)\right\rbrace \subseteq \mathfrak{gl}_4(F)$
with the following grading by the third symmetric group
$S_3$:
$$L^{(e)} = \left\lbrace\left(\begin{array}{cccc}
 \alpha & 0 &  0 & 0 \\
  0 & \beta & 0 & 0 \\
   0 & 0 & \gamma & 0 \\
    0 & 0 & 0 & \lambda
 \end{array}\right)\right\rbrace,\quad
L^{\bigl((12)\bigr)} = \left\lbrace\left(\begin{array}{cccc} 0 & \beta & 0 & 0 \\
 \alpha & 0 &  0 & 0 \\
  0 & 0 & 0 & 0 \\
   0 & 0 & 0 & 0
 \end{array}\right)\right\rbrace,$$ $$
L^{\bigl((23)\bigr)} = \left\lbrace\left(\begin{array}{cccc}
  0 & 0 & 0 & 0 \\
   0 & 0 & 0 & 0 \\
  0 & 0 & 0 & \lambda \\
   0 & 0 & \gamma & 0 \\
 \end{array}\right)\right\rbrace,$$
 the other components are zero, $\alpha,\beta,\gamma,\lambda \in F$. 
  Then  $\left[x^{\bigl((12)\bigr)},y^{\bigl((23)\bigr)}\right]\in \Id^{\mathrm{gr}}(L)$ 
  and $[x^{(e)},y^{(e)}]
\in \Id^{\mathrm{gr}}(L)$. In fact, $\left[x^{\bigl((12)\bigr)},y^{\bigl((23)\bigr)}\right]=0$
in $L(X^\mathrm{gr})$.
\end{example}

Let $S_n$ be the $n$th symmetric group, $n\in\mathbb N$, and
 $$V^{\mathrm{gr}}_n:=\langle [x^{(g_1)}_{\sigma(1)}, x^{(g_2)}_{\sigma(2)},
 \ldots, x^{(g_n)}_{\sigma(n)}] \mid g_i \in G, \sigma \in S_n \rangle_F.$$
(All long commutators in the article are left-normed, although this is not important
in this particular case in virtue of the Jacobi identity.)
  The non-negative integer
 $$c^{\mathrm{gr}}_n(L) := \dim
  \left(\frac {V^{\mathrm{gr}}_n}{V^{\mathrm{gr}}_n \cap
   \Id^{\mathrm{gr}}(L)}\right)$$
is called the $n$th {\itshape codimension
of graded polynomial identities} or the $n$th {\itshape graded codimension} of~$L$.

\begin{remark} Let $\tilde G \supseteq G$ be another group.
Denote by $L(X^{\widetilde{\mathrm{gr}}})$, $\Id^{\widetilde{\mathrm{gr}}}(L)$, $V^{\widetilde{\mathrm{gr}}}_n$, 
$c^{\widetilde{\mathrm{gr}}}_n(L)$ the objects corresponding to the $\tilde G$-grading. Let $I$ be the ideal of $L(X^{\widetilde{\mathrm{gr}}})$
generated by $x^{(g)}_j$, $j\in\mathbb N$, $g \notin G$.
We can identify $L(X^\mathrm{gr})$ with the subalgebra in 
$L(X^{\widetilde{\mathrm{gr}}})$. Then $L(X^{\widetilde{\mathrm{gr}}}) = L(X^\mathrm{gr}) \oplus I$, 
$\Id^{\widetilde{\mathrm{gr}}}(L) = \Id^{\mathrm{gr}}(L) \oplus I$, $V^{\widetilde{\mathrm{gr}}}_n=
V^{\mathrm{gr}}_n \oplus (V^{\widetilde{\mathrm{gr}}}_n \cap I)$,
$$V^{\widetilde{\mathrm{gr}}}_n \cap  \Id^{\widetilde{\mathrm{gr}}}(L) =
(V^{\mathrm{gr}}_n \cap \Id^{\mathrm{gr}}(L)) \oplus (V^{\widetilde{\mathrm{gr}}}_n \cap I)$$
(direct sums of subspaces). Hence $c^{\widetilde{\mathrm{gr}}}_n(L) = c^{\mathrm{gr}}_n(L)$
for all $n\in\mathbb N$. In particular, we can always replace the grading group with 
the subgroup generated by the elements corresponding to the nonzero components.
\end{remark}

The analog of Amitsur's conjecture for graded codimensions can be formulated
as follows.

\begin{conjecture} There exists
 $\PIexp^{\mathrm{gr}}(L):=\lim\limits_{n\to\infty} \sqrt[n]{c^\mathrm{gr}_n(L)}$
 which is a nonnegative integer.
\end{conjecture}

\begin{remark}
I.B.~Volichenko~\cite{Volichenko} gave an example
of an infinite dimensional Lie algebra~$L$ with
a nontrivial polynomial identity for which the
 growth of codimensions~$c_n(L)$ of ordinary polynomial identities
 is overexponential.
 M.V.~Zaicev and S.P.~Mishchenko~\cite{VerZaiMishch, ZaiMishch}
  gave an example
of an infinite dimensional Lie algebra~$L$
 with a nontrivial polynomial identity
 such that
there exists fractional
  $\PIexp(L):=\lim\limits_{n\to\infty} \sqrt[n]{c_n(L)}$.
\end{remark}

\begin{theorem}\label{TheoremMainGr}
Let $L$ be a finite dimensional non-nilpotent Lie algebra
over a field $F$ of characteristic $0$, graded by an arbitrary group $G$. Then
there exist constants $C_1, C_2 > 0$, $r_1, r_2 \in \mathbb R$, $d \in \mathbb N$
such that $C_1 n^{r_1} d^n \leqslant c^{\mathrm{gr}}_n(L) \leqslant C_2 n^{r_2} d^n$
for all $n \in \mathbb N$.
\end{theorem}
\begin{corollary}
The above analog of Amitsur's conjecture holds for such codimensions.
\end{corollary}
\begin{remark}
If $L$ is nilpotent, i.e. $[x_1, \ldots, x_p]\equiv 0$ for some $p\in\mathbb N$, then, by the Jacobi identity, $V^{\mathrm{gr}}_n \subseteq \Id^{\mathrm{gr}}(L)$ and $c^{\mathrm{gr}}_n(L)=0$ for all $n \geqslant p$.
\end{remark}

\begin{theorem}\label{TheoremMainGrSum}
Let $L=L_1 \oplus \ldots \oplus L_s$ (direct sum of graded ideals) be a finite dimensional Lie algebra
over a field $F$ of characteristic $0$, graded by an arbitrary group $G$.
 Then $\PIexp^\mathrm{gr}(L)=\max_{1 \leqslant i \leqslant s}
\PIexp^\mathrm{gr}(L_i)$.
\end{theorem}

Theorems~\ref{TheoremMainGr} and~\ref{TheoremMainGrSum} will be obtained as consequences of Theorems~\ref{TheoremMainH} and~\ref{TheoremMainHSum} below in Subsection~\ref{SubsectionApplGr}.

\subsection{Polynomial $G$-identities and their codimensions}\label{SubsectionG}
 
Let $L$ be a Lie algebra over a field $F$.
Recall that $\psi \in \GL(L)$ is an {\itshape
automorphism} of $L$ if $\psi([a,b]) = [\psi(a), \psi(b)]$
for all $a,b \in L$
and {\itshape
anti-automorphism} of $L$ if $\psi([a,b]) = [\psi(b), \psi(a)]$
for all $a, b \in L$.
Obviously, $\psi$ is an anti-automorphism if and only if
$(-\psi)$ is an automorphism.
 Automorphisms of $L$
form the group denoted by $\Aut(L)$.
Automorphisms and anti-automorphisms of $A$
form the group denoted by $\Aut^{*}(L)$.
Note that $\Aut(L)$ is a normal subgroup of $\Aut^{*}(L)$
of index $2$.

Let $G$ be a group
with a fixed (normal) subgroup $G_0$ of index ${}\leqslant 2$.
  We say that a Lie algebra $L$ is an \textit{algebra with $G$-action}
  or a \textit{$G$-algebra}
   if $L$ is endowed with a homomorphism $\zeta \colon G \to
  \Aut^{*}(L)$ such that $\zeta^{-1}(\Aut(L))=G_0$.
  Note that if $G$ is an affine algebraic group acting rationally on $L$
  then $G_0$ is closed.
  
   We use the exponential notation
  for the action of a group and its group algebra,
  and write $a^g := \zeta(g)(a)$ for $g\in G$ and $a\in L$.
  
 Denote by $L( X | G )$
the free Lie algebra over $F$ with free formal generators $x^g_j$, $j\in\mathbb N$,
 $g \in G$. Here $X := \lbrace x_1, x_2, x_3, \ldots \rbrace$, $x_j := x_j^1$. Define
 $$[x_{i_1}^{g_1}, x_{i_2}^{g_2},\ldots, x_{i_{n-1}}^{g_{n-1}},
  x_{i_n}^{g_n}]^g :=
[x_{i_1}^{gg_1}, x_{i_2}^{gg_2},\ldots, x_{i_{n-1}}^{gg_{n-1}}, x_{i_n}^{gg_n}]
\text { for } g \in G_0, $$
 $$[x_{i_1}^{g_1}, x_{i_2}^{g_2},\ldots, x_{i_{n-1}}^{g_{n-1}},
  x_{i_n}^{g_n}]^g :=
[x_{i_n}^{gg_n}, [x_{i_{n-1}}^{gg_{n-1}},  \ldots, [x_{i_2}^{gg_2},x_{i_1}^{gg_1}]\ldots]
\text { for } g \in G\backslash G_0.$$
 Then $L( X | G )$ becomes \textit{the free Lie $G$-algebra} with
 free generators $x_j$, $j \in \mathbb N$. We call its elements Lie
 $G$-polynomials.
 
 Let $L$ be a Lie  $G$-algebra over $F$. A $G$-polynomial
 $f(x_1, \ldots, x_n)\in L( X | G )$
 is a \textit{$G$-identity} of $L$ if $f(a_1, \ldots, a_n)=0$
for all $a_i \in L$. In this case we write
$f \equiv 0$.
The set $\Id^{G}(L)$ of all $G$-identities
of $L$ is an ideal in $L( X | G )$ invariant under $G$-action.

  \begin{example}\label{ExampleIdG} Consider $\psi \in \Aut(\mathfrak{gl}_2(F))$
defined by the formula $$\left(
\begin{array}{cc}
a & b \\
c & d
\end{array}
 \right)^\psi := \left(
\begin{array}{rr}
a & -b \\
-c & d
\end{array}
 \right).$$
Then $[x+x^{\psi},y+y^{\psi}]\in \Id^{G}(\mathfrak{gl}_2(F))$
where
$G=\langle \psi \rangle \cong \mathbb Z_2$.
\end{example}

Denote by $V^G_n$ the space of all multilinear Lie $G$-polynomials
in $x_1, \ldots, x_n$, i.e.
$$V^{G}_n = \langle [x^{g_1}_{\sigma(1)},
x^{g_2}_{\sigma(2)},\ldots, x^{g_n}_{\sigma(n)}]
\mid g_i \in G, \sigma\in S_n \rangle_F.$$
Then the number $c^G_n(L):=\dim\left(\frac{V^G_n}{V^G_n \cap \Id^G(L)}\right)$
is called the $n$th \textit{codimension of polynomial $G$-identities}
or the $n$th \textit{$G$-codimension} of $L$.

The analog of Amitsur's conjecture for $G$-codimensions can be formulated
as follows.

\begin{conjecture} There exists
 $\PIexp^G(L):=\lim\limits_{n\to\infty} \sqrt[n]{c^G_n(L)} \in \mathbb Z_+$.
\end{conjecture}

We claim that the following theorem holds:
 \begin{theorem}\label{TheoremMainGAffAlg}
Let $L$ be a finite dimensional non-nilpotent Lie algebra
over an algebraically closed field $F$ of characteristic $0$. Suppose a reductive affine algebraic group
  $G$ acts on $L$ rationally by automorphisms and anti-automorphisms. Then there exist
 constants $C_1, C_2 > 0$, $r_1, r_2 \in \mathbb R$,
  $d \in \mathbb N$ such that
   $C_1 n^{r_1} d^n \leqslant c^{G}_n(L)
    \leqslant C_2 n^{r_2} d^n$ for all $n \in \mathbb N$.
\end{theorem}

In addition, we derive the following theorem (which one could regard as a special case
of Theorem~\ref{TheoremMainGAffAlg}).
 \begin{theorem}[{\cite[Theorem~2]{ASGordienko2}}]\label{TheoremMainGFin}
Let $L$ be a finite dimensional non-nilpotent Lie algebra
over a field $F$ of characteristic $0$. Suppose a finite
  group $G$ acts on $L$ by automorphisms and anti-automorphisms. Then there exist
 constants $C_1, C_2 > 0$, $r_1, r_2 \in \mathbb R$,
  $d \in \mathbb N$ such that
   $C_1 n^{r_1} d^n \leqslant c^{G}_n(L)
    \leqslant C_2 n^{r_2} d^n$ for all $n \in \mathbb N$.
\end{theorem}

\begin{corollary}
The above analog of Amitsur's conjecture holds for
 such codimensions.
\end{corollary}

\begin{remark}
If $L$ is nilpotent, i.e. $[x_1, \ldots, x_p]\equiv 0$ for some $p\in\mathbb N$, then
$V^{G}_n \subseteq \Id^{G}(L)$ and $c^G_n(L)=0$ for all $n \geqslant p$.
\end{remark}

\begin{theorem}\label{TheoremMainGAffAlgSum}
Let $L=L_1 \oplus \ldots \oplus L_s$ (direct sum of ideals) be a finite dimensional Lie algebra
over an algebraically closed field $F$ of characteristic $0$. Suppose a reductive affine algebraic group
  $G$ acts on $L$ rationally by automorphisms and anti-automorphisms and the ideals 
$L_i$ are $G$-invariant. Then $\PIexp^G(L)=\max_{1 \leqslant i \leqslant s}
\PIexp^G(L_i)$.
\end{theorem}
\begin{theorem}\label{TheoremMainGFinSum}
Let $L=L_1 \oplus \ldots \oplus L_s$ (direct sum of ideals) be a finite dimensional Lie algebra
over a field $F$ of characteristic $0$.  Suppose a finite
  group $G$ acts on $L$ by automorphisms and anti-automorphisms, and the ideals 
$L_i$ are $G$-invariant. Then $$\PIexp^G(L)=\max_{1 \leqslant i \leqslant s}
\PIexp^G(L_i).$$
\end{theorem}

Theorems~\ref{TheoremMainGAffAlg}, \ref{TheoremMainGFin}, \ref{TheoremMainGAffAlgSum}, and \ref{TheoremMainGFinSum} will be obtained as consequences of Theorems~\ref{TheoremMainH}
and~\ref{TheoremMainHSum}  below in Subsection~\ref{SubsectionApplG}.

\subsection{Polynomial $H$-identities and their codimensions}\label{SubsectionHopf}

Analogously, one can consider polynomial $H$-identities
for $H$-module Lie algebras. 
The connection between graded, $G$- and $H$-identities will become clear
in Section~\ref{SectionAppl}.

Let $H$ be a Hopf algebra over a field $F$.
An algebra $A$ is an \textit{$H$-module algebra} if $A$
is endowed with a left $H$-action $h\otimes a \mapsto ha$ or, equivalently, with a homomorphism $H \to \End_F(A)$,  such that
$h(ab)=(h_{(1)}a)(h_{(2)}b)$
for all $h \in H$, $a,b \in A$.
Here we use Sweedler's notation
$\Delta h = h_{(1)} \otimes h_{(2)}$ where $\Delta$ is the comultiplication
in $H$.
We refer the reader to~\cite{Abe, Danara, Montgomery, Sweedler}
   for an account
  of Hopf algebras and algebras with Hopf algebra actions.

 In particular, a Lie algebra $L$
is an \textit{$H$-module Lie algebra}
if \begin{equation}\label{EqHmoduleLieAlgebra}
h\,[a,b]=[h_{(1)}a, h_{(2)}b]
\text{ for all }h \in H,\ a,b \in L.
\end{equation}

  Let $F \lbrace X \rbrace$ be the absolutely free nonassociative algebra
   on the set $X := \lbrace x_1, x_2, x_3, \ldots \rbrace$.
  Then $F \lbrace X \rbrace = \bigoplus_{n=1}^\infty F \lbrace X \rbrace^{(n)}$
  where $F \lbrace X \rbrace^{(n)}$ is the linear span of all monomials of total degree $n$.
   Let $H$ be a Hopf algebra over a field $F$. Consider the algebra $$F \lbrace X | H\rbrace
   :=  \bigoplus_{n=1}^\infty H^{{}\otimes n} \otimes F \lbrace X \rbrace^{(n)}$$
   with the multiplication $(u_1 \otimes w_1)(u_2 \otimes w_2):=(u_1 \otimes u_2) \otimes w_1w_2$
   for all $u_1 \in  H^{{}\otimes j}$, $u_2 \in  H^{{}\otimes k}$,
   $w_1 \in F \lbrace X \rbrace^{(j)}$, $w_2 \in F \lbrace X \rbrace^{(k)}$.
We use the notation $$x^{h_1}_{i_1}
x^{h_2}_{i_2}\ldots x^{h_n}_{i_n} := (h_1 \otimes h_2 \otimes \ldots \otimes h_n) \otimes x_{i_1}
x_{i_2}\ldots x_{i_n}$$ (the arrangements of brackets on $x_{i_j}$ and on $x^{h_j}_{i_j}$
are the same). Here $h_1 \otimes h_2 \otimes \ldots \otimes h_n \in H^{{}\otimes n}$,
$x_{i_1} x_{i_2}\ldots x_{i_n} \in F \lbrace X \rbrace^{(n)}$. 

Note that if $(\gamma_\beta)_{\beta \in \Lambda}$ is a basis in $H$, 
then $F \lbrace X | H\rbrace$ is isomorphic to the absolutely free nonassociative algebra over $F$ with free formal  generators $x_i^{\gamma_\beta}$, $\beta \in \Lambda$, $i \in \mathbb N$.
 
    Define on $F \lbrace X | H\rbrace$ the structure of a left $H$-module
   by $$h\,(x^{h_1}_{i_1}
x^{h_2}_{i_2}\ldots x^{h_n}_{i_n})=x^{h_{(1)}h_1}_{i_1}
x^{h_{(2)}h_2}_{i_2}\ldots x^{h_{(n)}h_n}_{i_n},$$
where $h_{(1)}\otimes h_{(2)} \otimes \ldots \otimes h_{(n)}$
is the image of $h$ under the comultiplication $\Delta$
applied $(n-1)$ times, $h\in H$. Then $F \lbrace X | H\rbrace$ is \textit{the absolutely free $H$-module nonassociative algebra} on $X$, i.e. for each map $\psi \colon X \to A$ where $A$ is an $H$-module algebra,
there exists a unique homomorphism $\bar\psi \colon 
F \lbrace X | H\rbrace \to A$ of algebras and $H$-modules, such that $\bar\psi\bigl|_X=\psi$.
Here we identify $X$ with the set $\lbrace x^1_j \mid j \in \mathbb N\rbrace \subset F \lbrace X | H\rbrace$.

Consider the $H$-invariant ideal $I$ in $F\lbrace X | H \rbrace$
generated by the set \begin{equation}\label{EqSetOfHGen}
\bigl\lbrace u(vw)+v(wu)+w(uv) \mid u,v,w \in  F\lbrace X | H \rbrace\bigr\rbrace \cup\bigl\lbrace u^2 \mid u \in  F\lbrace X | H \rbrace\bigr\rbrace.
\end{equation}
 Then $L(X | H) := F\lbrace X | H \rbrace/I$
is \textit{the free $H$-module Lie algebra}
on $X$, i.e. for any $H$-module Lie algebra $L$ 
and a map $\psi \colon X \to L$, there exists a unique homomorphism $\bar\psi \colon L(X | H) \to L$
of algebras and $H$-modules such that $\bar\psi\bigl|_X =\psi$. 
 We refer to the elements of $L(X | H)$ as \textit{Lie $H$-polynomials}.


\begin{remark} If $H$ is cocommutative and $\ch F \ne 2$, then $L(X | H)$ is the ordinary
free Lie algebra with free generators $x_i^{\gamma_\beta}$, $\beta \in \Lambda$, $i \in \mathbb N$
where   $(\gamma_\beta)_{\beta \in \Lambda}$ is a basis in $H$, since the ordinary ideal of 
$F\lbrace X | H \rbrace$ generated by~(\ref{EqSetOfHGen})
is already $H$-invariant.
However, if $h_{(1)} \otimes h_{(2)} \ne h_{(2)} \otimes h_{(1)}$ for some $h \in H$,
we still have $$[x^{h_{(1)}}_i, x^{h_{(2)}}_j]=h[x_i, x_j]=-h[x_j, x_i]=-[x^{h_{(1)}}_j, x^{h_{(2)}}_i]
= [x^{h_{(2)}}_i, x^{h_{(1)}}_j]$$ in $L(X | H)$ for all $i,j \in\mathbb N$,
i.e. in the case $h_{(1)} \otimes h_{(2)} \ne h_{(2)} \otimes h_{(1)}$ the algebra $L(X | H)$ is not free as an ordinary Lie algebra.
\end{remark}

Let $L$ be an $H$-module Lie algebra for
some Hopf algebra $H$ over a field $F$.
 An $H$-polynomial
 $f \in L ( X | H )$
 is a \textit{$H$-identity} of $L$ if $\psi(f)=0$
for all homomorphisms $\psi \colon L(X|H) \to L$
of algebras and $H$-modules. In other words, $f(x_1, x_2, \ldots, x_n)$
 is a polynomial $H$-identity of $L$
if and only if $f(a_1, a_2, \ldots, a_n)=0$ for any $a_i \in L$.
 In this case we write $f \equiv 0$.
The set $\Id^H(L)$ of all polynomial $H$-identities
of $L$ is an $H$-invariant ideal of $L(X|H)$.

\begin{example}\label{ExampleIdH}
Consider $e_0, e_1 \in \End_F(\mathfrak{gl}_2(F))$
defined by the formulas $$e_0 \left(
\begin{array}{cc}
a & b \\
c & d
\end{array}
 \right) := \left(
\begin{array}{rr}
a & 0 \\
0 & d
\end{array}
 \right)$$
 and $$e_1 \left(
\begin{array}{cc}
a & b \\
c & d
\end{array}
 \right) := \left(
\begin{array}{rr}
0 & b \\
c & 0
\end{array}
 \right).$$
 Then $H:=F e_0 \oplus F e_1$ (direct sum of ideals)
 is a Hopf algebra with the counit $\varepsilon$, where $\varepsilon(e_0):=1$,
 $\varepsilon(e_1):=0$, the comultiplication $\Delta$ where
 $$\Delta(e_0):=e_0 \otimes e_0 + e_1 \otimes e_1,$$
 $$\Delta(e_1):=e_0 \otimes e_1 + e_1 \otimes e_0,$$ and the antipode $S:=\id$.
 Moreover, $\mathfrak{gl}_2(F)$ is an $H$-module Lie algebra and
  $[x^{e_0},y^{e_0}]\in \Id^{H}(\mathfrak{gl}_2(F))$.
  As we shall see in Subsection~\ref{SubsectionApplGr}, this $H$-action and this $H$-identity corresponds to
  the $\mathbb Z_2$-grading and the graded identity from Example~\ref{ExampleIdGr}.
\end{example}

Denote by $V^H_n$ the space of all multilinear Lie $H$-polynomials
in $x_1, \ldots, x_n$, $n\in\mathbb N$, i.e.
$$V^{H}_n = \langle [x^{h_1}_{\sigma(1)},
x^{h_2}_{\sigma(2)}, \ldots, x^{h_n}_{\sigma(n)}]
\mid h_i \in H, \sigma\in S_n \rangle_F \subset L( X | H ).$$
Then the number $c^H_n(L):=\dim\left(\frac{V^H_n}{V^H_n \cap \Id^H(L)}\right)$
is called the $n$th \textit{codimension of polynomial $H$-identities}
or the $n$th \textit{$H$-codimension} of $L$.

\begin{remark}
One can treat polynomial $H$-identities of $L$ as identities of a nonassociative
algebra (i.e. use $F\lbrace X | H\rbrace$ instead of $L(X | H)$) and define their codimensions.
However these codimensions will coincide with $c^{H}_n(L)$ since the $n$th $H$-codimension
equals the dimension of the subspace in $\Hom_F(L^{{}\otimes n}; L)$ that consists of those $n$-linear functions that can be represented by $H$-polynomials.
(See also the proof of Lemma~\ref{LemmaCodimDim} below.)
\end{remark}

\subsection{Some bounds for $H$-codimensions}\label{SubsectionGenHopfBounds}

In this subsection we prove several inequalities for $H$-codimensions and ordinary
codimensions. They show that $H$-codimensions have an asymptotic behaviour
that is in some sense similar to the behaviour of the ordinary
codimensions. Hence we indeed may expect from $H$-codimensions
to satisfy the analog of Amitsur's conjecture (see Subsection~\ref{SubsectionAmitsur}).

As in the case of ordinary codimensions,
 we have the following upper bound:
 
 \begin{lemma}\label{LemmaCodimDim}
Let $L$ be an $H$-module Lie algebra for
some Hopf algebra $H$ over an arbitrary field $F$.
 Then
$c_n^H(L) \leqslant (\dim L)^{n+1}$ for all  $n \in \mathbb N$.
\end{lemma}
\begin{proof}
Consider $H$-polynomials as $n$-linear maps from $L$ to $L$.
Then we have a natural map $V^{H}_n \to \Hom_{F}(L^{{}\otimes n}; L)$
with the kernel $V^{H}_n \cap \Id^H(L)$
that leads to the embedding $$\frac{V^{H}_n}{V^{H}_n \cap \Id^H(L)}
\hookrightarrow \Hom_{F}(L^{{}\otimes n}; L).$$
Thus $$c^H_n(L)=\dim \left(\frac{V^{H}_n}{V^{H}_n \cap \Id^H(L)}\right)
\leqslant \dim \Hom_{F}(L^{{}\otimes n}; L)=(\dim L)^{n+1}.$$
\end{proof}

Denote by $V_n$ the space of ordinary multilinear Lie polynomials
in the noncommuting variables $x_1, \ldots, x_n$ and by $\Id(L)$
the set of ordinary polynomial identities of $L$.
In other words, $V_n = V^{\mathrm{gr}}_n$ and $\Id(L)=\Id^{\mathrm{gr}}(L)$
for the trivial grading on $L$ by the trivial group $\lbrace e \rbrace$. Then $c_n(L)
:=\dim\frac{V_n}{V_n \cap \Id(L)}$.

 The next lemma is an analog of \cite[Lemmas 10.1.2 and 10.1.3]{ZaiGia}.

\begin{lemma}\label{LemmaOrdinaryAndHopf}
Let $L$ be an $H$-module Lie algebra for
some Hopf algebra $H$ over a field $F$ and let $\zeta \colon H \to \End_F(L)$ be the homomorphism
corresponding to the $H$-action.
 Then
$$c_n(L) \leqslant c^{H}_n(L)
  \leqslant (\dim \zeta(H))^n c_n(L) \text{ for all } n \in \mathbb N.$$
\end{lemma}
\begin{proof}
As in Lemma~\ref{LemmaCodimDim}, we consider polynomials as $n$-linear maps from $L$ to $L$ and identify  $\frac{V_n}{V_n \cap \Id(L)}$ and $\frac{V^{H}_n}{V^{H}_n \cap \Id^H(L)}$
with the corresponding subspaces in $\Hom_{F}(L^{{}\otimes n}; L)$.
Then
 $$\frac{V_n}{V_n \cap \Id(L)} \subseteq \frac{V^{H}_n}{V^{H}_n \cap \Id^H(L)}
\subseteq \Hom_{F}(L^{{}\otimes n}; L)$$
and the lower bound follows.

Choose such $f_1, \ldots, f_t \in V_n$ that their images form a basis in 
$\frac{V_n}{V_n \cap \Id(L)}$.  Then for any monomial $[x_{\sigma(1)}, x_{\sigma(2)},
\ldots, x_{\sigma(n)}]$, $\sigma \in S_n$, there exist $\alpha_{i,\sigma} \in F$ such that
\begin{equation}\label{EqPnModuloId}[x_{\sigma(1)}, x_{\sigma(2)},
\ldots, x_{\sigma(n)}] - \sum_{i=1}^t \alpha_{i,\sigma} f_i(x_1, \ldots, x_n) \in \Id(L).
\end{equation}

Let $\bigl(\zeta(\gamma_j)\bigr)_{j=1}^m$, $\gamma_j \in H$, be a basis in $\zeta(H)$. Then
for every $h \in H$ there exist such $\alpha_j \in F$
that $\zeta(h) = \sum_{j=1}^m \alpha_j \zeta(\gamma_j)$ and 
\begin{equation}\label{Eqhtogammaj}
x^h - \sum_{j=1}^m \alpha_j x^{\gamma_j} \in \Id^H(L).
\end{equation}
 Thus the linear span of $[x^{\gamma_{i_1}}_{\sigma(1)}, x^{\gamma_{i_2}}_{\sigma(2)},
\ldots, x^{\gamma_{i_n}}_{\sigma(n)}]$, $\sigma \in S_n$, $1 \leqslant i_j \leqslant m$,
coincides with $V^H_n$ modulo $V^H_n \cap \Id^H(L)$. Note that (\ref{EqPnModuloId}) implies
$$[x^{\gamma_{i_1}}_{\sigma(1)},x^{\gamma_{i_2}}_{\sigma(2)},
\ldots, x^{\gamma_{i_n}}_{\sigma(n)}] - \sum_{i=1}^t \alpha_{i,\sigma} f_i(x^{\gamma_{i_1}}_1,
 \ldots, x^{\gamma_{i_n}}_n) \in \Id^H(L).$$
Hence any $H$-polynomial from $V^H_n$ can be expressed modulo $\Id^H(L)$
as a linear combination of $H$-polynomials $f_i(x^{\gamma_{i_1}}_1,
 \ldots, x^{\gamma_{i_n}}_n)$. The number of such polynomials equals $m^n t = \bigl(\dim \zeta(H)\bigr)^n c_n(L)$
 that finishes the proof.
\end{proof}

\subsection{The analog of Amitsur's conjecture for $H$-codimensions}\label{SubsectionAmitsur}

Let $L$ be an $H$-module Lie algebra.
The analog of Amitsur's conjecture for $H$-codimensions of $L$ can be formulated
as follows.

\begin{conjecture}  There exists
 $\PIexp^H(L):=\lim\limits_{n\to\infty}
 \sqrt[n]{c^H_n(L)} \in \mathbb Z_+$.
\end{conjecture}

We call $\PIexp^H(L)$ the \textit{Hopf PI-exponent} of $L$.

\begin{theorem}\label{TheoremMainHSS}
Let $L$ be a finite dimensional non-nilpotent $H$-module Lie algebra
over a field $F$ of characteristic $0$ where $H$ is a finite dimensional
semisimple Hopf algebra.
 Then there exist constants $C_1, C_2 > 0$, $r_1, r_2 \in \mathbb R$, $d \in \mathbb N$ such that $C_1 n^{r_1} d^n \leqslant c^{H}_n(L) \leqslant C_2 n^{r_2} d^n$ for all $n \in \mathbb N$.
\end{theorem}

\begin{remark}
If $L$ is nilpotent, i.e. $[x_1, \ldots, x_p]\equiv 0$ for some $p\in\mathbb N$, then
$V^{H}_n \subseteq \Id^{H}(L)$ and $c^H_n(L)=0$ for all $n \geqslant p$.
\end{remark}

\begin{corollary}
The above analog of Amitsur's conjecture holds
 for such codimensions.
\end{corollary}

\begin{theorem}\label{TheoremMainHSSSum}
Let $L=L_1 \oplus \ldots \oplus L_s$ (direct sum of $H$-invariant ideals) be a finite dimensional $H$-module Lie algebra
over a field $F$ of characteristic $0$ where $H$ is a finite dimensional
semisimple Hopf algebra. Then $\PIexp^H(L)=\max_{1 \leqslant i \leqslant s}
\PIexp^H(L_i)$.
\end{theorem}

Theorems~\ref{TheoremMainHSS} and~\ref{TheoremMainHSSSum} are obtained as consequences of more general results in Subsection~\ref{SubsectionHnice}.

\subsection{$(H,L)$-modules}\label{SubsectionHLmodules}
In this subsection we give the definitions needed to formulate the more general result and to
present the explicit formula for the Hopf PI-exponent.

Let $L$ be an $H$-module Lie algebra for
some Hopf algebra $H$ over a field $F$.

We say that $M$ is \textit{$(H,L)$-module} if $M$ is both left $H$- and $L$-module
and \begin{equation}\label{EqHLModule}
h(a v)=(h_{(1)} a) (h_{(2)} v) \text{ for all } a \in L,\ v\in M,\ h\in H.
\end{equation}

By~(\ref{EqHmoduleLieAlgebra}), each $H$-invariant ideal of $L$ can be regarded as
a left $(H,L)$-module under the adjoint representation
of $L$.

 An $(H,L)$-module $M$ is \textit{irreducible}
if for any $H$- and $L$-invariant
subspace $M_1 \subseteq M$ we have either $M_1=0$ or $M_1=M$.
We say that $M$ is \textit{completely reducible} if $M$ is the direct sum of
irreducible $(H,L)$-submodules.  

If $M$ is an $H$-module, then $\End_F(M)$ has the natural structure of an associative $H$-module algebra:
\begin{equation}\label{EqHActionOnEnd}
h\psi:=\zeta(h_{(1)})\psi\zeta(Sh_{(2)}) \text{ where } h \in H \text{ and }\psi \in \End_F(M).
\end{equation}
Here $\zeta \colon H \to \End_F(M)$ is the homomorphism
corresponding to the $H$-action.
The Lie algebra $\mathfrak{gl}(M)$ inherits the structure of an $H$-module from $\End_F(M)$.
However, if $H$ is not cocommutative, we cannot claim that $\mathfrak{gl}(M)$ satisfies~(\ref{EqHmoduleLieAlgebra}).

If $M$ is an $(H,L)$-module where $L$ is an $H$-module Lie algebra and $\varphi \colon 
L \to \mathfrak{gl}(M)$ is the corresponding homomorphism, then $\varphi$ is a homomorphism
of $H$-modules.  Moreover,~(\ref{EqHLModule}) is equivalent to
\begin{equation}\label{EqHLModule2}\zeta(h)\varphi(a) = \varphi(h_{(1)}a) \zeta(h_{(2)})
\text{ where }
h\in H,\ a\in L.
\end{equation}
We say that an $(H,L)$-module $M$ is \textit{faithful}
if it is faithful as an $L$-module, i.e. if $\ker\varphi = 0$.

\subsection{$H$-nice Lie algebras}\label{SubsectionHnice}

Let $L$ be a finite dimensional $H$-module Lie algebra
where $H$ is a Hopf algebra over an algebraically closed field $F$
of characteristic $0$.
We say that $L$ is \textit{$H$-nice} if either $L$ is semisimple or the following conditions hold:
\begin{enumerate}
\item \label{ConditionRNinv}
the nilpotent radical $N$ and the solvable radical $R$ of $L$ are $H$-invariant;
\item \label{ConditionLevi} \textit{(Levi decomposition)}
there exists an $H$-invariant maximal semisimple subalgebra $B \subseteq L$ such that
$L=B\oplus R$ (direct sum of $H$-modules);
\item \label{ConditionWedderburn} \textit{(Wedderburn~--- Mal'cev decompositions)}
for any $H$-submodule $W \subseteq L$ and associative $H$-module subalgebra $A_1 \subseteq \End_F(W)$, 
the Jacobson radical $J(A_1)$ is $H$-invariant and
there exists an $H$-invariant maximal semisimple associative subalgebra $\tilde A_1 \subseteq A_1$
such that $A_1 = \tilde A_1 \oplus J(A_1)$ (direct sum of $H$-submodules);
\item \label{ConditionLComplHred}
for any $H$-invariant Lie subalgebra $L_0 \subseteq \mathfrak{gl}(L)$
such that $L_0$ is an $H$-module algebra and $L$ is a completely reducible $L_0$-module disregarding $H$-action, $L$ is a completely reducible $(H,L_0)$-module.
\end{enumerate}

\begin{example}\label{ExampleHniceHSS}
Suppose $L$ is a finite dimensional $H$-module Lie algebra where $H$ is a finite dimensional
semisimple Hopf algebra over $F$.
Then $L$ is $H$-nice.
\end{example}
\begin{proof}
  By~\cite[Theorem~1]{ASGordienko4}, the solvable radical $R$
  and the nilpotent radical $N$ of $L$ are $H$-invariant.
  Moreover, by~\cite[Theorem~3]{ASGordienko4},
 $L=B\oplus R$ (direct sum of $H$-submodules) where 
$B$ is some maximal semisimple subalgebra of $L$.
By~\cite{LinchenkoJH} and \cite[Corollary~2.7]{SteVanOyst}, we have an $H$-invariant Wedderburn~--- Malcev decomposition.
By~\cite[Theorem~10]{ASGordienko4}, Condition~\ref{ConditionLComplHred} holds.
 Hence $L$ is $H$-nice.
\end{proof}

Let $G$ be any group.
Denote by $FG$ the group algebra of $G$. Then $FG$ is a Hopf algebra
with the comultiplication $\Delta(g)=g\otimes g$, the counit $\varepsilon(g)=1$,
and  the antipode $S(g)=g^{-1}$, $g \in G$.

\begin{example}\label{ExampleHniceAffAlg}
Let $L$ be a finite dimensional Lie algebra
over $F$. Suppose a reductive affine algebraic group
  $G$ acts on $L$ rationally by automorphisms.
  Then $L$ is an $FG$-nice algebra where the action of the Hopf algebra $FG$ on $L$
is the extension of the $G$-action by linearity.
\end{example}
\begin{proof}
 First, the solvable radical $R$ and the nilpotent radical $N$ are $G$-invariant
 since they are invariant under all automorphisms. Hence Condition~\ref{ConditionRNinv} holds. By~\cite[Theorem~5]{ASGordienko4}, we have a $G$-invariant
 Levi decomposition.
 
 Suppose $W \subseteq L$ is an $FG$-submodule. 
  Consider the $FG$-action~(\ref{EqHActionOnEnd}) on $\End_F(W)$. It corresponds to the
  natural rational $G$-action on $\End_F(W)$: $\psi^g w = (\psi w^{g^{-1}})^g$
  for $w \in W$, $\psi \in \End_F(W)$, $g \in G$.
  Hence, by \cite[Corollary 2.10]{SteVanOyst}, for any $FG$-module subalgebra $A_1 \subseteq \End_F(W)$
  we have an $FG$-invariant Wedderburn~--- Mal'cev decomposition.
  
  Analogously, for any $FG$-module Lie subalgebra $L_0 \subseteq \mathfrak{gl}(L)$
  the space $L$ is a \textit{$(G, L_0)$-module},
  i.e. $(\psi a)^g = \psi^g a^g$  for all $\psi \in L_0$, $a \in L$, $g \in G$.
  By~\cite[Theorem~11]{ASGordienko4}, if $L$ is a completely reducible $L_0$-module
  disregarding $G$-action, it is a completely reducible $(FG,L_0)$-module.
  
  Therefore, $L$ is an $FG$-nice algebra.
\end{proof}

Another important example of an $H$-nice algebra is Example~\ref{ExampleHniceGr}
in Subsection~\ref{SubsectionApplGr} below. (In fact, one can regard Example~\ref{ExampleHniceGr}
as a special case of Example~\ref{ExampleHniceAffAlg}.)

Theorem~\ref{TheoremMainH} is the main result of the article.

\begin{theorem}\label{TheoremMainH}
Let $L$ be a non-nilpotent $H$-nice Lie algebra over an algebraically closed field $F$ of characteristic $0$.  Then there exist constants $C_1, C_2 > 0$, $r_1, r_2 \in \mathbb R$, $d \in \mathbb N$ such that $C_1 n^{r_1} d^n \leqslant c^{H}_n(L) \leqslant C_2 n^{r_2} d^n$ for all $n \in \mathbb N$.
\end{theorem}

\begin{theorem}\label{TheoremMainHSum}
Let $L=L_1 \oplus \ldots \oplus L_s$ (direct sum of $H$-invariant ideals) be an $H$-module Lie algebra
over an algebraically closed field $F$ of characteristic $0$ where $H$ is a Hopf algebra. Suppose $L_i$ are $H$-nice algebras. Then there exists $\PIexp^H(L)=\max_{1 \leqslant i \leqslant s}
\PIexp^H(L_i)$.
\end{theorem}

Theorems~\ref{TheoremMainH} and~\ref{TheoremMainHSum} are proved in the end of Section~\ref{SectionLower}.
A formula for $d=\PIexp^H(L)$ is given in Subsection~\ref{SubsectionHopfPIexp}.

\begin{proof}[Proof of Theorems~\ref{TheoremMainHSS} and~\ref{TheoremMainHSSSum}]
Suppose $L$ is a finite dimensional $H$-module Lie algebra where $H$ is a finite dimensional
semisimple Hopf algebra over a field of characteristic $0$.

$H$-codimensions do not change upon an extension of the base field.
The proof is analogous to the cases of ordinary codimensions of
associative~\cite[Theorem~4.1.9]{ZaiGia} and
Lie algebras~\cite[Section~2]{ZaiLie}.
Thus without loss of generality we may assume
 $F$ to be algebraically closed.
 
 Using Example~\ref{ExampleHniceHSS}, we derive Theorem~\ref{TheoremMainHSS}
  from Theorem~\ref{TheoremMainH}. Analogously, Theorem~\ref{TheoremMainHSSSum} is obtained from Theorem~\ref{TheoremMainHSum}.
 \end{proof}

\subsection{Formula for the Hopf PI-exponent}\label{SubsectionHopfPIexp}

Suppose $L$ is an $H$-nice Lie algebra.

Consider $H$-invariant ideals $I_1, I_2, \ldots, I_r$,
$J_1, J_2, \ldots, J_r$, $r \in \mathbb Z_+$, of the algebra $L$ such that $J_k \subseteq I_k$,
satisfying the conditions
\begin{enumerate}
\item $I_k/J_k$ is an irreducible $(H,L)$-module;
\item for any $H$-invariant $B$-submodules $T_k$
such that $I_k = J_k\oplus T_k$, there exist numbers
$q_i \geqslant 0$ such that $$\bigl[[T_1, \underbrace{L, \ldots, L}_{q_1}], [T_2, \underbrace{L, \ldots, L}_{q_2}], \ldots, [T_r,
 \underbrace{L, \ldots, L}_{q_r}]\bigr] \ne 0.$$
\end{enumerate}

Let $M$ be an $L$-module. Denote by $\Ann M$ its annihilator in $L$.
Let $$d(L) := \max \left(\dim \frac{L}{\Ann(I_1/J_1) \cap \dots \cap \Ann(I_r/J_r)}
\right)$$
where the maximum is found among all $r \in \mathbb Z_+$ and all $I_1, \ldots, I_r$, $J_1, \ldots, J_r$
satisfying Conditions 1--2. We claim
that $\PIexp^H(L)=d(L)$ and prove
Theorem~\ref{TheoremMainH} for $d=d(L)$.

The following example will be used in the proof of Theorems~\ref{TheoremMainH} and~\ref{TheoremMainHSum}
in the case of semisimple $L$ (when we do not require Conditions~\ref{ConditionWedderburn}
and~\ref{ConditionLComplHred} from Subsection~\ref{SubsectionHnice}).

\begin{example}\label{ExamplePIexpBSS}
If $L$ is a semisimple $H$-module Lie algebra, then
by~\cite[Theorem 6]{ASGordienko4},  
  $$L=B=B_1 \oplus B_2 \oplus \ldots \oplus B_q$$ (direct sum of $H$-invariant ideals)
   where $B_i$ are $H$-simple Lie
  algebras. In this case $$d(L) = \max_{1 \leqslant k
  \leqslant q} \dim B_k.$$ 
 \end{example}
 \begin{proof}
 Note that if $I$ is a $H$-simple ideal of $L$, then $[I,L] \ne 0$ and
hence $[I,B_i] \ne 0$
for some $1 \leqslant i \leqslant q$.
However $[I,B_i] \subseteq B_i \cap I$ is a $H$-invariant ideal.
Thus $I=B_i$. And if $I$ is a $H$-invariant ideal of $L$,
then it is semisimple and each of its simple components coincides with
one of $B_i$. Thus if $I \subseteq J$ are $H$-invariant ideals of $L$
and $I/J$ is an irreducible $(H,L)$-module,
then $I = B_i \oplus J$ for some $1 \leqslant i \leqslant q$
and $\dim(L/\Ann(I/J))=\dim B_i$. 

Suppose $I_1, \ldots, I_r$, $J_1, \ldots, J_r$ satisfy Conditions 1--2.
Let $I_k = B_{i_k}\oplus J_k$,
 $1 \leqslant k \leqslant r$.
 Then
  $$[[B_{i_1}, L, \ldots, L], [B_{i_2}, L, \ldots, L], \ldots, [B_{i_r}, L, \ldots, L]]\ne 0$$
  for some number of copies of $L$.
 Hence $i_1 = \ldots = i_r$ and $$\dim \frac{L}{\Ann(I_1/J_1) \cap \dots \cap \Ann(I_r/J_r)} = \dim B_{i_1}.$$ Therefore $d(L)=\max_{1 \leqslant k \leqslant q} \dim B_k$.
 \end{proof}

\subsection{$S_n$-cocharacters}\label{SubsectionSnCocharacters}

One of the main tools in the investigation of polynomial
identities is provided by the representation theory of symmetric groups.
 The symmetric group $S_n$  acts
 on the space $\frac {V^H_n}{V^H_{n}
  \cap \Id^H(L)}$
  by permuting the variables.
  Irreducible $FS_n$-modules are described by partitions
  $\lambda=(\lambda_1, \ldots, \lambda_s)\vdash n$ and their
  Young diagrams $D_\lambda$.
   The character $\chi^H_n(L)$ of the
  $FS_n$-module $\frac {V^H_n}{V^H_n
   \cap \Id^H(L)}$ is
   called the $n$th
  \textit{cocharacter} of polynomial $H$-identities of $L$.
  We can rewrite it as
  a sum $$\chi^H_n(L)=\sum_{\lambda \vdash n}
   m(L, H, \lambda)\chi(\lambda)$$ of
  irreducible characters $\chi(\lambda)$.
Let  $e_{T_{\lambda}}=a_{T_{\lambda}} b_{T_{\lambda}}$
and
$e^{*}_{T_{\lambda}}=b_{T_{\lambda}} a_{T_{\lambda}}$
where
$a_{T_{\lambda}} = \sum_{\pi \in R_{T_\lambda}} \pi$
and
$b_{T_{\lambda}} = \sum_{\sigma \in C_{T_\lambda}}
 (\sign \sigma) \sigma$,
be Young symmetrizers corresponding to a Young tableau~$T_\lambda$.
Then $M(\lambda) = FS e_{T_\lambda} \cong FS e^{*}_{T_\lambda}$
is an irreducible $FS_n$-module corresponding to
 a partition~$\lambda \vdash n$.
  We refer the reader to~\cite{Bahturin, DrenKurs, ZaiGia}
   for an account
  of $S_n$-representations and their applications to polynomial
  identities.

\medskip

In Section~\ref{SectionSomePropofHL} and~\ref{SectionAux} we discuss $(H,L)$-modules and their
annihilators.

In Section~\ref{SectionUpper}
we prove that if $m(L, H, \lambda) \ne 0$, then
the corresponding Young diagram $D_\lambda$
has at most $d$ long rows. This implies the upper bound.

In Section~\ref{SectionAlt} we
consider faithful irreducible $(H,L_0)$-modules
where $L_0$ is a reductive $H$-module Lie algebra.
For an arbitrary $k\in\mathbb N$,
we construct an associative $H$-polynomial that is
alternating in $2k$ sets,
 each consisting of $\dim L_0$ variables. This polynomial
 is not an identity of the corresponding representation
 of $L_0$. In Section~\ref{SectionLower} we choose reductive
 algebras and faithful irreducible modules,
 and glue the corresponding alternating polynomials.
 This enables us to find $\lambda \vdash n$
 with $m(L, H, \lambda)\ne 0$ such that $\dim M(\lambda)$
 has the desired asymptotic behavior and the lower bound is proved.

\section{Some properties of $(H,L)$-modules}\label{SectionSomePropofHL}

\begin{lemma}\label{LemmaAnnHLmodule}
Let $M$ be an $(H,L)$-module for some Hopf algebra $H$ and $H$-module Lie algebra $L$
over a field $F$. Then $\Ann M$ is an $H$-invariant ideal of $L$.
\end{lemma}
\begin{proof}
Note that $[a,b]v=a(bv)-b(av)=0$ for all $v \in M$, $a \in \Ann M$, $b \in L$.
Hence $\Ann M$ is an ideal. 
 Moreover $$(ha)v=(h_{(1)}\varepsilon(h_{(2)})a)v
= (h_{(1)}a)((\varepsilon(h_{(2)})1)v)=(h_{(1)}a)(h_{(2)} (Sh_{(3)})v)=
$$ $$(h_{(1)(1)}a)(h_{(1)(2)} (Sh_{(2)})v)
=h_{(1)}(a(Sh_{(2)})v)=0$$  for all $v \in M$, $a \in \Ann M$,  $h\in H$.
Thus $\Ann M$ is an $H$-submodule.
\end{proof}

In Lemmas~\ref{LemmaLR}, \ref{LemmaGeneralizedJordan}, \ref{LemmaRedIrr}, $M$ is a finite dimensional
$(H,L)$-module where $H$ is a Hopf algebra
over an algebraically closed field $F$ of characteristic $0$ and $L$ is an $H$-module
Lie algebra with the solvable radical $R$ which we require to be $H$-invariant.
Recall that we denote by $\zeta \colon H \to \End_F(M)$ and
$\varphi \colon L \to \mathfrak{gl}(M)$ the homomorphisms
corresponding to the $(H,L)$-module structure. Denote by $A$ the associative subalgebra of $\End_F(M)$ generated by the operators from $\varphi(L)$
and $\zeta(H)$.

\begin{lemma}\label{LemmaLR}
$\varphi([L, R]) \subseteq J(A)$ where $J(A)$ is the Jacobson radical of $A$.
\end{lemma}
\begin{proof}
Let $M = W_0 \supseteq W_1 \supseteq W_2 \supseteq \ldots \supseteq W_t = \left\lbrace 0 \right\rbrace$
be a composition series in $M$ of $L$-submodules not necessarily $H$-invariant.
Then each $W_i/W_{i+1}$ is an irreducible $L$-module.
Denote the corresponding homomorphism by $\varphi_i \colon L \to
\mathfrak{gl}(W_i/W_{i+1})$. Then by E.~Cartan's theorem~\cite[Proposition~1.4.11]{GotoGrosshans},
$\varphi_i (L)$ is semisimple or the direct sum of a semisimple ideal and
the center of $\mathfrak{gl}(W_i/W_{i+1})$.
Thus $\varphi_i ([L, L])$ is semisimple
and $\varphi_i ([ L, L ] \cap R) = 0$.
Since $[L, R] \subseteq [L, L] \cap R$,
we have $\varphi_i([L, R])=0$ and $[L,R]W_i \subseteq W_{i+1}$.

Denote by $I$ the associative ideal of $A$
generated by $\varphi([L, R])$.
Then $I^t$ is the associative ideal generated by elements of the form
$$
a_1 \bigl(\zeta(h_{10})b_{11}\zeta(h_{11})b_{12} \ldots \zeta(h_{1,s_1-1})b_{1,s_1}
\zeta(h_{1,s_1})\bigr) a_2 \bigl(\zeta(h_{20})b_{21}\zeta(h_{21})b_{22} \ldots \zeta(h_{2,s_2-1})b_{2,s_2}
\zeta(h_{2,s_2})\bigr) 
$$ $$\ldots
a_{t-1}\bigl(\zeta(h_{t-1,0})b_{t-1,1}\zeta(h_{t-1,1})b_{t-1,2} \ldots
 \zeta(h_{t-1,s_{t-1}-1})b_{t-1,s_{t-1}}\zeta(h_{t-1,s_{t-1}})\bigr) a_t$$
 where $a_i \in \varphi([L, R])$, $b_{ij}\in\varphi(L)$, $h_{ij}\in H$.
Using~(\ref{EqHLModule2}), we move all $\zeta(h_{ij})$
to the right and obtain that $I^t$ is generated by elements
  $$b =
a_1 \bigl((h'_{11}b_{11}) \ldots (h'_{1,s_1} b_{1,s_1})
\bigr)\ (h_2a_2)\left((h'_{21}b_{21}) \ldots (h'_{2,s_2} b_{2,s_2})\right) \ldots $$ $$
(h_{t-1}a_{t-1})
\bigl((h'_{t-1,1}b_{t-1,1}) \ldots
 (h'_{t-1,s_{t-1}}b_{t-1,s_{t-1}})\bigr) (h_t a_t) \zeta(h_{t+1})
$$ where $h_i, h'_{ij} \in H$.
However, each $h_i a_i \in \varphi([L, R])$ since $R$ is $H$-invariant.
Hence $(h_i a_i) W_{k-1} \subseteq W_k$ for all $1 \leqslant k \leqslant t$.
Thus $b=0$, $I^t=0$, and $\varphi([L, R]) \subseteq J(A)$.
\end{proof}

\begin{lemma}\label{LemmaGeneralizedJordan} 
Suppose for every $H$-invariant associative subalgebra $A_1 \subseteq \End_F(M)$ there exists an $H$-invariant Wedderburn~--- Mal'cev
decomposition.
Let $W=\langle a_1, \ldots, a_t \rangle_F$ be a subspace in $R$
such that $\varphi(W)$ is an $H$-submodule in $\varphi(R)$.
 Then $\varphi(a_i) = c_i+d_i$, 
$1 \leqslant i \leqslant t$, where $c_i$ and $d_i$ are polynomials in $\varphi(a_j)$, $1 \leqslant j
\leqslant t$,
without a constant term, $c_i$ are commuting diagonalizable operators on $M$,
and $d_i \in J(A)$. Moreover, $\langle c_1, \ldots, c_t \rangle_F$
and $\langle d_1, \ldots, d_t \rangle_F$ are $H$-submodules in $\End_F(M)$.
\end{lemma}
\begin{proof}
Consider the $H$-invariant Lie algebra $\varphi(R)+J(A)\subseteq \mathfrak{gl}(M)$.
Note that $\varphi(R)+J(A)$ is solvable since $J(A)$ is nilpotent and $(\varphi(R)+J(A))/J(A)
\cong \varphi(R)/(\varphi(R) \cap J(A))$ is a homomorphic image of the solvable algebra $R$.
 By Lie's Theorem, there exists
 a basis in $M$ in which all the operators from $\varphi(R)+J(A)$ have
 upper triangular matrices. Denote the corresponding
 embedding $A  \hookrightarrow M_s(F)$ by $\psi$.
 Here $s := \dim M$.

Let $A_1$ be the associative subalgebra of $\End_F(M)$ generated by $\varphi(a_i)$,
$1 \leqslant i \leqslant t$.
This algebra is $H$-invariant since $\varphi(W)$ is an $H$-submodule in $\varphi(R)$.
Hence $A_1$ is an associative $H$-module algebra.
By our assumptions, the Jacobson radical $J(A_1)$ is $H$-invariant
 and we have an $H$-invariant Wedderburn~--- Malcev decomposition
   $A_1 = \tilde A_1 \oplus J(A_1)$
 (direct sum of $H$-submodules)
where $\tilde A_1$ is a $H$-invariant semisimple subalgebra
of $A_1$.
Since $\psi(\varphi(R)) \subseteq \mathfrak{t}_s(F)$, we have
$\psi(A_1) \subseteq UT_s(F)$. Here $UT_s(F)$ is the associative algebra
 of upper triangular $s\times s$ matrices.
 There is a decomposition $$UT_s(F) = Fe_{11}\oplus Fe_{22}\oplus
 \dots\oplus Fe_{ss}\oplus \tilde N$$
 where $$\tilde N := \langle e_{ij} \mid 1 \leqslant i < j \leqslant s \rangle_F$$
 is a nilpotent ideal. Thus there is no subalgebras in $A_1$
 isomorphic to $M_2(F)$ and
  $\tilde A_1=Fe_1 \oplus \dots \oplus Fe_q$
    for some idempotents $e_i \in A_1$.
For every $a_j$, denote its component in $J(A_1)$ by $d_j$
and its component in $Fe_1 \oplus \dots \oplus Fe_q$ by $c_j$.
Note that $e_i$ are commuting diagonalizable operators. Thus they
have a common basis of eigenvectors in $M$ and $c_i$
are commuting diagonalizable operators too.
Moreover $$hc_j+hd_j=h\varphi(a_j) \in
 \langle \varphi(a_i)
\mid 1 \leqslant i \leqslant t \rangle_F
\subseteq \langle c_i
\mid 1 \leqslant i \leqslant t \rangle_F
\oplus \langle d_i
\mid 1 \leqslant i \leqslant t \rangle_F
\subseteq \tilde A_1 \oplus J(A_1)$$
for all $h \in H$. However, $\tilde A_1$ and $J(A_1)$ are $H$-invariant
and $hc_j \in \tilde A_1$, $hd_j \in J(A_1)$.
Thus $hc_j\in\langle c_1, \ldots, c_t \rangle_F$
and $hd_j\in\langle d_1, \ldots, d_t \rangle_F$. Therefore, $\langle c_1, \ldots, c_t \rangle_F$
and $\langle d_1, \ldots, d_t \rangle_F$ are $H$-submodules in $\End_F(M)$.

We claim that the space $J(A_1)+J(A)$ generates a nilpotent
$H$-invariant ideal $I$ in $A$.
First, $\psi(J(A_1))$ and $\psi(J(A))$ are contained in $UT_s(F)$
and consist of nilpotent elements. Thus the corresponding
matrices have zero diagonal elements and
$\psi(J(A_1)), \psi(J(A)) \subseteq \tilde N$.
Denote $$\tilde N_k := \langle e_{ij} \mid i+k \leqslant j \rangle_F \subseteq \tilde N.$$
Then $$\tilde N = \tilde N_1 \supsetneqq \tilde N_2 \supsetneqq \ldots \supsetneqq \tilde N_{m-1} \supsetneqq \tilde N_s = \lbrace 0\rbrace.$$
Let $\height_{\tilde N} a := k$ if $\psi(a) \in \tilde N_k$, $\psi(a) \notin \tilde N_{k+1}$.

Since $J(A)$ is nilpotent, $(J(A))^p =0$ for some $p \in\mathbb N$.
We claim that $I^{s+p} = 0$.
Using~(\ref{EqHLModule2}), we move all $\zeta(h)$, $h \in H$, to the right,
 and obtain that
 the space $I^{s+p}$
 is the span of $b_1 j_1 b_2 j_2 \ldots j_{s+p} b_{s+p+1} \zeta(h)$
 where $j_k \in J(A_1) \cup J(A)$, $b_k \in A_2 \cup \lbrace 1\rbrace$,
 $h \in H$. Here $A_2$ is the subalgebra of $\End_F(M)$
 generated by the operators from $\varphi(L)$ only.
 If at least $p$ elements $j_k$ belong to $J(A)$,
 then the product equals $0$.
Thus we may assume that at least $s$
elements $j_k$ belong to $J(A_1)$.

 Let $j_i \in J(A_1)$, $b_i \in A_2 \cup \lbrace 1\rbrace$.
We prove by induction on $\ell$ that
$j_1 b_1 j_2 b_2 \ldots b_{\ell-1} j_{\ell}$
can be expressed as a sum of $\tilde j_1 \tilde j_2 \ldots \tilde j_\alpha j'_1 j'_2\ldots j'_\beta
a$
where $\tilde j_i \in J(A_1)$, $j'_i \in J(A)$,
$a \in A_2 \cup \lbrace 1\rbrace$,
 and $\alpha+\sum_{i=1}^\beta \height_{\tilde N} j'_i \geqslant \ell$.
 Indeed, suppose that
  $j_1 b_1 j_2 b_2 \ldots b_{\ell-2} j_{\ell-1}$
 can be expressed as a sum of $\tilde j_1 \tilde j_2 \ldots \tilde j_\gamma j'_1 j'_2\ldots j'_\varkappa
a$
where $\tilde j_i \in J(A_1)$, $j'_i \in J(A)$,
$a \in A_2 \cup \lbrace 1\rbrace$,
 and $\gamma+\sum_{i=1}^\varkappa \height_{\tilde N} j'_i \geqslant \ell-1$.
 Then
 $j_1 b_1 j_2 b_2 \ldots j_{\ell-1} b_{\ell-1}j_{\ell}$
 is a sum of
 $$\tilde j_1 \tilde j_2 \ldots \tilde j_\gamma j'_1 j'_2\ldots j'_\varkappa
a b_{\ell-1}j_{\ell} =
\tilde j_1 \tilde j_2 \ldots \tilde j_\gamma j'_1 j'_2\ldots j'_\varkappa
[ab_{\ell-1}, j_{\ell}] + \tilde j_1 \tilde j_2 \ldots \tilde j_\gamma j'_1 j'_2\ldots j'_\varkappa
j_{\ell} (a b_{\ell-1}).$$
Note that, in virtue of the Jacobi identity
and Lemma~\ref{LemmaLR},  $[ab_{\ell-1}, j_{\ell}] \in
J(A)$. Thus it is sufficient to consider only the second term.
However
$$\tilde j_1 \tilde j_2 \ldots \tilde j_\gamma j'_1 j'_2\ldots j'_\varkappa
j_{\ell} (a b_{\ell-1})
= \tilde j_1 \tilde j_2 \ldots \tilde j_\gamma j_{\ell} j'_1 j'_2\ldots j'_\varkappa
 (a b_{\ell-1})+$$ $$\sum_{i=1}^{\varkappa}
\tilde j_1 \tilde j_2 \ldots \tilde j_\gamma  j'_1 j'_2\ldots j'_{i-1}[j'_{i}, j_\ell]
j'_{i+1}\ldots j'_\varkappa (a b_{\ell-1}).$$
Since $[j'_{i}, j_\ell] \in J(A)$ and $\height_{\tilde N} [j'_{i}, j_\ell] \geqslant 1+ \height_{\tilde N} j'_i$,
all the terms have  the desired form.
 Therefore, $$j_1 b_1 j_2 b_2 \ldots j_{s-1} b_{s-1}j_{s}
 \in \psi^{-1}(\tilde N_s) = \lbrace 0 \rbrace,$$ $I^{s+p}=0$, and $$
 J(A) \subseteq J(A_1)+J(A) \subseteq I \subseteq J(A).$$ In particular,
$d_i \in J(A_1) \subseteq J(A)$.
\end{proof}

\begin{lemma} \label{LemmaRedIrr}
Let $M$ be a finite dimensional irreducible $(H,L)$-module.
Suppose for every $H$-invariant associative subalgebra $A_1 \subseteq \End_F(M)$ there exists an $H$-invariant Wedderburn~--- Mal'cev
decomposition.
Then \begin{enumerate}
\item \label{RedSum} $M=M_1 \oplus \ldots \oplus M_q$ for some
$L$-submodules $M_i$, $1 \leqslant i \leqslant q$;
\item \label{RedScalar} elements of $R$ act on each $M_i$ by scalar operators.
\end{enumerate}
\end{lemma}
\begin{proof}
Let $\varphi(r_1), \ldots, \varphi(r_t)$ be a basis in $\varphi(R)$. By Lemma~\ref{LemmaGeneralizedJordan},
 $\varphi(r_i)=r'_i+r''_i$ where  $r'_i$ are commuting diagonalizable operators on $M$ and
 $r''_i \in J(A)$.
 Note  that by the Density Theorem, $A = \End_F(M)$. Hence $J(A)=0$ and $\varphi(r_i)=r'_i$.
Therefore, $\varphi(r_i)$ have  a common basis of eigenvectors and we can choose
subspaces $M_i$, $1 \leqslant i \leqslant q$, $q\in \mathbb N$,
such that $$M=M_1 \oplus \ldots \oplus M_q,$$ and each $M_i$
is the intersection of eigenspaces of $\varphi(r_i)$.
Note that by Lemma~\ref{LemmaLR}, $$[\varphi(r_i),\varphi(a)]\in J(a)=0\text{ for all }a\in L.$$
Thus $M_i$ are $L$-submodules and the lemma is proved.
\end{proof}

\section{Polynomial $H$-identities of representations and alternating $H$-polynomials}
\label{SectionAlt}

In this section we prove auxiliary propositions needed to obtain the lower bound.

Let $H$ be a Hopf algebra over a field $F$.
Analogously to the absolutely free nonassociative $H$-module algebra $F \lbrace X | H \rbrace$,
we define the free associative $H$-module algebra $F \langle X | H \rangle$
on the set $X := \lbrace x_1, x_2, x_3, \ldots \rbrace$.
Here we follow~\cite{BahturinLinchenko}.
First, consider the free associative algebra $F\langle X \rangle
 = \bigoplus_{n=1}^\infty F \langle X \rangle^{(n)}$
  where $F \langle X \rangle^{(n)}$ is the linear span of all associative monomials
  in variables from $X$ of total degree $n$.
  Consider the algebra $$F \langle X | H \rangle
   :=  \bigoplus_{n=1}^\infty H^{{}\otimes n} \otimes F \langle X \rangle^{(n)}$$
   with the multiplication $(u_1 \otimes w_1)(u_2 \otimes w_2):=(u_1 \otimes u_2) \otimes w_1w_2$
   for all $u_1 \in  H^{{}\otimes j}$, $u_2 \in  H^{{}\otimes k}$,
   $w_1 \in F \langle X \rangle^{(j)}$, $w_2 \in F\langle X \rangle^{(k)}$.
We use the notation $$x^{h_1}_{i_1}
x^{h_2}_{i_2}\ldots x^{h_n}_{i_n} := (h_1 \otimes h_2 \otimes \ldots \otimes h_n) \otimes x_{i_1}
x_{i_2}\ldots x_{i_n}.$$  Here $h_1 \otimes h_2 \otimes \ldots \otimes h_n \in H^{{}\otimes n}$,
$x_{i_1} x_{i_2}\ldots x_{i_n} \in F\langle X \rangle^{(n)}$. 

Note that if $(\gamma_\alpha)_{\alpha \in \Lambda}$ is a basis in $H$, 
then $F \langle X | H \rangle$ is isomorphic to the free associative algebra over $F$ with free formal  generators $x_i^{\gamma_\alpha}$, $\alpha \in \Lambda$, $i \in \mathbb N$.
 
    Define on $F \langle X | H \rangle$ the structure of a left $H$-module
   by $$h\,(x^{h_1}_{i_1}
x^{h_2}_{i_2}\ldots x^{h_n}_{i_n})=x^{h_{(1)}h_1}_{i_1}
x^{h_{(2)}h_2}_{i_2}\ldots x^{h_{(n)}h_n}_{i_n}$$
where $h\in H$. Then $F \langle X | H \rangle$ is \textit{the free $H$-module associative algebra} on $X$, i.e. for each map $\psi \colon X \to A$ where $A$ is an
associative $H$-module algebra,
there exists a unique homomorphism $\bar\psi \colon 
F \langle X | H \rangle \to A$ of algebras and $H$-modules, such that $\bar\psi\bigl|_X=\psi$.
Here we identify $X$ with the set $\lbrace x^1_j \mid j \in \mathbb N\rbrace \subset F \langle X | H \rangle$.

 Let $L$ be an $H$-module Lie algebra, let $M$ be an $(H,L)$-module, and let $\varphi \colon  L \to \mathfrak{gl}(M)$ be the corresponding representation. A polynomial $f(x_1, \ldots, x_n)\in F\langle X | H \rangle$ is an \textit{$H$-identity} of $\varphi$ if $f(\varphi(a_1), \ldots, \varphi(a_n))=0$ for all $a_i \in L$.
 In other words, $f$ is an $H$-identity of  $\varphi$ if for every homomorphism $\psi \colon 
 F\langle X | H \rangle \to \End_F(M)$ of algebras and $H$-modules such that $\psi(X)\subseteq \varphi(L)$,
 we have $\psi(f)=0$.  
  The set $\Id^{H}(\varphi)$ of all $H$-identities of $\varphi$ is an $H$-invariant two-sided ideal in $F\langle X | H \rangle$.

Denote by $P^H_n$, $n\in \mathbb N$, the subspace of associative multilinear $H$-polynomials
in variables $x_1, \ldots, x_n$.
In other words, $$P^H_n = \left\langle  x_{\sigma(1)}^{h_1} x_{\sigma(2)}^{h_2}
\ldots x_{\sigma(n)}^{h_n} \, \biggl| \, \sigma \in S_n,\,
 h_1, \ldots, h_n \in H
\right\rangle_F \subset F \langle X | H \rangle.$$

Lemma~\ref{LemmaAlternateFirst} is an analog of~\cite[Lemma~1]{GiaSheZai}. Recall that if $M$ is an $(H,L)$-module, we denote by $\varphi \colon L \to \mathfrak{gl}(M)$
and $\zeta \colon H \to \End_F(M)$ the corresponding homomorphisms.

\begin{lemma}\label{LemmaAlternateFirst}
Let $L$ be an $H$-module Lie algebra where $H$ is a Hopf algebra over an algebraically closed field $F$ of characteristic $0$, let
$M$ be a faithful finite dimensional irreducible $(H,L)$-module.  Let $(\zeta(\gamma_j))_{j=1}^m$ be a basis of $\zeta(H)$.
Then for some $n\in\mathbb N$ there exist
polynomials $f_j \in P^H_n$, $1 \leqslant j \leqslant m$,
  alternating in $\lbrace x_1, \ldots, x_\ell \rbrace$ and in $\lbrace y_1, \ldots, y_\ell \rbrace \subseteq \lbrace x_{\ell+1}, \ldots, x_n \rbrace$ where $\ell:=\dim L$,
  such that $$\sum_{j=1}^m{f_j(\varphi(\bar x_1), \ldots,  \varphi(\bar x_n))\, \zeta(\gamma_j)}=\id_M$$
  for some $\bar x_i \in L$. 
  In particular, there exists
 a polynomial $f \in P^H_n \backslash \Id^H(\varphi)$
  alternating in $\lbrace x_1, \ldots, x_\ell \rbrace$ and in $\lbrace y_1, \ldots, y_\ell \rbrace \subseteq \lbrace x_{\ell+1}, \ldots, x_n \rbrace$.
\end{lemma}
\begin{proof}
Since $M$ is irreducible, by the Density Theorem,
 $\End_F(M) \cong M_q(F)$ is generated by operators from $\zeta(H)$
and $\varphi(L)$. Here $q := \dim M$. Consider Regev's polynomial
$$
 \hat f(x_1, \ldots, x_{q^2}; y_1, \ldots, y_{q^2})
:=\sum_{\substack{\sigma \in S_q, \\ \tau \in S_q}} (\sign(\sigma\tau))
x_{\sigma(1)}\ y_{\tau(1)}\ x_{\sigma(2)}x_{\sigma(3)}x_{\sigma(4)}
\ y_{\tau(2)}y_{\tau(3)}y_{\tau(4)}\ldots
$$
$$
 x_{\sigma\left(q^2-2q+2\right)}\ldots x_{\sigma\left(q^2\right)}
\ y_{\tau\left(q^2-2q+2\right)}\ldots y_{\tau\left(q^2\right)}.
$$
 This is a central polynomial~\cite[Theorem~5.7.4]{ZaiGia}
for $M_q(F)$, i.e. $\hat f$ is not a polynomial identity for $M_q(F)$
and its values belong to the center of $M_q(F)$.
Since $\hat f$ is alternating, it becomes a nonzero scalar operator
under a substitution of the elements of any basis for 
$\lbrace x_1, \ldots, x_\ell \rbrace$ and $\lbrace y_1, \ldots, y_\ell \rbrace$.

Let $a_1, \ldots, a_\ell$  be a basis of $L$.
Recall that if we have the product of elements of
$\varphi(L)$ and $\zeta(H)$,
we can always move the elements from $\zeta(H)$
to the right, using~(\ref{EqHLModule2}).
Then
$\varphi(a_1), \ldots, \varphi(a_\ell)$, $\left(\varphi\left(a_{i_{11}}\right)
\ldots \varphi\left(a_{i_{1,m_1}}\right)\right)\zeta(h_1)$,
\ldots, $\left(\varphi\left(a_{i_{r,1}}\right)
\ldots \varphi\left(a_{i_{r,m_r}}
\right)\right)\zeta(h_r)$, is a basis of $\End_F(M)$ for appropriate $i_{jk} \in \lbrace 1,2, \ldots, \ell \rbrace$, $h_j \in H$, since $\End_F(M)$ is generated by operators from $\zeta(H)$
and $\varphi(L)$. We replace $x_{\ell+j}$ with $z_{j1}
z_{j2} \ldots z_{j,m_j} \zeta(h_j)$ and $y_{\ell+j}$ with $z'_{j1}
z'_{j2} \ldots z'_{j,m_j} \zeta(h_j)$ in $\hat f$ and denote the expression
obtained by $\tilde f$. Using~(\ref{EqHLModule2}) again,
we can move all $\zeta(h)$, $h \in H$, in $\tilde f$ to the right and rewrite
 $\tilde f$ as $\sum_{j=1}^m{f_j\, \zeta(\gamma_j)}$
where each $f_j \in P^H_{2\ell+2\sum_{i=1}^r m_i}$ is
a  polynomial alternating in $x_1, \ldots, x_\ell$ and
in $y_1, \ldots, y_\ell$. 
Note that $\tilde f$ becomes a nonzero scalar operator on~$M$ under
the substitution $x_i=y_i=\varphi(a_i)$ for $ 1 \leqslant i \leqslant \ell$
and $z_{jk}=z_{jk}'=\varphi(a_{i_{jk}})$ for $1 \leqslant j \leqslant r$, $1 \leqslant k \leqslant m_j$. Dividing all $f_j$ by this scalar, we obtain
the required polynomials.
 In particular, $f_j \notin \Id^H(\varphi)$
for some $1 \leqslant j \leqslant m$ and we can take $f=f_j$.
\end{proof}

\begin{lemma}\label{LemmaTwoColumns}
Let $\alpha_1, \alpha_2, \ldots, \alpha_q$, $\beta_1, \ldots, \beta_q \in F$
where $F$ is an infinite field, $
1 \leqslant k \leqslant q$,
 $\alpha_i\ne 0$ for $1 \leqslant i < k$,
$\alpha_k=0$, and $\beta_k\ne 0$. Then there exists such
$\gamma \in F$ that $\alpha_i + \gamma \beta_i \ne 0$
for all $1 \leqslant i \leqslant k$.
\end{lemma}
\begin{proof} It is sufficient to choose
$\gamma \notin \left\lbrace -\frac{\alpha_1}{\beta_1},
\ldots, -\frac{\alpha_{k-1}}{\beta_{k-1}}, 0\right\rbrace$. It is possible to do
since $F$ is infinite.
\end{proof}

\begin{lemma}\label{LemmaS}
Let $L=B \oplus Z(L)$ be a finite dimensional reductive $H$-module Lie algebra where $H$ is a Hopf algebra
over an algebraically closed field $F$ of characteristic $0$, $B$ is an
$H$-invariant maximal semisimple subalgebra, and
$Z(L)$ is the center of $L$ with a basis $r_1$, $r_2$, \ldots, $r_t$.
Let $M$ be a faithful finite dimensional irreducible $(H,L)$-module.
Suppose that for all $H$-invariant associative subalgebras in $\End_F(M)$
there exists an $H$-invariant Wedderburn~--- Mal'cev decomposition.
Then there exists such alternating in $x_1, x_2, \ldots, x_t$
 polynomial $f \in P^H_t$  that
$f(\varphi(r_1), \ldots, \varphi(r_t))$ is a nondegenerate operator on~$M$.
\end{lemma}
\begin{proof} By Lemma~\ref{LemmaAnnHLmodule}, $Z(L)$ is $H$-invariant.
By Lemma~\ref{LemmaRedIrr},
$M=M_1\oplus\ldots \oplus M_q$ where $M_j$ are $L$-submodules
and $r_i$ acts on each $M_j$ as a scalar operator.
Note that it is sufficient to prove that for each $j$ there exists
such alternating in $x_1, x_2, \ldots, x_t$
 polynomial $f_j \in P^H_t$  that
$f_j(\varphi(r_1), \ldots, \varphi(r_t))$ multiplies
each element of~$M_j$ by a nonzero scalar.
Indeed, in this case Lemma~\ref{LemmaTwoColumns} implies the existence
of such $f=\lambda_1 f_1 +\ldots + \lambda_q f_q$, $\lambda_i \in F$,
that $f(\varphi(r_1), \ldots, \varphi(r_t))$ acts
 on each $M_i$ as a nonzero scalar.

By Lemma~\ref{LemmaAlternateFirst}, there exist $n\in\mathbb N$
and polynomials $\hat f_k \in P^H_n$ alternating in
$\lbrace x_1, x_2, \ldots, x_\ell\rbrace$, $\ell := \dim L$,
such that
\begin{equation}\label{EqSIdM}
\sum_{k=1}^m{\hat f_k(\varphi(r_1), \varphi(r_2), \ldots, \varphi(r_t), \varphi(\bar x_{t+1}), \varphi(\bar x_{t+2}), \ldots,  \varphi(\bar x_n))\, \zeta(\gamma_k)}=\id_M
\end{equation}
  for some $\bar x_i \in L$. (We can substitute
  the elements of any basis of $L$ for $\lbrace x_1, x_2, \ldots, x_\ell\rbrace$ since each $\hat f_k$ is alternating in the first $\ell$ variables.)
 Note that $[\varphi(r_i), \varphi(a)]=0$ for all $a \in L$ and $1 \leqslant i
\leqslant t$. Hence we can move all $\varphi(r_i)$ to the left and rewrite~(\ref{EqSIdM}) as
$$\sum_k \tilde f_k(\varphi(r_1), \varphi(r_2), \ldots, \varphi(r_t)) b_k = \id_M$$
where $\tilde f_k \in P^H_t$ are polynomials
alternating in $\lbrace x_1, x_2, \ldots, x_t\rbrace$, and $b_k \in \End_F(M)$.
Hence for every $1 \leqslant j \leqslant q$ there exists $k$ such that  $\tilde f_k(\varphi(r_1), \varphi(r_2), \ldots, \varphi(r_t))\bigr|_{M_j}\ne 0$.
 Since $\tilde f_j(\varphi(r_1), \varphi(r_2), \ldots, \varphi(r_t))$ is a scalar
operator on $M_j$, we can take $f_j := \tilde f_k$.
\end{proof}

Let $k\ell \leqslant n$ where $k,\ell,n \in \mathbb N$ are some numbers.
 Denote by $Q^H_{\ell,k,n} \subseteq P^H_n$
the subspace spanned by all polynomials that are alternating in
$k$ disjoint subsets of variables $\{x^i_1, \ldots, x^i_\ell \}
\subseteq \lbrace x_1, x_2, \ldots, x_n\rbrace$, $1 \leqslant i \leqslant k$.

Theorem~\ref{TheoremAlternateFinal}
 is an analog of~\cite[Theorem~1]{GiaSheZai}.

\begin{theorem}\label{TheoremAlternateFinal}
Let $L=B \oplus Z(L)$ be a reductive $H$-module Lie algebra where $H$ is a Hopf algebra
over an algebraically closed field $F$ of characteristic $0$, $B$ is an
$H$-invariant maximal semisimple subalgebra, $\ell := \dim L$.
Let $M$ be a faithful finite dimensional irreducible $(H,L)$-module.
Denote the corresponding representation $L \to \mathfrak{gl}(M)$ by $\varphi$.
Suppose that either $Z(L)=0$ or for any $H$-invariant associative subalgebra in $\End_F(M)$
there exists an $H$-invariant Wedderburn~--- Mal'cev decomposition.
Then there exists $T \in \mathbb Z_+$ such that
for any $k \in \mathbb N$
there exists $f \in Q^H_{\ell, 2k, 2k\ell+T} \backslash \Id^H(\varphi)$.
\end{theorem}
\begin{proof} 
Let $f_1=f_1(x_1,\ldots, x_\ell,\ y_1,\ldots, y_\ell,
z_1, \ldots, z_T)$ be the polynomial $f$ from Lemma~\ref{LemmaAlternateFirst}
alternating in $x_1,\ldots, x_\ell$ and in $y_1,\ldots, y_\ell$.
Since $f_1 \in Q^H_{\ell, 2, 2\ell+T} \backslash \Id^H(\varphi)$,
 we may assume that $k > 1$. Note that
$$
f^{(1)}_1(u_1, v_1, x_1, \ldots, x_\ell,\ y_1,\ldots, y_\ell,
z_1, \ldots, z_T) :=$$ $$
\sum^\ell_{i=1} f_1(x_1, \ldots, [u_1, [v_1, x_i]],  \ldots, x_\ell,\ y_1,\ldots, y_\ell,
z_1, \ldots, z_T)$$
is alternating in $x_1,\ldots, x_\ell$ and in $y_1,\ldots, y_\ell$
and $$
f^{(1)}_1(\bar u_1, \bar v_1, \bar x_1, \ldots, \bar x_\ell,\
\bar y_1,\ldots, \bar y_\ell,
\bar z_1, \ldots, \bar z_T) =$$
$$
 \tr(\ad_{\varphi(L)} \bar u_1 \ad_{\varphi(L)} \bar v_1)
f_1(\bar x_1, \bar x_2, \ldots, \bar x_\ell,\ \bar y_1,\ldots, \bar y_\ell,
\bar z_1, \ldots, \bar z_T)
$$
 for any substitution of elements from $\varphi(L)$
 since we may assume $\bar x_1, \ldots, \bar x_\ell$ to be different basis elements.
Here $(\ad a) b = [a,b]$.

Let $$
f^{(j)}_1(u_1, \ldots, u_j, v_1, \ldots, v_j, x_1, \ldots, x_\ell,\ y_1,\ldots, y_\ell,
z_1, \ldots, z_T) :=$$ $$
\sum^\ell_{i=1} f^{(j-1)}_1(u_1, \ldots,  u_{j-1}, v_1, \ldots, v_{j-1},
 x_1, \ldots, [u_j, [v_j, x_i]],  \ldots, x_\ell,\ y_1,\ldots, y_\ell,
z_1, \ldots, z_T),$$
$2 \leqslant j \leqslant s$, $s := \dim B$. Note that
if we substitute an element from $\varphi(Z(L))$ for $u_i$ or $v_i$,
then $f^{(j)}_1$ vanish since $Z(L)$ is the center of $L$.
Again,
$$
f^{(j)}_1(\bar u_1, \ldots, \bar u_j, \bar v_1, \ldots, \bar v_j, \bar x_1, \ldots, \bar x_\ell,\ \bar y_1,\ldots, \bar y_\ell, \bar z_1, \ldots, \bar z_T) =$$
$$ \tr(\ad_{\varphi(L)} \bar u_1 \ad_{\varphi(L)} \bar v_1)
 \tr(\ad_{\varphi(L)} \bar u_2 \ad_{\varphi(L)} \bar v_2)
 \ldots
 \tr(\ad_{\varphi(L)} \bar u_j \ad_{\varphi(L)} \bar v_j)\cdot
$$
\begin{equation}\label{EqKilling}
\cdot
f_1(\bar x_1, \bar x_2, \ldots, \bar x_\ell,\ \bar y_1,\ldots, \bar y_\ell,
\bar z_1, \ldots, \bar z_T)
\end{equation}

Let $\eta$ be the polynomial $f$ from Lemma~\ref{LemmaS}.
We define
$$ f_2(u_1, \ldots, u_\ell, v_1, \ldots, v_\ell,
x_1, \ldots, x_\ell, y_1, \ldots, y_\ell, z_1, \ldots, z_T) :=
$$ $$\sum_{\sigma, \tau \in S_\ell}
\sign(\sigma\tau)
f^{(s)}_1(u_{\sigma(1)}, \ldots, u_{\sigma(s)}, v_{\tau(1)}, \ldots, v_{\tau(s)}, x_1, \ldots, x_\ell,\ y_1,\ldots, y_\ell,
z_1, \ldots, z_T)$$ $$\cdot \eta(u_{\sigma(s+1)}, \ldots, u_{\sigma(\ell)})
\eta(v_{\tau(s+1)}, \ldots, v_{\tau(\ell)}).$$
Then $f_2 \in Q^H_{\ell, 4, 4\ell+T}$. Suppose $a_1, \ldots, a_s \in \varphi(B)$
and $a_{s+1}, \ldots, a_\ell \in \varphi(Z(L))$ form a basis of $\varphi(L)$.
Consider a substitution $x_i=y_i=u_i=v_i=a_i$, $1 \leqslant i \leqslant \ell$.
Suppose that the values $z_j=\bar z_j$, $1 \leqslant j \leqslant T$, are chosen
in such a way that $f_1(a_1, \ldots, a_\ell, a_1, \ldots, a_\ell,
\bar z_1, \ldots, \bar z_T)\ne 0$. We claim that $f_2$ does not vanish either.
Indeed,
$$ f_2(a_1, \ldots, a_\ell, a_1, \ldots, a_\ell,
a_1, \ldots, a_\ell, a_1, \ldots, a_\ell, \bar z_1, \ldots, \bar z_T) = $$
$$\sum_{\sigma, \tau \in S_\ell}
\sign(\sigma\tau)
f^{(s)}_1(a_{\sigma(1)}, \ldots, a_{\sigma(s)}, a_{\tau(1)}, \ldots, a_{\tau(s)}, a_1, \ldots, a_\ell,\ a_1,\ldots, a_\ell,
\bar z_1, \ldots, \bar z_T)$$ $$\cdot \eta(a_{\sigma(s+1)}, \ldots, a_{\sigma(\ell)})
\eta(a_{\tau(s+1)}, \ldots, a_{\tau(\ell)})=$$
$$\left(\sum_{\sigma, \tau \in S_s}
\sign(\sigma\tau)
f^{(s)}_1(a_{\sigma(1)}, \ldots, a_{\sigma(s)}, a_{\tau(1)}, \ldots, a_{\tau(s)}, a_1, \ldots, a_\ell,\ a_1,\ldots, a_\ell,
\bar z_1, \ldots, \bar z_T)\right)\cdot$$
 $$ \left(
\sum_{ \pi, \omega \in S\lbrace s+1, \ldots,
\ell \rbrace} \sign(\pi\omega)
 \eta(a_{\pi(s+1)}, \ldots, a_{\pi(\ell)})
\eta(a_{\omega(s+1)}, \ldots, a_{\omega(\ell)})\right)$$
 since $a_j$, $s < j \leqslant \ell$,
belong to the center of $\varphi(L)$
and $f^{(s)}_j$ vanishes if we substitute
such $a_i$ for $u_i$ or $v_i$.
Here $S\lbrace s+1, \ldots,
\ell \rbrace$ is the symmetric group on $\lbrace s+1, \ldots,
\ell \rbrace$. Note that $\eta$ is alternating. Using~(\ref{EqKilling}), we obtain
$$ f_2(a_1, \ldots, a_\ell, a_1, \ldots, a_\ell,
a_1, \ldots, a_\ell, a_1, \ldots, a_\ell, \bar z_1, \ldots, \bar z_T) = $$
$$ \left(
\sum_{\sigma, \tau \in S_s}
\sign(\sigma\tau) \tr(\ad_{\varphi(L)} a_{\sigma(1)}
\ad_{\varphi(L)} a_{\tau(1)})  \ldots \tr(\ad_{\varphi(L)} a_{\sigma(s)}
\ad_{\varphi(L)} a_{\tau(s)}) \right)\cdot$$ $$
f_1(a_1, \ldots, a_\ell,\ a_1,\ldots, a_\ell,
\bar z_1, \ldots, \bar z_T)  ((\ell-s)!)^2
\left(\eta(a_{s+1}, \ldots, a_\ell)\right)^2.
$$
 Note that
$$\sum_{\sigma, \tau \in S_s}
\sign(\sigma\tau) \tr(\ad_{\varphi(L)} a_{\sigma(1)}
\ad_{\varphi(L)} a_{\tau(1)})  \ldots \tr(\ad_{\varphi(L)} a_{\sigma(s)}
\ad_{\varphi(L)} a_{\tau(s)})
=$$ $$\sum_{\sigma, \tau \in S_s}
\sign(\sigma\tau) \tr(\ad_{\varphi(L)} a_{1}
\ad_{\varphi(L)} a_{\tau\sigma^{-1}(1)})  \ldots
 \tr(\ad_{\varphi(L)} a_{s}
\ad_{\varphi(L)} a_{\tau\sigma^{-1}(s)})\mathrel{\stackrel{(\tau'=\tau\sigma^{-1})}{=}}$$
$$\sum_{\sigma, \tau' \in S_s}
\sign(\tau') \tr(\ad_{\varphi(L)} a_{1}
\ad_{\varphi(L)} a_{\tau'(1)})  \ldots
 \tr(\ad_{\varphi(L)} a_{s}
\ad_{\varphi(L)} a_{\tau'(s)})=$$
$$s!\det(\tr(\ad_{\varphi(L)} a_i \ad_{\varphi(L)} a_j))_{i,j=1}^s=
s!\det(\tr(\ad_{\varphi(B)} a_i \ad_{\varphi(B)} a_j))_{i,j=1}^s \ne 0$$
since the Killing form $\tr(\ad x \ad y)$ of the semisimple
Lie algebra $\varphi(B)$ is nondegenerate.
Thus $$ f_2(a_1, \ldots, a_\ell, a_1, \ldots, a_\ell,
a_1, \ldots, a_\ell, a_1, \ldots, a_\ell, \bar z_1, \ldots, \bar z_T) \ne 0. $$
Note that if $f_1$ is alternating in some of $z_1,\ldots, z_T$,
the polynomial $f_2$
is alternating in those variables too.
Thus if we apply the same procedure to
$f_2$ instead of $f_1$, we obtain $f_3 \in Q^H_{\ell, 6, 6\ell+T}$.
Analogously, we define $f_4$ using $f_3$, $f_5$ using $f_4$, etc.
Eventually, we obtain
$f:=f_k \in Q^H_{\ell, 2k, 2k\ell+T} \backslash \Id^H(\varphi)$.
\end{proof}

\section{On factors of the adjoint representation of $L$}
\label{SectionAux}

In Sections~\ref{SectionAux}--\ref{SectionLower}
we consider an $H$-nice Lie algebra $L$ (see Subsection~\ref{SubsectionHnice}).

\begin{lemma}\label{LemmaLBSN} Consider the adjoint action of $B$ on $L$.
Then $L$ is a completely reducible $(H,B)$-module.
Moreover, there exists an $H$-submodule $S \subseteq R$
such that $L=B\oplus S\oplus N$ (direct sum of $H$-submodules)
and $[B,S]=0$.
\end{lemma}
\begin{proof}
If $L$ is semisimple, i.e. $L=B$, then by~\cite[Theorem~6]{ASGordienko4},
$L$ is a direct sum of $H$-simple ideals, i.e. $L$ is a completely reducible $(H,B)$-module.

Suppose $L\ne B$. Note that $(\ad B) \subseteq \mathfrak{gl}(L)$ is a finite dimensional semisimple Lie algebra. Hence, by the Weyl theorem, $L$ is a completely reducible $(\ad B)$-module.
Thus, by Condition~\ref{ConditionLComplHred} of Subsection~\ref{SubsectionHnice},
$L$ is a completely reducible $(H,\ad B)$- and $(H,B)$-module.

Since the nilpotent radical $N$ of $L$ is an $H$-submodule,
 there exists an $H$-invariant $B$-submodule $S \subseteq R$ such that
 $R=S\oplus N$ (direct sum of $H$-submodules).
 Therefore $[B, S] \subseteq S$. However, 
by \cite[Proposition 2.1.7]{GotoGrosshans},
 $[B, S] \subseteq [L, R] \subseteq N$. Hence $[B, S]=0$.
\end{proof}

Lemma~\ref{LemmaIrrAnnBS} is a $H$-invariant analog of~\cite[Lemma 4]{ZaiLie}.

\begin{lemma} \label{LemmaIrrAnnBS}
Let $J \subseteq I \subseteq L$ be $H$-invariant ideals
such that  $I/J$ is an irreducible $(H,L)$-module.
Then \begin{enumerate}
\item $\Ann (I/J)\cap B$ and $\Ann (I/J)\cap S$ are $H$-submodules
of $L$; \label{BSInv}
\item $\Ann (I/J)=(\Ann (I/J)\cap B)\oplus (\Ann (I/J)\cap S)
\oplus N$. \label{BSDecomp}
\end{enumerate}
\end{lemma}
\begin{proof}
By Lemma~\ref{LemmaAnnHLmodule},
$\Ann (I/J)$,
$\Ann (I/J)\cap B$, and $\Ann (I/J)\cap S$ are $H$-submodules.

Note that $N$ is a nilpotent ideal. Hence $[[\underbrace{N, [N, \ldots,[N}_p, I]\ldots]]=0$
for some $p\in\mathbb N$. Thus $N^p(I/J)=0$. However $I/J$ is an irreducible $(H,L)$-module
and either $N(I/J)=0$ or $N(I/J)=I/J$. 
Hence $N(I/J)=0$ and $N \subseteq \Ann(I/J)$. In order to prove the lemma, it is sufficient to show that if $b+s \in \Ann (I/J)$, $b \in B$, $s \in S$, then
$b,s \in \Ann (I/J)$.
Denote by
$\varphi \colon L \to \mathfrak{gl}(I/J)$ the homomorphism corresponding to $L$-module
structure on $I/J$. Then $\varphi(b)+\varphi(s)=0$
and $$[\varphi(b), \varphi(B)]=[-\varphi(s), \varphi(B)]=0$$
since $[B,S]=0$.
Hence $\varphi(b)$ belongs to the center of $\varphi(B)$ and $\varphi(b)=\varphi(s)=0$ since $\varphi(B)$ is semisimple.
Thus $b,s \in \Ann (I/J)$ and the lemma is proved.
\end{proof}

\section{Upper bound}
\label{SectionUpper}

Fix a composition chain of $H$-invariant ideals
$$L=L_0 \supsetneqq L_1 \supsetneqq L_2 \supsetneqq \ldots \supsetneqq
N\supsetneqq \ldots \supsetneqq L_{\theta-1}
 \supsetneqq L_\theta = \{0\}.$$
 Let $\height a := \max_{a \in L_k} k$ for $a \in L$.

\begin{remark}
If $d=d(L)=0$, then $L = \Ann(L_{i-1}/L_i)$
for all $1 \leqslant i \leqslant \theta$ and
 $[a_1, a_2, \ldots, a_n] =0$ for all $a_i \in L$
 and $n \geqslant \theta +1$. Thus $c^H_n(L)=0$
 for all $n \geqslant \theta +1$. Therefore we assume $d > 0$.
\end{remark}

 Let $Y:=\lbrace y_{11}, y_{12}, \ldots, y_{1j_1};\,
 y_{21}, y_{22}, \ldots, y_{2j_2}; \ldots;\,
 y_{m1}, y_{m2}, \ldots, y_{mj_m}\rbrace$,
 $Y_1$, \ldots, $Y_q$, and $\lbrace z_1, \ldots, z_m\rbrace$
 be subsets of $\lbrace x_1, x_2, \ldots, x_n\rbrace$
 such that $Y_i \subseteq Y$, $|Y_i|=d+1$, $ Y_i \cap Y_j = \varnothing$
 for $i \ne j$,
 $Y \cap \lbrace z_1, \ldots, z_m\rbrace = \varnothing$,
  $j_i \geqslant 0$.
  Denote $$f_{m,q}:=\Alt_{1} \ldots \Alt_{q} \Bigl[[z_1^{h_1}, y_{11}^{h_{11}}, y_{12}^{h_{12}},
  \ldots, y_{1j_1}^{h_{1j_1}}],
 [z_2^{h_2},y_{21}^{h_{21}},y_{22}^{h_{22}},\ldots, y_{2j_2}^{h_{2j_2}}], \ldots,
 $$ $$[ z_m^{h_m}, y_{m1}^{h_{m1}}, y_{m2}^{h_{m2}}, \ldots, y_{mj_m}^{h_{mj_m}}]\Bigr]$$
 where $\Alt_i$ is the operator of alternation on the variables of $Y_i$,
 $h_i, h_{ij}\in H$.

Let $\xi \colon L(X | H) \to L$
be the homomorphism of $H$-module algebras induced by some substitution $\lbrace x_1, x_2, \ldots, x_n, \ldots \rbrace \to L$.
We say that $\xi$ is \textit{proper} for  $f_{m,q}$ if
 $\xi(z_1) \in B \cup S \cup N$,
 $\xi(z_i) \in N$ for $2\leqslant i \leqslant m$,
  and
  $\xi(y_{ik})\in B \cup S$ for $1\leqslant i \leqslant m$,
   $1 \leqslant k \leqslant j_i$.

\begin{lemma}\label{LemmaReduct}
Let $\xi$ be a \textit{proper} homomorphism for $f_{m,q}$.
Then $\xi(f_{m,q})$ can be rewritten as a
sum of $\psi(f_{m+1,q'})$ where $\psi$
is a proper homomorphism for $f_{m+1,q'}$, $q' \geqslant q - (\dim L)m - 2$.
  ($Y'$, $Y'_i$, $z'_1, \ldots, z'_{m+1}$ may be different
for different terms.)
\end{lemma}
\begin{proof}
Let $\alpha_i := \height \xi(z_i)$.
We will use induction on $\sum_{i=1}^m \alpha_i$.
(The sum will grow.)
 Note that $ \alpha_i \leqslant \theta \leqslant \dim L$ and the induction will
 eventually stop.
Denote $I_i := L_{\alpha_i}$, $J_i := L_{\alpha_{i+1}}$.

First, consider the case when $I_1, \ldots, I_m$,
$J_1, \ldots, J_m$ do not satisfy Conditions 1--2.
In this case we can choose $H$-invariant $B$-submodules
$T_i$, $I_i = T_i \oplus J_i$, such that
\begin{equation}\label{EqTqUpperZero}
 \bigl[[T_1, \underbrace{L, \ldots, L}_{q_1}], [T_2, \underbrace{L, \ldots, L}_{q_2}], \ldots, [T_m,
 \underbrace{L, \ldots, L}_{q_m}]\bigr] = 0
 \end{equation}
 for all $q_i \geqslant 0$.
Rewrite $\xi(z_i)=a'_i+a''_i$, $a'_i \in T_i$, $a''_i \in J_i$.
Note that $\height a''_i > \height \xi(z_i)$.
Since $f_{m,q}$ is multilinear, we can rewrite
$\xi(f_{m,q})$ as a sum of similar
terms $\tilde\xi(f_{m,q})$ where $\tilde\xi(z_i)$ equals either
$a'_i$ or $a''_i$. By~(\ref{EqTqUpperZero}), the term where all
$\tilde\xi(z_i)=a'_i \in T_i$, equals $0$.
For the other terms $\tilde\xi(f_{m,q})$ we
have $\sum_{i=1}^m \height \tilde\xi(z_i) > \sum_{i=1}^m \height \xi(z_i)$.

Thus without lost of generality
we may assume that $I_1, \ldots, I_m$,
$J_1, \ldots, J_m$ satisfy Conditions 1--2.
In this case, $\dim(\Ann(I_1/J_1) \cap \ldots \cap \Ann(I_m/J_m))
\geqslant \dim(L)-d$.
In virtue of Lemma~\ref{LemmaIrrAnnBS},
$$\Ann(I_1/J_1) \cap \ldots \cap \Ann(I_m/J_m)
= (B \cap \Ann(I_1/J_1) \cap \ldots \cap \Ann(I_m/J_m)) \oplus$$ $$
(S \cap \Ann(I_1/J_1) \cap \ldots \cap \Ann(I_m/J_m))\ \oplus\ N.$$
Choose a basis in $B$ that includes a basis of
$B \cap \Ann(I_1/J_1) \cap \ldots \cap \Ann(I_m/J_m)$,
and a basis in $S$
that includes the basis of $S \cap \Ann(I_1/J_1) \cap \ldots \cap \Ann(I_m/J_m)$.
Since $f_{m,q}$ is multilinear, we may assume
that only basis elements are substituted for $y_{k\ell}$. Note that $f_{m,q}$
is alternating in $Y_i$. Hence, if $\xi(f_{m,q})\ne 0$, then for every $1 \leqslant i \leqslant q$
there exists $y_{kj} \in Y_i$ such that
either $$\xi(y_{kj}) \in B \cap \Ann(I_1/J_1) \cap \ldots \cap \Ann(I_m/J_m)$$
or $$\xi(y_{kj}) \in S \cap \Ann(I_1/J_1) \cap \ldots \cap \Ann(I_m/J_m).$$

Consider the case when $\xi(y_{kj}) \in B \cap \Ann(I_1/J_1) \cap \ldots \cap \Ann(I_m/J_m)$
for some $y_{kj}$.
Since $L$ is a completely reducible $(H,B)$-module (Lemma~\ref{LemmaLBSN}), we can choose $H$-invariant $B$-submodules $T_k$ such that $I_k = T_k \oplus J_k$. We may assume
that $\xi(z_k) \in T_k$ since elements of $J_k$ have greater heights.
Therefore $[\xi(z_k^{h_k}), a] \in T_k \cap J_k$
 for all $a \in B \cap \Ann(I_1/J_1) \cap \ldots \cap \Ann(I_m/J_m)$.
Hence $[\xi(z_k^{h_k}),a]=0$. Moreover, $B \cap \Ann(I_1/J_1) \cap \ldots \cap \Ann(I_m/J_m)$ is a $H$-invariant ideal of $B$ and $[B,S]=0$. Thus, applying Jacobi's identity
several times, we obtain
$$\xi([z_k^{h_{k}},y_{k1}^{h_{k1}},\ldots, y_{kj_k}^{h_{kj_k}}]) = 0.$$
Expanding the alternations, we get $\xi(f_{m,q})=0$.

Consider the case when $\xi(y_{k\ell}) \in S \cap \Ann(I_1/J_1) \cap \ldots \cap \Ann(I_m/J_m)$
for some $y_{k\ell} \in Y_q$. Expand the alternation $\Alt_q$ in $f_{m,q}$
and rewrite $f_{m,q}$ as a sum of
$$\tilde f_{m,q-1} :=\Alt_{1} \ldots \Alt_{q-1} \bigl[[z_1^{h_1}, y_{11}^{h_{11}}, y_{12}^{h_{12}},
  \ldots, y_{1j_1}^{h_{1j_1}}],
 [z_2^{h_2},y_{21}^{h_{21}},y_{22}^{h_{22}},\ldots, y_{2j_2}^{h_{2j_2}}], \ldots,
 $$ $$[ z_m^{h_m}, y_{m1}^{h_{m1}}, y_{m2}^{h_{m2}}, \ldots, y_{mj_m}^{h_{mj_m}}]\bigr].$$
 The operator $\Alt_q$ may change indices, however we
 keep the notation $y_{k\ell}$ for the variable with the property
 $\xi(y_{k\ell}) \in S \cap \Ann(I_1/J_1) \cap \ldots \cap \Ann(I_m/J_m)$.
Now the alternation does not affect $y_{k\ell}$.
Note that $$[z_k^{h_{k}},y_{k1}^{h_{k1}},\ldots, y_{k\ell}^{h_{k\ell}}, \ldots, y_{kj_k}^{h_{kj_k}}] = [z_k^{h_{k}},y_{k\ell}^{h_{k\ell}},y_{k1}^{h_{k1}},\ldots, y_{kj_k}^{h_{kj_k}}] +$$ $$
\sum\limits_{\beta=1}^{\ell-1}
[z_k^{h_{k}},y_{k1}^{h_{k1}},\ldots, y_{k,\beta-1}^{h_{k,\beta-1}}, [y_{k\beta}^{h_{k\beta}}, y_{k\ell}^{h_{k\ell}}],
 y_{k,\beta+1}^{h_{k,\beta+1}},\ldots, y_{k,{\ell-1}}^{h_{k,{\ell-1}}}, y_{k,{\ell+1}}^{h_{k,{\ell+1}}}, \ldots, y_{kj_k}^{h_{kj_k}}].$$

In the first term we replace $[z_k^{h_{k}},y_{k\ell}^{h_{k\ell}}]$
with $z'_k$ and define $\xi'(z'_k)
:= \xi([z_k^{h_{k}},y_{k\ell}^{h_{k\ell}}])$, $\xi'(x) := \xi(x)$ for other
variables~$x$. Then $\height\xi'(z'_k) >
\height\xi(z_k)$ and we can use the inductive assumption.
If $y_{k\beta}\in Y_j$ for some $j$, then we expand the alternation $\Alt_j$
in this term in $\tilde f_{m,q-1}$.
If $\xi(y_{k\beta}) \in B$, then the term is zero.
If $\xi(y_{k\beta}) \in S$, then $\xi([y_{k\beta}^{h_{k\beta}}, y_{k\ell}^{h_{k\ell}}]) \in N$.
We replace $[y_{k\beta}^{h_{k\beta}}, y_{k\ell}^{h_{k\ell}}]$
 with an additional variable $z'_{m+1}$
and define $\psi(z'_{m+1}):=\xi([y_{k\beta}^{h_{k\beta}},
 y_{k\ell}^{h_{k\ell}}])$, $\psi(x):=\xi(x)$
 for other variables $x$.
Applying Jacobi's identity several times,
we obtain the polynomial of the desired form. In each inductive step
we reduce $q$ no more than by $1$ and the maximal number of
inductive steps equals $(\dim L)m$. This finishes the proof.
\end{proof}

Since $N$ is a nilpotent ideal,
 $N^{p} = 0$ for some $p\in \mathbb N$.

\begin{lemma}\label{LemmaUpper}
If $\lambda = (\lambda_1, \ldots, \lambda_s) \vdash n$
and $\lambda_{d+1} \geqslant p((\dim L)p+3)$ or $\lambda_{\dim L+1} > 0$, then
$m(L, H, \lambda) = 0$.
\end{lemma}

\begin{proof}
It is sufficient to prove that $e^{*}_{T_\lambda} f \in \Id^H(L)$
for every $f\in V^H_n$ and a Young tableau $T_\lambda$, $\lambda \vdash n$, with
$\lambda_{d+1} \geqslant p((\dim L)p+3)$ or $\lambda_{\dim L+1} > 0$.

Fix some basis of $L$ that is a union of
bases of $B$, $S$, and $N$.
Since polynomials are multilinear, it is
sufficient to substitute only basis elements.
 Note that
$e^{*}_{T_\lambda} = b_{T_\lambda} a_{T_\lambda}$
and $b_{T_\lambda}$ alternates the variables of each column
of $T_\lambda$. Hence if we make a substitution and $
e^{*}_{T_\lambda} f$ does not vanish, then this implies that different basis elements
are substituted for the variables of each column.
But if $\lambda_{\dim L+1} > 0$, then the length of the first column is greater
than $(\dim L)$. Therefore, $e^{*}_{T_\lambda} f \in \Id^H(L)$.

Consider the case $\lambda_{d+1} \geqslant p((\dim L)p+3)$.
 Let $\xi$ be a substitution of basis elements for the variables
 $x_1, \ldots, x_n$.
Then $e^{*}_{T_\lambda}f$ can be rewritten as a sum of polynomials $f_{m,q}$
where $1 \leqslant m \leqslant p$, $q \geqslant p((\dim L)p+2)$, and
$z_i$, $2\leqslant i \leqslant m$, are replaced with
 elements of $N$.  (For different terms $f_{m,q}$,
 numbers $m$ and $q$,
 variables $z_i$, $y_{ij}$, and sets $Y_i$ can be different.)
 Indeed, we expand symmetrization on all variables and alternation on
 the variables replaced with elements from  $N$.
    If we have no variables replaced
 with elements from $N$, then we take $m=1$,
 rewrite the polynomial $f$ as a sum of long commutators,
 in each long commutator expand the alternation on the set that
 includes one of the variables in the inner commutator, and denote
 that variable by $z_1$.
 Suppose we have variables replaced
 with elements from $N$. We denote them by $z_k$.
Then, using Jacobi's identity, we can put one of such variables
inside a long commutator and group all the variables,
replaced with elements from  $B \cup S$, around $z_k$
such that each $z_k$ is inside the corresponding long commutator.

 Applying Lemma~\ref{LemmaReduct} many times, we increase $m$.
 The ideal $N$ is nilpotent and $\xi(f_{p+1,q})=0$
 for every $q$ and a proper homomorphism~$\xi$.
  Reducing $q$ no more than by $p((\dim L)p+2)$,
 we obtain $\xi(e^{*}_{T_\lambda}f)=0$.
\end{proof}

We also need the upper bound for the multiplicities.

\begin{theorem}\label{TheoremColength}
Let $L$ be a finite dimensional $H$-module Lie algebra 
over a field $F$ of characteristic $0$ where $H$ is a Hopf algebra.
Then there exist constants $C_3 > 0$, $r_3 \in \mathbb N$ such that
 $$\sum_{\lambda \vdash n} m(L,H,\lambda)
\leqslant C_3n^{r_3}\text{ for all }n \in \mathbb N.$$
\end{theorem}
\begin{remark}
Cocharacters do not change
upon an extension of the base field $F$
(the proof is completely analogous to \cite[Theorem 4.1.9]{ZaiGia}),
 so we may assume $F$ to be algebraically closed.
\end{remark}
\begin{proof}
Consider ordinary polynomial identities and cocharacters of $L$.
In fact, we may define them as $H$-identities and $H$-cocharacters
for $H=F$:
$V_n := V^F_n$,
$\chi_n(L) := \chi^F_n(L)$,
$m(L,\lambda):= m(L,F,\lambda)$,
$\Id(L):=\Id^F(L)$.
By~\cite[Theorem~3.1]{GiaRegZaicev},
\begin{equation}\label{EqOrdinaryColength}
\sum_{\lambda \vdash n} m(L,\lambda)
\leqslant C_4 n^{r_4}
\end{equation}
 for some $C_4 > 0$ and $r_4 \in \mathbb N$.

Let $G_1 \subseteq G_2$ be finite groups and let $W$ be
an $FG_2$-module.
Denote by $W \downarrow G_1$ the module $W_2$ with the $G_2$-action restricted to $G_1$.

Let $\zeta \colon H \to \End_F(L)$  be the homomorphism corresponding to the $H$-action,
and let $\bigl(\zeta(\gamma_j)\bigr)_{j=1}^m$, $\gamma_j \in H$, be a basis in $\zeta(H)$. 

Consider the diagonal embedding $\varphi \colon S_n \to S_{mn}$,
$$\varphi(\sigma):=\left(\begin{array}{cccc|lllc|c}
1 & 2 & \ldots & n & n+1 & n+2 & \ldots & 2n & \ldots  \\
\sigma(1) & \sigma(2) & \ldots & \sigma(n) &
n+\sigma(1) & n+\sigma(2) & \ldots & n+\sigma(n) & \ldots
\end{array}
\right)$$
and the $S_n$-homomorphism $\pi \colon (V_{mn}\downarrow \varphi(S_n)) \to V_n^H$
 defined by $\pi(x_{n(i-1)+t})=x^{\gamma_i}_t$, $1 \leqslant i \leqslant m$,
 $1 \leqslant t \leqslant n$. Note that $\pi(V_{mn} \cap \Id(L))
 \subseteq V^H_n \cap \Id^H(L)$.
 By~(\ref{Eqhtogammaj}), the $FS_n$-module $\frac{V^H_n}{V^H_n \cap \Id^H(L)}$
 is a homomorphic image of the $FS_n$-module
 $\left(\frac{V_{mn}}{V_{mn} \cap \Id(L)}\right)\downarrow \varphi(S_n)$.
 Denote by  $\length(M)$ the number of irreducible components
of a module $M$.
Then
$$ \sum_{\lambda \vdash n} m(L,H,\lambda)
= \length\left(\frac{V^H_n}{V^H_n \cap \Id^H(L)}\right)
 \leqslant \length\left(\left(\frac{V_{mn}}{V_{mn} \cap \Id(L)}\right)\downarrow \varphi(S_n)\right).
$$

Therefore, it is sufficient to prove
 that  $\length\left(\left(\frac{V_{mn}}{V_{mn} \cap \Id(L)}\right)\downarrow \varphi(S_n)
 \right)$ is polynomially bounded.
Replacing $|G|$ with $m$ in~\cite[Lemma 10 and 12]{ASGordienko2} (or, alternatively, using the proof of~\cite[Theorem~13~(b)]{BereleHopf}),
we derive this from~(\ref{EqOrdinaryColength}).
\end{proof}

Now we can prove
\begin{theorem}\label{TheoremUpper} If $L$ is an $H$-nice algebra and $d > 0$, then there exist constants $C_2 > 0$, $r_2 \in \mathbb R$
such that $c^H_n(L) \leqslant C_2 n^{r_2} d^n$
for all $n \in \mathbb N$. In the case $d=0$, the algebra $L$ is nilpotent.
\end{theorem}
\begin{proof}
Lemma~\ref{LemmaUpper} and~\cite[Lemmas~6.2.4, 6.2.5]{ZaiGia}
imply
$$
\sum_{m(L,H, \lambda)\ne 0} \dim M(\lambda) \leqslant C_5 n^{r_5} d^n
$$
for some constants $C_5, r_5 > 0$.
Together with Theorem~\ref{TheoremColength} this inequality yields the upper bound.
\end{proof}
\begin{remark}
In fact, we prove Theorem~\ref{TheoremUpper} for all $H$-module Lie algebras $L$ such that Lemma~\ref{LemmaLBSN} holds for $L$.
\end{remark}

\section{Lower bound}
\label{SectionLower}

By the definition of $d=d(L)$ (see Subsection~\ref{SubsectionHopfPIexp}),
 there exist $H$-invariant ideals $I_1$, $I_2$, \ldots, $I_r$,
$J_1$, $J_2$, \ldots, $J_r$, $r \in \mathbb Z_+$, of the algebra $L$,
satisfying Conditions 1--2, $J_k \subseteq I_k$, such that
$$d = \dim \frac{L}{\Ann(I_1/J_1) \cap \dots \cap \Ann(I_r/J_r)}.$$
We consider the case $d > 0$.

Without loss of generality we may assume that
$$ \bigcap\limits_{k=1}^r \Ann(I_k/J_k) \ne
\bigcap\limits_{\substack{\phantom{,}k=1,\\ k\ne\ell}}^r \Ann(I_k/J_k)$$
for all $1 \leqslant \ell \leqslant r$.
In particular, $L$ has nonzero action on each $I_k/J_k$.

\begin{remark}
If $L$ is semisimple, then, by Example~\ref{ExamplePIexpBSS},
we may assume that $r=1$, $I_1=B_1$ for some $H$-simple ideal $B_1$ of $L=B$, and $J_1=0$.
In this case, we skip Lemmas~\ref{LemmaChooseReduct}--\ref{ChooseSubmodule},
and define $T_1= \tilde T_1 = B_1$ (which is a faithful irreducible $(H,B_1)$-module), $j_1 = 1$, and $q_1 = 0$.
In Lemmas~\ref{LemmaChooseReduct}--\ref{ChooseSubmodule} we assume that
$L$ satisfy Conditions 1--4 of Subsection~\ref{SubsectionHnice}.
\end{remark}

Our aim is to present a partition $\lambda \vdash n$
with $m(L, H, \lambda)\ne 0$ such that $\dim M(\lambda)$
has the desired asymptotic behavior.
We will glue alternating polynomials constructed
 in Theorem~\ref{TheoremAlternateFinal}
for faithful irreducible modules
over reductive algebras. In order to do this,
we have to choose the reductive algebras.

\begin{lemma}\label{LemmaChooseReduct}
There exist $H$-invariant ideals $B_1, \ldots, B_r$
in $B$ and $H$-submodules
 $\tilde R_1,\ldots,\tilde R_r \subseteq S$
(some of $\tilde R_i$ and $B_j$ may be zero)
such that
\begin{enumerate}
\item $B_1+ \ldots + B_r=B_1\oplus \ldots \oplus B_r$;
\item $\tilde R_1+ \ldots + \tilde R_r=\tilde R_1\oplus \ldots \oplus \tilde R_r$;
\item $\sum\limits_{k=1}^r \dim (B_k\oplus  \tilde R_k) = d$;
\item $I_k/J_k$ is a faithful
$(B_k\oplus\tilde R_k\oplus N)/N$-module;
\item $I_k/J_k$ is an irreducible
$\left(H, \left(\sum_{i=1}^r (B_i\oplus \tilde R_i)\oplus N
\right)/N\right)$-module;
\item $B_i I_k/J_k = \tilde R_i I_k/J_k = 0$ for $i > k$.
\end{enumerate}
\end{lemma}
\begin{proof}
Consider $N_\ell := \bigcap\limits_{k=1}^\ell \Ann (I_k/J_k)$,
$1 \leqslant \ell \leqslant r$, $N_0 := L$.
Note that $N_{\ell}$ are $H$-invariant.
Since $B$ is semisimple, by~\cite[Theorem~6]{ASGordienko4}, we can choose such
$H$-invariant ideals $B_\ell$
that $N_{\ell-1} \cap B =
 B_\ell \oplus (N_\ell \cap B)$.
Also, applying Condition~\ref{ConditionLComplHred} of Subsection~\ref{SubsectionHnice} to zero algebra,
we obtain that $L$ is a completely reducible $H$-module. Hence we can choose such $H$-submodules
 $\tilde R_\ell$ that
 $N_{\ell-1} \cap S = \tilde R_\ell
  \oplus (N_\ell \cap S)$.
Therefore Properties 1, 2, 6 hold.

 By Lemma~\ref{LemmaIrrAnnBS},
 $N_k = (N_k \cap B) \oplus  (N_k \cap S) \oplus  N$.
 Thus Property~4 holds. Furthermore,
 $$N_{\ell-1} = B_\ell \oplus  (N_\ell \cap B) \oplus
  \tilde R_\ell \oplus
   (N_\ell \cap S) \oplus  N
   = (B_\ell\oplus \tilde R_\ell)\oplus N_\ell$$
   (direct sum of subspaces).
   Hence $L = \left(\bigoplus_{i=1}^r (B_i\oplus \tilde R_i)\right)
   \oplus  N_r$,  and
   Properties 3 and 5 hold too.
\end{proof}

Denote by $\zeta \colon H \to \End_F(L)$ the homomorphism corresponding to the $H$-action.
Let $A$ be the associative subalgebra
in $\End_F (L)$ generated by operators from $(\ad L)$
and $\zeta(H)$.  Let $a_{\ell 1}, \ldots, a_{\ell, k_\ell}$ be a basis of $\tilde R_\ell$.

\begin{lemma}\label{LemmaGeneralizedJordanAd}
There exist decompositions $\ad a_{ij} = c_{ij} + d_{ij}$,
$1 \leqslant i \leqslant r$, $1 \leqslant j \leqslant k_i$,
such that $c_{ij} \in A$ acts as a diagonalizable operator on $L$, $d_{ij} \in J(A)$,
 elements $c_{ij}$ commute with each other,
 and $c_{ij}$ and $d_{ij}$ are polynomials in $\ad a_{ij}$.
 Moreover, $R_\ell := \langle c_{\ell1}, \ldots, c_{\ell, k_\ell} \rangle_F$
 are $H$-submodules in $A$.
\end{lemma}
\begin{proof}
We apply Lemma~\ref{LemmaGeneralizedJordan} for $W=\tilde R_\ell$ and $\varphi = \ad$.
\end{proof}

Let $$\tilde B := \left(\bigoplus_{i=1}^r \ad B_i\right)\oplus \langle c_{ij} \mid 1\leqslant i \leqslant r,  1 \leqslant j \leqslant k_i
  \rangle_F,$$
  $$\tilde B_0 := (\ad B)\oplus \langle c_{ij} \mid 1\leqslant i \leqslant r,  1 \leqslant j \leqslant k_i
  \rangle_F \subseteq A.$$
  In virtue of the definition of $S$ (see Section~\ref{SectionAux}), Lemma~\ref{LemmaGeneralizedJordanAd}, and the Jacobi identity, \begin{equation}\label{EqcijCentral}
 [c_{ij}, \ad B]=0.\end{equation} Hence
(\ref{EqHmoduleLieAlgebra}) holds for both $\tilde B$ and $\tilde B_0$. Thus $\tilde B$ and $\tilde B_0$ are $H$-module Lie algebras.
  
  \begin{lemma}\label{LemmaBcReducible}
The space $L$ is a completely reducible $(H,\tilde B_0)$-module.
Moreover, $L$ is a completely reducible $(H,(\ad B_k)\oplus R_k)$-module for any $1 \leqslant k \leqslant r$.
\end{lemma}
\begin{proof} By Condition~\ref{ConditionLComplHred} of Subsection~\ref{SubsectionHnice}, it is sufficient to show that
$L$ is a completely reducible $\tilde B_0$-module
and a completely reducible $(\ad B_k)\oplus R_k$-module
disregarding the $H$-action.
 The elements $c_{ij}$ are diagonalizable on $L$ and commute.
  Therefore, an eigenspace of any $c_{ij}$ is invariant
   under the action of other $c_{k\ell}$. Using induction,
    we split $L = \bigoplus_{i=1}^\alpha W_i$
where $W_i$ are intersections of eigenspaces of $c_{k\ell}$
and elements $c_{k\ell}$ act as scalar operators on $W_i$.
By~(\ref{EqcijCentral}), the spaces $W_i$ are $B$-submodules
 and $L$ is a completely reducible
 $\tilde B_0$-module and $(\ad B_k)\oplus R_k$-module since $B$ and $B_k$ are semisimple.
\end{proof}

\begin{lemma}\label{LemmaSiProperties}
There exist complementary subspaces $I_k=
\tilde T_k \oplus J_k$ such that
\begin{enumerate}
\item $\tilde T_k$ is an $H$-invariant $B$-submodule and an irreducible $(H,\tilde B)$-submodule;
\item $\tilde T_k$ is a completely reducible faithful $(H,(\ad B_k)\oplus R_k)$-module;
\item $\sum\limits_{k=1}^r\dim ((\ad B_k)\oplus R_k) = d$;
\item $B_i \tilde T_k = R_i \tilde T_k = 0$ for $i > k$.
\end{enumerate}
\end{lemma}
\begin{proof} By Lemma~\ref{LemmaBcReducible}, $L$ is a completely reducible $(H,\tilde B_0)$-module.
  Therefore, for every $J_k$ we can choose
  a complementary $H$-invariant $\tilde B_0$-submodules
  $\tilde T_k$ in $I_k$. Then $\tilde T_k$ are both $B$- and $\tilde B$-submodules.

Note that $(\ad a_{ij})w=c_{ij}w$ for all $w \in I_k/J_k$
since $I_k/J_k$ is an irreducible $A$-module and $J(A)\,I_k/J_k = 0$.
Hence, by Lemma~\ref{LemmaChooseReduct}, $I_k/J_k$ is a
 faithful $(\ad B_k)\oplus R_k$-module,
  $R_i\, I_k/J_k = 0$ for $i > k$
and the elements $c_{ij}$ are linearly independent.
Moreover, by Property 5 of Lemma~\ref{LemmaChooseReduct},
 $I_k/J_k$ is an irreducible $\left(H,\left(\sum_{i=1}^r
 (B_i\oplus \tilde R_i)\oplus N
\right)/N\right)$-module.
However
$\left(\sum_{i=1}^r (B_i\oplus \tilde R_i)\oplus N
\right)/N$ acts on $I_k/J_k$
by the same operators as $\tilde B$. Thus
 $\tilde T_k \cong I_k/J_k$  is an irreducible $(H,\tilde B)$-module.
  Property 1 is proved. By Lemma~\ref{LemmaBcReducible}, $L$ is a completely reducible $(H,(\ad B_k)\oplus R_k)$-module for any $1 \leqslant k \leqslant r$. Using the isomorphism $\tilde T_k \cong I_k/J_k$, we obtain Properties 2 and 4 from the remarks above.
Property 3 is a consequence of Property 3 of Lemma~\ref{LemmaChooseReduct}.
\end{proof}

\begin{lemma}\label{ChooseSubmodule} For all $1 \leqslant k \leqslant r$
we have
 $$\tilde T_k = T_{k1} \oplus T_{k2} \oplus \dots
\oplus T_{km}$$ where $T_{kj}$ are faithful irreducible
$(H, (\ad B_k)\oplus R_k)$-submodules, $m \in \mathbb N$,
$1 \leqslant j \leqslant m$.
\end{lemma}
\begin{proof}
By Lemma~\ref{LemmaSiProperties}, Property 2,
$$\tilde T_k = T_{k1} \oplus T_{k2} \oplus \dots
\oplus T_{km}$$ for some irreducible
$(H,(\ad B_k)\oplus R_k)$-submodules.
Suppose $T_{kj}$ is not faithful for some $1 \leqslant j \leqslant m$.
Hence $b T_{kj}=0$ for some $b \in (\ad B_k)\oplus R_k$,
$b \ne 0$.
Note that $\tilde B = ((\ad B_k)\oplus R_k) \oplus \tilde B_k$
where $$\tilde B_k := \bigoplus_{i\ne k} (\ad B_i )\oplus
\bigoplus_{i\ne k} R_i$$ and $[(\ad B_k)\oplus R_k, \tilde B_k]=0$.
Denote by $\widehat B_k$ the associative subalgebra
of $\End_F(\tilde T_k)$ with $1$
generated by operators from $\tilde B_k$. This subalgebra is $H$-invariant.
Then $$[(\ad B_k)\oplus R_k, \widehat B_k]=0$$ and $\sum_{a \in \widehat B_k}
a T_{kj} \supseteq T_{kj}$ is a $H$-invariant $\tilde B$-submodule
of $\tilde T_k$ since $$h\left(\sum_{a \in \widehat B_k}
a T_{kj}\right) = \sum_{a \in \widehat B_k}
(h_{(1)}a) (h_{(2)}T_{kj}) \subseteq \sum_{a \in \widehat B_k}
a T_{kj}$$ for all $h\in H$. Thus
$ \tilde T_k = \sum_{a \in \widehat B_k}
a T_{kj}$ and $$b \tilde T_k
= \sum_{a \in \widehat B_k}
ba T_{kj} = \sum_{a \in \widehat B_k}
a (bT_{kj})=0.$$ We get a contradiction with faithfulness
of $\tilde T_{k}$.
\end{proof}

 By Condition~2 of the definition of $d$ (see Subsection~\ref{SubsectionHopfPIexp}),
 there exist numbers $q_1, \ldots, q_{r} \in \mathbb Z_+$
such that
$$[[\tilde T_1, \underbrace{L, \ldots, L}_{q_1}], [\tilde T_2, \underbrace{L, \ldots, L}_{q_2}] \ldots, [\tilde T_r,
 \underbrace{L, \ldots, L}_{q_r}]] \ne 0$$
 Choose $n_i \in \mathbb Z_+$ with the maximal $\sum\limits_{i=1}^r n_i$ such that
$$[[\left(\prod_{k=1}^{n_1} j_{1k}\right)\tilde T_1, \underbrace{L, \ldots, L}_{q_1}],
 [\left(\prod_{k=1}^{n_2} j_{2k}\right) \tilde T_2, \underbrace{L, \ldots, L}_{q_2}] \ldots, [\left(\prod_{k=1}^{n_r} j_{rk}\right) \tilde T_r,
 \underbrace{L, \ldots, L}_{q_r}]] \ne 0
$$ for some $j_{ik}\in J(A)$.
Let $j_i := \prod_{k=1}^{n_i} j_{ik}$.
Then $j_i  \in J(A) \cup \{1\}$ and
$$[[j_1 \tilde T_1, \underbrace{L, \ldots, L}_{q_1}], [j_2 \tilde T_2, \underbrace{L, \ldots, L}_{q_2}], \ldots, [j_r \tilde T_r,
 \underbrace{L, \ldots, L}_{q_r}]] \ne 0,
 $$ but
\begin{equation}\label{EquationJZero}
[[j_1 \tilde T_1, \underbrace{L, \ldots, L}_{q_1}],
\ldots, [j_k (j \tilde T_k), \underbrace{L, \ldots, L}_{q_k}], \ldots, [j_r \tilde T_r,
 \underbrace{L, \ldots, L}_{q_r}]] = 0
\end{equation}
for all $j \in J(A)$ and $1 \leqslant k \leqslant r$.

   In virtue of Lemma~\ref{ChooseSubmodule}, for every $k$ we can choose a
   faithful irreducible
 $(H,(\ad B_k)\oplus R_k)$-submodule
$T_k \subseteq \tilde T_k$
such that\begin{equation}\label{EquationqNonZero}
[[j_1 T_1, \underbrace{L, \ldots, L}_{q_1}], [j_2 T_2, \underbrace{L, \ldots, L}_{q_2}] \ldots, [j_r T_r, \underbrace{L, \ldots, L}_{q_r}]] \ne 0.
\end{equation}

\begin{lemma}\label{LemmaChange}
Let $\psi \colon \bigoplus_{i=1}^r(B_i \oplus \tilde R_i) \to
\bigoplus_{i=1}^r((\ad B_i)\oplus R_i) $
be the linear isomorphism defined by formulas $\psi(b)= \ad b$ for
all $b \in B_i$ and $\psi(a_{i\ell})=c_{i\ell}$, $1 \leqslant \ell
\leqslant k_\ell$.
Let $f_i$ be multilinear associative $H$-polynomials,
$\bar x^{(i)}_1, \ldots, \bar x^{(i)}_{n_i}
\in \bigoplus_{i=1}^r B_i \oplus \tilde R_i$, $\bar t_i \in \tilde T_i$, $\bar u_{ik}\in L$,
 be some elements.
Then
$$[[j_1 f_1(\ad \bar x^{(1)}_1, \ldots, \ad \bar x^{(1)}_{n_1})
 \bar t_1, \bar u_{11}, \ldots, \bar u_{1q_1}],  \ldots, [j_r
 f_r(\ad \bar x^{(r)}_1, \ldots, \ad \bar x^{(r)}_{n_r}) \bar t_r,
 \bar u_{r1}, \ldots, \bar u_{rq_r}]]=$$ $$
 [[j_1 f_1(\psi (\bar x^{(1)}_1), \ldots, \psi (\bar x^{(1)}_{n_1}))
 \bar t_1, \bar u_{11}, \ldots, \bar u_{1q_1}],  \ldots,$$ $$ [j_r
 f_r(\psi(\bar x^{(r)}_1), \ldots, \psi (\bar x^{(r)}_{n_r})) \bar t_r,
 \bar u_{r1}, \ldots, \bar u_{rq_r}]].$$ In other words, we can replace $\ad a_{i\ell}$ with $c_{i\ell}$ and the result does not change.
\end{lemma}
\begin{proof}
We rewrite $\ad a_{i\ell}=c_{i\ell}+d_{i\ell}=\psi(a_i)+d_{i\ell}$ and use the multilinearity
of $f_i$. By~(\ref{EquationJZero}), terms with $d_{i\ell}$ vanish.
\end{proof}

\begin{lemma}\label{LemmaAlt} If $d \ne 0$, then there exists a number $n_0 \in \mathbb N$ such that for every $n\geqslant n_0$
there exist disjoint subsets $X_1$, \ldots, $X_{2k} \subseteq \lbrace x_1, \ldots, x_n
\rbrace$, $k := \left[\frac{n-n_0}{2d}\right]$,
$|X_1| = \ldots = |X_{2k}|=d$ and a polynomial $f \in V^H_n \backslash
\Id^H(L)$ alternating in the variables of each set $X_j$.
\end{lemma}

\begin{proof}
Denote by $\varphi_i \colon (\ad B_i)\oplus R_i \to
\mathfrak{gl}(T_i)$ the representation
 corresponding to the action of $(\ad B_i)\oplus R_i$
on $T_i$.
In virtue of Theorem~\ref{TheoremAlternateFinal},
there exist constants $m_i \in \mathbb Z_+$
such that for any $k$ there exist
 multilinear polynomials $f_i \in Q^H_{d_i, 2k, 2k d_i+m_i}
  \backslash \Id^H(\varphi_i)$,
$d_i := \dim ((\ad B_i)\oplus R_i)$,
alternating in the variables from disjoint sets
$X^{(i)}_{\ell}$, $1 \leqslant \ell \leqslant 2k$, $|X^{(i)}_{\ell}|=d_i$.

In virtue of~(\ref{EquationqNonZero}),
$$[[j_1 \bar t_1, \bar u_{11}, \ldots, \bar u_{1,q_1}], [j_2 \bar t_2, \bar u_{21}, \ldots, \bar u_{2,q_2}],
 \ldots, [j_r \bar t_r, \bar u_{r1}, \ldots, \bar u_{r,q_r}]] \ne 0,
 $$
 for some $\bar u_{i\ell} \in L$ and $\bar t_i \in T_i$. All $j_i \in J(A)\cup \{1\}$ are polynomials in elements from $\zeta(H)$ and $\ad L$. Denote by $\tilde m$ the maximal degree of them.

 Recall that each $T_i$ is a faithful irreducible $(H,(\ad B_i)\oplus R_i)$-module.
Therefore by the Density Theorem,
 $\End_F(T_i)$ is generated by operators from $\zeta(H)$
and $(\ad B_i)\oplus R_i$.
Note that $\End_F(T_i) \cong M_{\dim T_i}(F)$.
Thus every matrix unit $e^{(i)}_{j\ell} \in M_{\dim T_i}(F)$ can be
represented as a polynomial in operators from $\zeta(H)$
and $(\ad B_i)\oplus R_i$. Choose such polynomials
for all $i$ and all matrix units. Denote by $m_0$ the maximal degree of those
polynomials.

Let $n_0 := r(2m_0+\tilde m+1)+ \sum_{i=1}^r (m_i+q_i)$.
Now we choose $f_i$ for $k = \left[\frac{n-n_0}{2d}\right]$.
In addition, we choose $\tilde f_1$ for $\tilde k = \left[\frac{n-2kd-m_1}{2d_1}\right]+1$
and $\varphi_1$. The polynomials $f_i$ will deliver us the required alternations.
However, the total degree of the product may be less than $n$. We will use $\tilde f_1$
to increase the number of variables and obtain a polynomial of degree $n$.

Since $f_i \notin \Id^H(\varphi_i)$ and $\tilde f_1 \notin \Id^H(\varphi_1)$,
there exist $\bar x_{i1}, \ldots, \bar x_{i, 2k d_i+m_i} \in (\ad B_i)\oplus R_i$
such that $f_i(\bar x_{i1}, \ldots, \bar x_{i, 2k d_i+m_i})\ne 0$,
and $\bar x_1, \ldots, \bar x_{2\tilde k d_1+m_1} \in (\ad B_1)\oplus R_1$
such that $\tilde f_1(\bar x_1, \ldots, \bar x_{2\tilde k d_1+m_1}) \ne 0$.
Hence $$e^{(i)}_{\ell_i \ell_i} f_i(\bar x_{i1}, \ldots, \bar x_{i, 2k d_i+m_i})
e^{(i)}_{s_i s_i} \ne 0$$
and $$e^{(1)}_{\tilde\ell \tilde\ell}\tilde f_1(\bar x_1, \ldots, \bar x_{2\tilde k d_1+m_1})
e^{(1)}_{\tilde s \tilde s} \ne 0$$
 for some matrix units $e^{(i)}_{\ell_i \ell_i},
e^{(i)}_{s_i s_i} \in \End_F(T_i)$, $1 \leqslant \ell_i, s_i \leqslant \dim {T_i}$,
$e^{(1)}_{\tilde\ell \tilde\ell}, e^{(1)}_{\tilde s \tilde s} \in \End_F(T_1)$, $1 \leqslant \tilde \ell,
\tilde s \leqslant \dim T_1$.
Thus $$\sum_{\ell=1}^{\dim_{T_i}}
e^{(i)}_{\ell \ell_i} f_i(\bar x_{i1}, \ldots, \bar x_{i, 2k d_i+m_i})
 e^{(i)}_{s_i \ell}$$ is a nonzero scalar operator in $\End_F(T_i)$.

Hence
$$ [[j_1\left(\sum_{\ell=1}^{\dim {T_1}}
e^{(1)}_{\ell \ell_1} f_1(\bar x_{11}, \ldots, \bar x_{1,2k d_1+m_1})
e^{(1)}_{s_1 \tilde \ell} \tilde f_1(\bar x_1, \ldots, \bar x_{2\tilde k d_1+m_1})
 e^{(1)}_{\tilde s \ell}\right)\bar t_1, \bar u_{11}, \ldots, \bar u_{1q_1}],$$
 $$ [j_2\left(\sum_{\ell=1}^{\dim {T_2}}
e^{(2)}_{\ell \ell_2} f_2(\bar x_{21}, \ldots, \bar x_{2,2k d_2+m_2})
 e^{(2)}_{s_2 \ell}\right)\bar t_2, \bar u_{21}, \ldots, \bar u_{2q_2}],
 \ldots, $$
 $$
 [j_r\left(\sum_{\ell=1}^{\dim {T_r}}
e^{(r)}_{\ell \ell_r} f_r(\bar x_{r1}, \ldots, \bar x_{r, 2k d_r+m_r})
 e^{(r)}_{s_r \ell}\right)\bar t_r, \bar u_{r1}, \ldots, \bar u_{rq_r}]]\ne 0.$$
 We assume that each $f_i$ is a polynomial in $x_{i1}, \ldots,
x_{i,2k d_i+m_i}$ and $\tilde f_1$ is a polynomial in $x_1, \ldots, x_{2\tilde k d_1 + m_1}$.
Denote $X_\ell := \bigcup_{i=1}^{r} X^{(i)}_{\ell}$
where $f_i$ is alternating in the variables of each $X^{(i)}_{\ell}$.
Let $\Alt_\ell$ be the operator of alternation
in the variables from $X_\ell$. 
Consider
$$\tilde f(x_1, \ldots, x_{2\tilde k d_1 + m_1}; x_{11}, \ldots, x_{1, 2k d_1+m_1};
\ldots; \ x_{r1}, \ldots, x_{r, 2k d_r+m_r}) :=
$$ $$
 \Alt_1 \Alt_2 \ldots \Alt_{2k} [[j_1\left(\sum_{\ell=1}^{\dim {T_1}}
e^{(1)}_{\ell \ell_1} f_1(x_{11}, \ldots, x_{1, 2k d_1+m_1})
e^{(1)}_{s_1 \tilde \ell}\ \cdot \right.$$
 $$\left.\tilde f_1(x_1, \ldots, x_{2\tilde k d_1+m_1})
 e^{(1)}_{\tilde s \ell}\right)\bar t_1, \bar u_{11}, \ldots, \bar u_{1q_1}],$$
 $$ [j_2\left(\sum_{\ell=1}^{\dim {T_2}}
e^{(2)}_{\ell \ell_2} f_2(x_{21}, \ldots, x_{2, 2k d_2+m_2})
 e^{(2)}_{s_2 \ell}\right)\bar t_2, \bar u_{21}, \ldots, \bar u_{2q_2}],
 \ldots, $$
 $$
 [j_r\left(\sum_{\ell=1}^{\dim {T_r}}
e^{(r)}_{\ell \ell_r} f_r(x_{r1}, \ldots, x_{r,2k d_r+m_r})
 e^{(r)}_{s_r \ell}\right)\bar t_r, \bar u_{r1}, \ldots, \bar u_{rq_r}]].$$
Then
$$\tilde f(\bar x_1, \ldots, \bar x_{2\tilde k d_1 + m_1}; \bar x_{11}, \ldots, \bar x_{1, 2k d_1+m_1};
\ldots; \ \bar x_{r1}, \ldots, \bar x_{r, 2k d_r+m_r})
= $$ $$(d_1!)^{2k} \ldots (d_r!)^{2k} [[j_1\left(\sum_{\ell=1}^{\dim {T_1}}
e^{(1)}_{\ell \ell_1} f_1(\bar x_{11}, \ldots, \bar x_{1,2k d_1+m_1})
e^{(1)}_{s_1 \tilde \ell}\ \cdot\right.$$ $$\left. \tilde f_1(\bar x_1, \ldots, \bar x_{2\tilde k d_1+m_1})
 e^{(1)}_{\tilde s \ell}\right)\bar t_1, \bar u_{11}, \ldots, \bar u_{1q_1}],
 \ldots, $$
 $$
 [j_r\left(\sum_{\ell=1}^{\dim {T_r}}
e^{(r)}_{\ell \ell_r} f_r(\bar x_{r1}, \ldots, \bar x_{r, 2k d_r+m_r})
 e^{(r)}_{s_r \ell}\right)\bar t_r, \bar u_{r1},
  \ldots, \bar u_{rq_r}]]\ne 0.$$
since $f_i$ are alternating in each $X^{(i)}_{\ell}$
and, by Lemma~\ref{LemmaSiProperties}, $((\ad B_i)\oplus R_i)\tilde T_\ell = 0$
for $i > \ell$. Now we rewrite
$e^{(i)}_{\ell j}$ as polynomials in elements of $(\ad B_i)\oplus R_i$
and $\zeta(H)$.
Using linearity of $\tilde f$ in $e^{(i)}_{\ell j}$,
we can replace $e^{(i)}_{\ell j}$ with the products
of elements from $(\ad B_i)\oplus R_i$
and $\zeta(H)$, and the expression will not vanish
for some choice of the products. Using~(\ref{EqHLModule2}),
we can move all $\zeta(h)$ to the right.
 By Lemma~\ref{LemmaChange},
we can replace all elements from $(\ad B_i)\oplus R_i$
with elements from $B_i\oplus \tilde R_i$
and the expression will be still nonzero.
Denote by $\psi \colon \bigoplus_{i=1}^r (B_i \oplus \tilde R_i) \to
\bigoplus_{i=1}^r ((\ad B_i)\oplus R_i) $ the corresponding linear isomorphism.
Now we rewrite $j_i$ as polynomials in elements $\ad L$ and $\zeta(H)$.
Since $\tilde f$ is linear in $j_i$,
we can replace $j_i$ with one of the monomials,
i.e. with the product of elements from $\ad L$ and $\zeta(H)$.
Using~(\ref{EqHLModule2}),
we again move all $\zeta(h)$ to the right. Then
we replace the elements from $\ad L$ with new variables,
and
$$\hat f :=
 \Alt_1 \Alt_2 \ldots \Alt_{2k} \biggl[\Bigl[\Bigl[y_{11}, [y_{12}, \ldots
 [y_{1 \alpha_1}, \Bigl[z_{11}, [z_{12},
 \ldots, [z_{1 \beta_1},
$$ $$
  (f_1(\ad x_{11}, \ldots, \ad x_{1, 2k d_1+m_1}))^{h_1}
 [w_{11}, [w_{12}, \ldots, [w_{1 \theta_1},$$ $$
 (\tilde f_1(\ad x_1, \ldots, \ad x_{2\tilde k d_1+m_1}))^{\tilde h}
 [w_{1}, [w_{2}, \ldots, [w_{\tilde \theta},
  t_1^{h'_1}]\ldots \Bigr],
  u_{11}, \ldots, u_{1q_1}\Bigr],$$
  $$\Bigl[\Bigl[y_{21}, [y_{22}, \ldots
 [y_{2 \alpha_2}, \Bigl[z_{21}, [z_{22},
 \ldots, [z_{2 \beta_2},
$$ $$
  (f_2(\ad x_{21}, \ldots, \ad x_{2, 2k d_2+m_2}))^{h_2}
 [w_{21}, [w_{22}, \ldots, [w_{2 \theta_2},
  t_2^{h'_2}]\ldots \Bigr],
  u_{21}, \ldots, u_{2q_2}\Bigr],
 \ldots, $$
 $$\Bigl[\Bigl[y_{r1}, [y_{r2}, \ldots,
 [y_{r \alpha_r}, \Bigr[z_{r1},
  [z_{r2},
 \ldots, [z_{r \beta_r},
 $$ $$
 (f_r(\ad x_{r1}, \ldots, \ad x_{r, 2k d_r+m_r}))^{h_r}
 [w_{r1}, [w_{r2}, \ldots, [w_{r \theta_r}, t_r^{h'_r}]\ldots \Bigr],
  u_{r1}, \ldots, u_{rq_r}\Bigr]\biggr]$$
  for some  $0 \leqslant \alpha_i \leqslant \tilde m$,
  \quad
  $0 \leqslant \beta_i, \theta_i, \tilde \theta \leqslant m_0$,
  \quad $h_i, h'_i, \tilde h \in H$,\quad
  $\bar y_{i\ell}, \bar z_{i\ell},
  \bar w_{i\ell}, \bar w_i \in L$
 does not vanish under the substitution
 $t_i=\bar t_i$, $u_{i\ell}=\bar u_{i\ell}$,
 $x_{i\ell}=\psi^{-1}(\bar x_{i\ell})$, $x_i = \psi^{-1}(\bar x_i)$, $y_{i\ell}=\bar y_{i\ell}$,
 $z_{i\ell}=\bar z_{i\ell}$, $w_{i\ell}=\bar w_{i\ell}$, $w_i = \bar w_i$.
 
 Hence $$f_0 :=
 \Alt_1 \Alt_2 \ldots \Alt_{2k} \biggl[\Bigl[\Bigl[y_{11}, [y_{12}, \ldots
 [y_{1 \alpha_1}, \Bigl[z_{11}, [z_{12},
 \ldots, [z_{1 \beta_1},
$$ $$
  (f_1(\ad x_{11}, \ldots, \ad x_{1, 2k d_1+m_1}))^{h_1}
 [w_{11}, [w_{12}, \ldots, [w_{1 \theta_1}, t_1]\ldots \Bigr],
  u_{11}, \ldots, u_{1q_1}\Bigr],$$
  $$\Bigl[\Bigl[y_{21}, [y_{22}, \ldots
 [y_{2 \alpha_2}, \Bigl[z_{21}, [z_{22},
 \ldots, [z_{2 \beta_2},
$$ $$
  (f_2(\ad x_{21}, \ldots, \ad x_{2, 2k d_2+m_2}))^{h_2}
 [w_{21}, [w_{22}, \ldots, [w_{2 \theta_2},
  t_2^{h'_2}]\ldots \Bigr],
  u_{21}, \ldots, u_{2q_2}\Bigr],
 \ldots, $$
 $$\Bigl[\Bigl[y_{r1}, [y_{r2}, \ldots,
 [y_{r \alpha_r}, \Bigr[z_{r1},
  [z_{r2},
 \ldots, [z_{r \beta_r},
 $$ $$
 (f_r(\ad x_{r1}, \ldots, \ad x_{r, 2k d_r+m_r}))^{h_r}
 [w_{r1}, [w_{r2}, \ldots, [w_{r \theta_r}, t_r^{h'_r}]\ldots \Bigr],
  u_{r1}, \ldots, u_{rq_r}\Bigr]\biggr]$$
   does not vanish under the substitution
 $$t_1 = (\tilde f_1(\ad \bar x_1, \ldots, \ad \bar x_{2\tilde k d_1+m_1}))^{\tilde h}
 [\bar w_{1}, [\bar w_{2}, \ldots, [\bar w_{\tilde \theta},
 h'_1 \bar t_1]\ldots],$$
 $t_i=\bar t_i$ for $2 \leqslant i \leqslant r$; $u_{i\ell}=\bar u_{i\ell}$,
 $x_{i\ell}=\psi^{-1}(\bar x_{i\ell})$, $y_{i\ell}=\bar y_{i\ell}$,
 $z_{i\ell}=\bar z_{i\ell}$, $w_{i\ell}=\bar w_{i\ell}$.

Note that $f_0 \in V_{\tilde n}^H$,
  $\tilde n: = 2kd +r+ \sum_{i=1}^r (m_i + q_i + \alpha_i+\beta_i+\theta_i)
  \leqslant n$. If $n=\tilde n$, then we take $f:=f_0$.
  Suppose $n > \tilde n$.
Note that $(\tilde f_1(\ad \bar x_1, \ldots, \ad \bar x_{2\tilde k d_1+m_1}))^{\tilde h}
 [\bar w_{1}, [\bar w_{2}, \ldots, [\bar w_{\tilde \theta},
  h'_1 \bar t_1]\ldots]$ is a linear combination of long commutators.
  Each of these commutators contains at least $2\tilde k d_1+m_1+1 > n-\tilde n+1$
  elements of $L$.
       Hence $ f_0$ does not vanish under a substitution
 $t_1 = [\bar v_1, [\bar v_2, [\ldots, [\bar v_q,  h'_1\bar t_1]\ldots]$
 for some $q \geqslant n-\tilde n$, $\bar v_i \in L$;
  $t_i=\bar t_i$ for $2 \leqslant i \leqslant r$; $u_{i\ell}=\bar u_{i\ell}$,
 $x_{i\ell}=\psi^{-1}(\bar x_{i\ell})$, $y_{i\ell}=\bar y_{i\ell}$,
 $z_{i\ell}=\bar z_{i\ell}$, $w_{i\ell}=\bar w_{i\ell}$.
Therefore, $$f :=
 \Alt_1 \Alt_2 \ldots \Alt_{2k} \biggl[\Bigl[\Bigl[y_{11}, [y_{12}, \ldots
 [y_{1 \alpha_1}, \Bigl[z_{11}, [z_{12},
 \ldots, [z_{1 \beta_1},
$$ $$
  (f_1(\ad x_{11}, \ldots, \ad x_{1, 2k d_1+m_1}))^{h_1}
 [w_{11}, [w_{12}, \ldots, [w_{1 \theta_1},
  $$
 $$\bigl[v_1, [v_2, [\ldots, [v_{n-\tilde n}, t_1]\ldots\bigr]\ldots \Bigr],
  u_{11}, \ldots, u_{1q_1}\Bigr],$$
 $$
  \Bigl[\Bigl[y_{21}, [y_{22}, \ldots
 [y_{2 \alpha_2}, \Bigl[z_{21}, [z_{22},
 \ldots, [z_{2 \beta_2},
$$ $$
  (f_2(\ad x_{21}, \ldots, \ad x_{2, 2k d_2+m_2}))^{h_2}
 [w_{21}, [w_{22}, \ldots, [w_{2 \theta_2},
  t_2^{h'_2}]\ldots \Bigr],
  u_{21}, \ldots, u_{2q_2}\Bigr],  $$
$$
 \ldots, \Bigl[\Bigl[y_{r1}, [y_{r2},\ldots,
 [y_{r \alpha_r}, \Bigr[z_{r1},
  [z_{r2},
 \ldots, [z_{r \beta_r},
 $$ $$
 (f_r(\ad x_{r1}, \ldots, \ad x_{r, 2k d_r+m_r}))^{h_r}
 [w_{r1}, [w_{r2}, \ldots, [w_{r \theta_r}, t_r^{h'_r}]\ldots \Bigr],
  u_{r1}, \ldots, u_{rq_r}\Bigr]\biggr]$$
  does not vanish under the substitution
  $v_\ell = \bar v_\ell$, $1 \leqslant \ell \leqslant n-\tilde n$,
  $$t_1 = [\bar v_{n-\tilde n +1}, [\bar v_{n-\tilde n +2}, [\ldots, [\bar v_q,  h'_1\bar t_1]\ldots];$$
  $t_i=\bar t_i$ for $2 \leqslant i \leqslant r$; $u_{i\ell}=\bar u_{i\ell}$,
 $x_{i\ell}=\psi^{-1}(\bar x_{i\ell})$, $y_{i\ell}=\bar y_{i\ell}$,
 $z_{i\ell}=\bar z_{i\ell}$, $w_{i\ell}=\bar w_{i\ell}$.
 Note that $f \in V_n^H$ and satisfies all the conditions of the lemma.
\end{proof}

\begin{lemma}\label{LemmaCochar} Let
 $k, n_0$ be the numbers from
Lemma~\ref{LemmaAlt}.   Then for every $n \geqslant n_0$ there exists
a partition $\lambda = (\lambda_1, \ldots, \lambda_s) \vdash n$,
$\lambda_i \geqslant 2k-C$ for every $1 \leqslant i \leqslant d$,
with $m(L, H, \lambda) \ne 0$.
Here $C := p((\dim L)p + 3)((\dim L)-d)$ where $p \in \mathbb N$ is such number that $N^p=0$.
\end{lemma}
\begin{proof}
Consider the polynomial $f$ from Lemma~\ref{LemmaAlt}.
It is sufficient to prove that $e^*_{T_\lambda} f \notin \Id^H(L)$
for some tableau $T_\lambda$ of the desired shape $\lambda$.
It is known that $FS_n = \bigoplus_{\lambda,T_\lambda} FS_n e^{*}_{T_\lambda}$ where the summation
runs over the set of all standard tableax $T_\lambda$,
$\lambda \vdash n$. Thus $FS_n f = \sum_{\lambda,T_\lambda} FS_n e^{*}_{T_\lambda}f
\not\subseteq \Id^H(L)$ and $e^{*}_{T_\lambda} f \notin \Id^H(L)$ for some $\lambda \vdash n$.
We claim that $\lambda$ is of the desired shape.
It is sufficient to prove that
$\lambda_d \geqslant 2k-C$, since
$\lambda_i \geqslant \lambda_d$ for every $1 \leqslant i \leqslant d$.
Each row of $T_\lambda$ includes numbers
of no more than one variable from each $X_i$,
since $e^{*}_{T_\lambda} = b_{T_\lambda} a_{T_\lambda}$
and $a_{T_\lambda}$ is symmetrizing the variables of each row.
Thus $\sum_{i=1}^{d-1} \lambda_i \leqslant 2k(d-1) + (n-2kd) = n-2k$.
In virtue of Lemma~\ref{LemmaUpper},
$\sum_{i=1}^d \lambda_i \geqslant n-C$. Therefore
$\lambda_d \geqslant 2k-C$.
\end{proof}

\begin{proof}[Proof of Theorem~\ref{TheoremMainH}]
The Young diagram~$D_\lambda$ from Lemma~\ref{LemmaCochar} contains
the rectangular subdiagram~$D_\mu$, $\mu=(\underbrace{2k-C, \ldots, 2k-C}_d)$.
The branching rule for $S_n$ implies that if we consider the restriction of
$S_n$-action on $M(\lambda)$ to $S_{n-1}$, then
$M(\lambda)$ becomes the direct sum of all non-isomorphic
$FS_{n-1}$-modules $M(\nu)$, $\nu \vdash (n-1)$, where each $D_\nu$ is obtained
from $D_\lambda$ by deleting one box. In particular,
$\dim M(\nu) \leqslant \dim M(\lambda)$.
Applying the rule $(n-d(2k-C))$ times, we obtain $\dim M(\mu) \leqslant \dim M(\lambda)$.
By the hook formula, $$\dim M(\mu) = \frac{(d(2k-C))!}{\prod_{i,j} h_{ij}}$$
where $h_{ij}$ is the length of the hook with edge in $(i, j)$.
By Stirling formula,
$$c_n^H(L)\geqslant \dim M(\lambda) \geqslant \dim M(\mu) \geqslant \frac{(d(2k-C))!}{((2k-C+d)!)^d}
\sim $$ $$\frac{
\sqrt{2\pi d(2k-C)} \left(\frac{d(2k-C)}{e}\right)^{d(2k-C)}
}
{
\left(\sqrt{2\pi (2k-C+d)}
\left(\frac{2k-C+d}{e}\right)^{2k-C+d}\right)^d
} \sim C_6 k^{r_6} d^{2kd}$$
for some constants $C_6 > 0$, $r_6 \in \mathbb Q$,
as $k \to \infty$.
Since $k = \left[\frac{n-n_0}{2d}\right]$,
this gives the lower bound.
The upper bound has been proved in Theorem~\ref{TheoremUpper}.
\end{proof}

\begin{proof}[Proof of Theorem~\ref{TheoremMainHSum}]
Suppose $L = L_1 \oplus \ldots \oplus L_q$ where $L_i$ are $H$-nice ideals.
First, $c^H_n(L) \geqslant c^H_n(L_i)$ for all $n\in\mathbb N$ and $1 \leqslant i \leqslant q$
since $L_i$ are $H$-invariant subalgebras of $L$. 
Hence $$\max_{1 \leqslant i \leqslant q} \PIexp^H(L_i) \leqslant \mathop{\underline{\lim}}_{n\to \infty}\sqrt[n]{c^H_n(L)}.$$

Suppose $f_0 \in V_n^H$, $n\in\mathbb N$. In order to prove that $f_0 \in \Id^H(L)$, it is sufficient to substitute only basis elements.
Choose a basis in $L$ that is the union of bases in $L_i$. Then if we substitute elements
from different $L_i$, the polynomial $f_0$ vanishes.
Hence it is sufficient to prove that $f_0 \in \Id^H(L_i)$
for all $1\leqslant i \leqslant s$. 
Let $d := \max_{1 \leqslant i \leqslant q} d(L_i)=\max_{1 \leqslant i \leqslant q} \PIexp^H(L_i)$. Then
Lemma~\ref{LemmaUpper} implies that 
if $\lambda = (\lambda_1, \ldots, \lambda_s) \vdash n$
and $\lambda_{d+1} \geqslant p((\dim L)p+3)$ or $\lambda_{\dim L+1} > 0$, then
$m(L, H, \lambda) = 0$. Repeating the arguments of Theorem~\ref{TheoremUpper}, we
apply Theorem~\ref{TheoremColength} and obtain that there exist constants $C_2 > 0$, $r_2 \in \mathbb R$
such that $c^H_n(L) \leqslant C_2 n^{r_2} d^n$ for all $n \in \mathbb N$.
Hence $$\mathop{\overline{\lim}}_{n\to \infty} \sqrt[n]{c^H_n(L)} \leqslant d = \max_{1 \leqslant i \leqslant q} \PIexp^H(L_i).$$
The lower bound has been already obtained.
\end{proof}

\section{Applications}\label{SectionAppl}

In this section we derive Theorems~\ref{TheoremMainGr}, \ref{TheoremMainGAffAlg}, and \ref{TheoremMainGFin} from Theorem~\ref{TheoremMainH} and 
Theorems~\ref{TheoremMainGrSum}, \ref{TheoremMainGAffAlgSum}, and \ref{TheoremMainGFinSum} from Theorem~\ref{TheoremMainHSum}.

\subsection{Applications to graded codimensions}\label{SubsectionApplGr}
\subsubsection{Gradings by finite groups}
In the case of an infinite group we will use a trick (see Lemma~\ref{LemmaInclExcl} below) to pass from an arbitrary group to a finitely generated
Abelian one. Unfortunately, this trick makes it impossible to use the explicit formula for the Hopf PI-exponent. However, in the case 
when the group is finite, we can avoid this trick and prove Lemma~\ref{LemmaGradAction} that enables us to derive properties of graded codimensions from properties of $H$-codimensions directly and, in particular, to calculate the graded PI-exponent using Subsection~\ref{SubsectionHopfPIexp}.

Let $L$ be a Lie algebra over a field $F$. Suppose $L$ is a right $H$-comodule for some Hopf algebra
$H$. Denote by $\rho \colon L \to L \otimes H$ the corresponding comodule map. We say that $L$ is an
\textit{$H$-comodule Lie algebra} if $\rho([a,b])=[a_{(0)},b_{(0)}] \otimes a_{(1)}b_{(1)}$
for all $a,b \in L$. Here we use
Sweedler's notation $\rho(a)=a_{(0)}\otimes a_{(1)}$.

\begin{example}\label{ExampleGraded}   
If $L=\bigoplus_{g \in G} L^{(g)}$ is a Lie
algebra over a field $F$ graded by a group $G$, then $L$ is an $FG$-comodule algebra where
$\rho(a^{(g)})=a^{(g)} \otimes g$ for all $g \in G$ and $a^{(g)} \in L^{(g)}$. 
Conversely, each $FG$-comodule Lie algebra $L$ has the $G$-grading
 $L=\bigoplus_{g\in G} L^{(g)}$ where $$L^{(g)}=\lbrace a \in L
 \mid \rho(a)  = a \otimes g\rbrace.$$
\end{example}

If $H$ is finite dimensional, then every $H$-comodule Lie algebra becomes an $H^*$-module
Lie algebra where $H^*:= \Hom_F(H, F)$ is the Hopf algebra dual to $H$ and $h^* a = h^*(a_{(1)}) a_{(0)}$, $h^* \in H^*$, $a \in L$.
In particular, if $G$
is finite, a $G$-graded Lie algebra $L$ is
an $(FG)^*$-module Lie algebra.
 Conversely, each $(FG)^*$-module Lie algebra $L$ has the $G$-grading
 $L=\bigoplus_{g\in G} L^{(g)}$ where $$L^{(g)}=\lbrace a \in L
 \mid h  a  = h(g) a \text{ for all } h \in (FG)^*\rbrace.$$
 Furthermore, any homomorphism of graded algebras is a homomorphism
 of $(FG)^*$-module algebras and vice versa.
  
 Let $(h_g)_{g\in G}$ be the basis in $(FG)^*$ dual
 to the basis $(g)_{g\in G}$ of $FG$, i.e. \begin{equation}\label{Eqhg1h2}
 h_{g_1}(g_2)=\left\lbrace
  \begin{array}{ll}
  1, & g_1 = g_2, \\
  0, & g_1 \ne g_2.
  \end{array} \right.
  \end{equation}
  Note that $$h_{g_1}h_{g_2}=\left\lbrace
  \begin{array}{ll}
  h_{g_1}, & g_1 = g_2, \\
  0, & g_1 \ne g_2,
  \end{array} \right.$$ i.e. $(FG)^*$ is isomorphic as an algebra to the direct sum of copies of $F$.
  
  We treat $L( X | (FG)^* )$ and $L( X^{\mathrm{gr}})$ as both graded and $(FG)^*$-module algebras.
 The homomorphism $\varphi \colon L( X | (FG)^* )
 \to L( X^{\mathrm{gr}})$ of $(FG)^*$-module algebras defined
 by $\varphi(x_j)=\sum_{g\in G} x^{(g)}_j$, $h\in (FG)^*$, $j\in \mathbb N$,
 is an isomorphism since the homomorphism $\xi \colon L( X^{\mathrm{gr}} )
 \to L( X | (FG)^* )$ of graded algebras defined by $\xi(x^{(g)}_j)
 = x^{h_g}_j$, $g\in G$, $j \in \mathbb N$, is the inverse of $\varphi$.

Indeed, $$\varphi(\xi(x^{(g)}_j))=\varphi(x^{h_g}_j)=h_g\varphi(x_j)=\sum_{g_0 \in G}h_g(g_0) x^{(g_0)}_j
=x^{(g)}_j$$
and $$\xi(\varphi(x_j))=\sum_{g\in G} \xi(x^{(g)}_j)
=\sum_{g\in G} x^{h_g}_j=x^{\sum_{g\in G} h_g}_j=x_j.$$

\begin{lemma}\label{LemmaGradAction}
Let $L$ be a $G$-graded Lie algebra where $G$ is a finite group.
Consider the corresponding $(FG)^*$-action on $L$. Then
 \begin{enumerate}
\item $\varphi\left(\Id^{(FG)^*}(L)\right)=\Id^{\mathrm{gr}}(L)$;
\item $c^{(FG)^*}_n(L)=c^{\mathrm{gr}}_n(L)$.
\end{enumerate}
\end{lemma}
\begin{proof} Let $f\in \Id^{(FG)^*}(L)$. Suppose $\psi \colon L(X^{\mathrm{gr}}) \to L$
is a homomorphism of graded algebras.
Then $\psi\varphi \colon L(X|{(FG)^*}) \to L$ is a homomorphism of ${(FG)^*}$-module algebras.
Hence $\psi(\varphi(f))=0$ and $\varphi(f)\in \Id^{\mathrm{gr}}(L)$.

Conversely, let $f\in \Id^{\mathrm{gr}}(L)$. Suppose $\psi\colon L(X|{(FG)^*}) \to L$
is a homomorphism of ${(FG)^*}$-module algebras.
Then $\psi\varphi^{-1} \colon L(X^{\mathrm{gr}}) \to L$
is a homomorphism of graded algebras.
 Therefore $\psi(\varphi^{-1}(f))= 0$ and $\varphi^{-1}(f) \in 
\Id^{(FG)^*}(L)$. The first assertion is proved.

 The second assertion follows
from the first one and the equality
$\varphi\left(V^{(FG)^*}_n\right)=V^{\mathrm{gr}}_n$.
\end{proof}

 \begin{remark}
  The $H$-action and the polynomial $H$-identity from Example~\ref{ExampleIdH}
 correspond to the grading and the graded polynomial identity from Example~\ref{ExampleIdGr}.
 \end{remark}

Now we can easily derive Theorem~\ref{TheoremMainGr}, in the case when $G$ is finite, from
Theorem~\ref{TheoremMainHSS}, Lemma~\ref{LemmaGradAction}, and the fact that $(FG)^*$ is isomoriphic as an algebra to the direct sum of fields.

\subsubsection{Gradings by arbitrary groups}
The case when $G$ is infinite is treated using a similar duality. Although Lemma~\ref{LemmaAbelianDual} is known, we sketch the proof for the reader's convenience.

\begin{lemma}\label{LemmaAbelianDual}
Let $F$ be an algebraically closed field of characteristic $0$, let $G$ be a finitely generated Abelian group, and let $\hat G = \Hom(G, F^\times)$ be the group of homomorphisms from $G$ into the multiplicative
group $F^\times$ of the field $F$. Consider the elements of $F\hat G$ as functions on $G$.
Then for any pairwise distinct $\gamma_1, \ldots, \gamma_m \in G$ there exist $h_1, \ldots, h_m \in F\hat G$
that $h_i(\gamma_j)=\left\lbrace
  \begin{array}{ll}
  1, & i = j, \\
  0, & i \ne j.
  \end{array} \right.$
\end{lemma}
\begin{proof} First consider the case when $G$ is finite. Then, applying the orthogonality relations,
we obtain that  the space of functions on $G$ is a linear span of $\hat G$,
 and we can find such $h_i$.

Now consider the case when $G = \langle g\rangle$ is an infinite cyclic group.
Then we can take $\chi \in \hat G$, $\chi(g^k)=\lambda^k$ where $\lambda \in F^\times$
is a fixed element of an infinite order. Using the Vandermonde
argument, we obtain that $1$, $\chi$, $\chi^2$, \ldots, $\chi^{m-1}$ are linearly independent
as linear functions on $\gamma_1, \ldots, \gamma_m$, and we can find the required  $h_1, \ldots, h_m$. 

In the general case $G = G_1 \times \mathbb Z^s$ where $G_1$ is a finite group.
Hence $\hat G = \hat G_1 \times (F^\times)^s$ where $\hat G_1$ is the group of characters of $G_1$.
Now we choose the elements for each component, consider their products, and obtain
the required  $h_1, \ldots, h_m$.
\end{proof}

If $G$ is a finitely generated Abelian group and $\hat G = \Hom(G, F^\times)$, then each $G$-graded space $V=\bigoplus_{g\in G} V^{(g)}$ becomes an $F\hat G$-module:
$\chi v^{(g)} = \chi(g)v^{(g)}$ for all $\chi \in \hat G$ and $v^{(g)} \in V^{(g)}$.

We have the following important property:

\begin{lemma}\label{LemmaGrToHatG}
Let $L$ be a finite dimensional Lie algebra
over an algebraically closed field $F$ of characteristic $0$, graded by a finitely generated Abelian group $G$.
Consider the $F\hat G$-action on $L$ defined above. Then $c^{\mathrm{gr}}_n(L)=c^{F\hat G}(L)$ for all $n\in \mathbb N$.
\end{lemma}
\begin{proof}
Let $\lbrace \gamma_1, \ldots, \gamma_m \rbrace := \lbrace g\in G \mid L^{(g)}\ne 0\rbrace$.
This set is finite since $L$ is finite dimensional.

Define the homomorphism $\xi \colon L(X | F\hat G) \to L(X^{\mathrm{gr}})$ of algebras and $F\hat G$-modules
by the formula $\xi(x_i)=\sum_{i=1}^m x_i^{(\gamma_i)}$. Note that $\xi(\Id^{F\hat G}(L))\subseteq \Id^{\mathrm{gr}}(L)$
since for any homomorphism $\psi \colon L(X^{\mathrm{gr}}) \to L$ of $G$-graded algebras, $\psi$ is a homomorphism
of $F\hat G$-modules and if $f\in \Id^{F\hat G}(L)$, then $\psi(\xi(f))=0$.
Hence we can define $\tilde\xi \colon L(X | F\hat G)/\Id^{F\hat G}(L) \to L(X^{\mathrm{gr}})/\Id^{\mathrm{gr}}(L)$.

Let $h_1, \ldots, h_m \in F\hat G$ be the elements from Lemma~\ref{LemmaAbelianDual}
corresponding to $\gamma_1, \ldots, \gamma_m$. Then $a^{h_i}\in L^{(\gamma_i)}$
is the $\gamma_i$-component of $a$ for all $a \in L$ and $1 \leqslant i\leqslant m$.
In particular, \begin{equation}\label{Eqxhiequiv}
x^h - \sum_{i=1}^m h(\gamma_i)x^{h_i} \in \Id^{F\hat G} (L) \text { for all }h\in F\hat G.\end{equation}

 We define the homomorphism of algebras
$\eta \colon L(X^{\mathrm{gr}}) \to L(X | F\hat G)$ by the formula $\eta(x^{(\gamma_i)}_j)=x^{h_i}_j$
for all $1\leqslant i\leqslant m$, $j\in\mathbb N$, and $\eta(x^{(g)}_j)=0$
for $g\notin \lbrace \gamma_1, \ldots, \gamma_m \rbrace$. Note that
$\eta(\Id^{\mathrm{gr}}(L)) \subseteq \Id^{F\hat G}(L)$.
Indeed, if $\psi \colon L(X | F\hat G) \to L$ is a homomorphism of algebras and $F\hat G$-modules,
then $\psi(\eta(x^{(\gamma_i)}_j)) = \psi(x_j)^{h_i} \in L^{(\gamma_i)}$
for any choice $\psi(x_j)\in L$. Hence $\psi\eta \colon L(X^{\mathrm{gr}}) \to L$
is a graded homomorphism, $\psi(\eta(\Id^{\mathrm{gr}}(L)))=0$, and $\eta(\Id^{\mathrm{gr}}(L)) \subseteq \Id^{F\hat G}(L)$.
Thus we can define $\tilde\eta \colon L(X^{\mathrm{gr}})/\Id^{\mathrm{gr}}(L) \to  L(X | F\hat G)/\Id^{F\hat G}(L)$. 

Denote by $\bar f$ the image of a polynomial $f$ in a factor space.
Then $$\tilde\xi\tilde\eta(\bar x^{(\gamma_i)}_j)=\left(\sum_{k=1}^m \bar x^{(\gamma_k)}_j\right)^{h_i}
=\bar x^{(\gamma_i)}_j
\text{ for all }1 \leqslant i \leqslant m,\ j\in\mathbb N.$$ Since $x^{(g)}_j \in \Id^{\mathrm{gr}}(L)$
for all $g\notin \lbrace \gamma_1, \ldots, \gamma_m \rbrace$, the map $\tilde\xi\tilde\eta$ coincides with
the identity map on the generators. Hence  $\tilde\xi\tilde\eta = \id_{L(X^{\mathrm{gr}})/\Id^{\mathrm{gr}}(L)}$. By~(\ref{Eqxhiequiv}), $$\tilde\eta\tilde\xi(\bar x_j^h)=
\tilde\eta\left(\sum_{i=1}^m h(\gamma_i) \bar x_j^{(\gamma_i)}\right)=
 \sum_{i=1}^m h(\gamma_i)\bar x_j^{h_i} = \bar x_j^h \text{ for all } j \in\mathbb N,\ h \in F\hat G.$$
Similarly, we have $\tilde\eta\tilde\xi = \id_{L(X | F\hat G)/\Id^{F\hat G}}$.
Hence $$L(X | F\hat G)/\Id^{F\hat G}(L) \cong L(X^{\mathrm{gr}})/\Id^{\mathrm{gr}}(L).$$
In particular, $\frac{V^{F\hat G}_n}{V^{F\hat G}_n \cap \Id^{F\hat G}(L)}
\cong \frac{V^{\mathrm{gr}}_n}{V^{\mathrm{gr}}_n \cap \Id^{\mathrm{gr}}(L)}$
and $c^{\mathrm{gr}}_n(L)=c^{F\hat G}(L)$ for all $n\in \mathbb N$.
\end{proof}

Furthermore, $G$-graded algebras deliver us another example of an $H$-nice algebra:
\begin{example}\label{ExampleHniceGr}
Let $L$ be a finite dimensional Lie algebra
over an algebraically closed field $F$ of characteristic $0$, graded by a finitely generated Abelian group $G$.
Then $L$ is an $F\hat G$-nice algebra.
\end{example}
\begin{proof}
First, the nilpotent and the solvable radical are invariant under all automorphisms.
Hence they are $\hat G$- and $F\hat G$-invariant, and by Lemma~\ref{LemmaAbelianDual}, $G$-graded. By~\cite[Theorem~4]{ASGordienko4},
we have an $F\hat G$-invariant Levi decomposition. If $V=\bigoplus_{g\in G} V^{(g)}$ is a finite dimensional
 $G$-graded vector space, we have a natural $G$-grading on $\End_F(V)$: $\Hom_F(V^{(g_1)}, V^{(g_2)})
 \subseteq \End_F(V)^{(g_2 g_1^{-1})}$, $g_1, g_2 \in G$. (The action of the maps from 
 $\Hom_F(V^{(g_1)}, V^{(g_2)})$ on the other components is zero.)
 This $G$-grading corresponds to the natural $F\hat G$-action on $\End_F(V)$:
 $(h\psi)(v)=h_{(1)}\psi((Sh_{(2)})v)$ for $h\in F\hat G$, $\psi \in \End_F(V)$,
 $v\in V$. Hence the graded Wedderburn~--- Mal'cev theorem~\cite[Corollary~2.8]{SteVanOyst}
 implies the $F\hat G$-invariant one. (The Jacobson radical is invariant under all automorphisms.) Analogously, Condition~\ref{ConditionLComplHred} is
 a consequence of~\cite[Theorem~9]{ASGordienko4}. Hence $L$ is $F\hat G$-nice.
\end{proof}

Now we can prove a particular case of Theorem~\ref{TheoremMainGr}: 
 
\begin{theorem}\label{TheoremFinGenAbelianGr}
Let $L$ be a finite dimensional non-nilpotent Lie algebra
over a field $F$ of characteristic $0$, graded by a finitely generated Abelian group $G$. Then
there exist constants $C_1, C_2 > 0$, $r_1, r_2 \in \mathbb R$, $d \in \mathbb N$
such that $C_1 n^{r_1} d^n \leqslant c^{\mathrm{gr}}_n(L) \leqslant C_2 n^{r_2} d^n$
for all $n \in \mathbb N$.
\end{theorem} 
\begin{proof}
Graded codimensions do not change upon an extension of the base field.
The proof is analogous to the cases of ordinary codimensions of
associative~\cite[Theorem~4.1.9]{ZaiGia} and
Lie algebras~\cite[Section~2]{ZaiLie}.
Thus without loss of generality we may assume
 $F$ to be algebraically closed.

%

By Example~\ref{ExampleHniceGr}, $L$ is $F\hat G$-nice,
 and Theorem~\ref{TheoremFinGenAbelianGr} is a consequence of Theorem~\ref{TheoremMainH}
 and Lemma~\ref{LemmaGrToHatG}.
\end{proof}
 
  Now we show that the case of an arbitrary group $G$ can be reduced to the case of a
 finitely generated Abelian group.
 
We need the following known result (see e.g.~\cite[Lemma 2.1]{PaReZai}): 
\begin{lemma}\label{LemmaGradedNonZero}
Let $L$ be a Lie algebra graded by a group $G$.
Suppose $[L^{(g_1)}, \ldots, L^{(g_k)}]\ne 0
$ for some $g_1, \ldots, g_k \in G$. Then
$g_i g_j = g_j g_i$ for all $1 \leqslant i, j \leqslant k$.
\end{lemma}

Let $G_0$ be a subgroup of $G$.
Denote $L_{G_0} := \bigoplus_{g\in G_0} L^{(g)}$.

\begin{lemma}\label{LemmaInclExcl}
Let $L$ be a finite dimensional Lie algebra over a field $F$ of characteristic $0$
graded by an arbitrary group $G$.
Then there exist finitely generated Abelian subgroups $G_1, \ldots, G_r \subseteq G$
such that
\begin{equation}\label{EqInclExcl}
c^{\mathrm{gr}}_n(L)=
 \sum_{i=1}^r c^{\mathrm{gr}}_n(L_{G_i})
- \sum_{i,j=1}^r c^{\mathrm{gr}}_n(L_{G_i \cap G_j})+ \sum_{i,j,k=1}^r c^{\mathrm{gr}}_n(L_{G_i \cap G_j \cap G_k}) - \ldots +(-1)^{r-1} 
c^{\mathrm{gr}}_n(L_{G_1 \cap G_2 \cap \ldots \cap G_r}).\end{equation}
\end{lemma}
\begin{proof}
 First, we notice that
 $V^{\mathrm{gr}}_n = \bigoplus_{g_1, \ldots, g_n \in G} V_{g_1, \ldots, g_n}$
 where 
$$
 V_{g_1, \ldots, g_n} := \left\langle \left[x^{(g_{\sigma(1)})}_{\sigma(1)},
x^{(g_{\sigma(2)})}_{\sigma(2)}, \ldots, x^{(g_{\sigma(n)})}_{\sigma(n)}\right]
\mathrel{\Bigr|} \sigma\in S_n \right\rangle_F.
$$
Using the process of linearization (see e.g. \cite[Theorem 4.2.3]{Bahturin}),
we get \begin{equation}\label{EqVnGrDecomp}\frac{V^{\mathrm{gr}}_n}{V^{\mathrm{gr}}_n \cap \Id^{\mathrm{gr}}(L)}
\cong \bigoplus_{g_1, \ldots, g_n \in G} \frac{V_{g_1, \ldots, g_n}}{V_{g_1, \ldots, g_n} \cap \Id^{\mathrm{gr}}(L)}.\end{equation}

Let $\lbrace \gamma_1, \ldots, \gamma_m \rbrace := \lbrace g\in G \mid L^{(g)}\ne 0\rbrace$. This set is finite
since $L$ is finite dimensional. 
Denote $V_{n_1, \ldots, n_m} := V_{\bar g}$ where $$\bar g := (\underbrace{\gamma_1, \ldots, \gamma_1}_{n_1}, \underbrace{\gamma_2, \ldots, \gamma_2}_{n_2},
\ldots, \underbrace{\gamma_m, \ldots, \gamma_m}_{n_m}),\ n_i \in \mathbb Z_+.$$
 Then~(\ref{EqVnGrDecomp}) implies
\begin{equation}\label{EqCn1nm}
c^{\mathrm{gr}}_n(L)=\sum_{n_1+\ldots+n_m = n} \binom{n}{n_1, \ldots, n_m}
c_{n_1, \ldots, n_m}(L)
\end{equation} where $c_{n_1, \ldots, n_m}(L) := \dim \frac{V_{n_1,\ldots,n_m}}{V_{n_1,\ldots,n_m} \cap \Id^{\mathrm{gr}}(L)}$.

 Let $G_0$ be a subgroup of $G$.
 If $n_i = 0$ for all $\gamma_i \notin G_0$,
we have $c_{n_1, \ldots, n_m}(L)=c_{n_1, \ldots, n_m}(L_{G_0})$.
(By a remark in Subsection~\ref{SubsectionGraded}, it is not important whether we consider $G_0$-
or $G$-grading on $L_{G_0}$.)
 Fix $n\in\mathbb N$ and consider the sets $\Theta(G_0)$ of all $m$-tuples
$(n_1, \ldots, n_m)$ such that $n_i\geqslant 0$, $n_1 + \ldots + n_m = n$, and $n_i = 0$ for all $\gamma_i \notin G_0$.

Now we introduce a discrete measure $\mu$ on $\Theta(G)$
by the formula $$\mu\bigl((n_1, \ldots, n_m)\bigr)=\binom{n}{n_1, \ldots, n_m}
c_{n_1, \ldots, n_m}(L).$$ Then~(\ref{EqCn1nm}) implies
$c^{\mathrm{gr}}_n(L_{G_0})=\mu(\Theta(G_0))$.
By Lemma~\ref{LemmaGradedNonZero}, $c_{n_1, \ldots, n_m}(L)=0$
if $n_i,n_j \ne 0$ for some $\gamma_i \gamma_j \ne \gamma_j \gamma_i$.
Denote the set of such $m$-tuples by $\Theta_0$. Then $\mu(\Theta_0)=0$.
 Hence each nonzero
$c_{n_1, \ldots, n_m}(L)$ equals $c_{n_1, \ldots, n_m}(L_{G_0})$
for some finitely generated Abelian subgroup $G_0$ in $G$. Suppose $G_1, \ldots, G_r$ are all Abelian
subgroups in $G$ generated by subsets of $\lbrace \gamma_1, \ldots, \gamma_m \rbrace$. Then $\Theta(G)=\Theta_0 \cup \bigcup_{i=1}^r\Theta(G_i)$
and using the inclusion-exclusion principle, we get
 $$c^{\mathrm{gr}}_n(L)=\mu(\Theta(G))= \mu\left(\bigcup_{i=1}^r\Theta(G_i)\right)
= $$ $$\sum_{i=1}^r \mu(\Theta(G_i))
- \sum_{i,j=1}^r \mu(\Theta(G_i) \cap \Theta(G_j)) + \ldots +(-1)^{r-1} 
\mu(\Theta(G_1) \cap \Theta(G_2) \cap \ldots \cap \Theta(G_r))= $$ $$\sum_{i=1}^r \mu(\Theta(G_i))
- \sum_{i,j=1}^r \mu(\Theta(G_i \cap G_j)) + \ldots +(-1)^{r-1} 
\mu(\Theta(G_1 \cap G_2 \cap \ldots \cap G_r))
 =$$ $$
 \sum_{i=1}^r c^{\mathrm{gr}}_n(L_{G_i})
- \sum_{i,j=1}^r c^{\mathrm{gr}}_n(L_{G_i \cap G_j}) + \ldots +(-1)^{r-1} 
c^{\mathrm{gr}}_n(L_{G_1 \cap G_2 \cap \ldots \cap G_r}).$$\end{proof}
 
 \begin{proof}[Proof of Theorem~\ref{TheoremMainGr}]
By Theorem~\ref{TheoremFinGenAbelianGr}, each summand of~(\ref{EqInclExcl}) has the desired asymptotic behaviour for its own $d$. Take the maximal $d$. Then 
there exist $C_2 > 0$ and $r_2 \in\mathbb R$ such that
$c_n^\mathrm{gr}(L)\leqslant C_2 n^{r_2} d^n$ for all $n\in\mathbb N$. 
Our choice of $d$ implies that $\PIexp^\mathrm{gr}(L_{G_0})=d$
for some finitely generated Abelian subgroup $G_0 \subseteq G$.
Hence there exist $C_1 > 0$ and $r_1 \in\mathbb R$ such that
$C_1 n^{r_1} d^n \leqslant c_n^\mathrm{gr}(L_{G_0})\leqslant  c_n^\mathrm{gr}(L)$ for all $n\in\mathbb N$.
 Therefore, graded codimensions of $L$ satisfy the analog of the Amitsur's conjecture.
\end{proof}
 \begin{proof}[Proof of Theorem~\ref{TheoremMainGrSum}]
Suppose $$L = L_1 \oplus \ldots \oplus L_q$$ where $L_i$ are graded ideals.
First, $\PIexp^\mathrm{gr}(L)\geqslant  \PIexp^\mathrm{gr}(L_i)$ for all $1 \leqslant i \leqslant q$
since $L_i$ are graded subalgebras of $L$.
 By~Lemma~\ref{LemmaInclExcl}, it is sufficient to show that
 $$\PIexp^\mathrm{gr}(L_{G_0})\leqslant\max_{1\leqslant i\leqslant q}  \PIexp^\mathrm{gr}(L_i)$$
 for all finitely generated Abelian subgroups $G_0 \subseteq G$.
 However $L_{G_0} = (L_1)_{G_0} \oplus \ldots \oplus (L_q)_{G_0}$
 and by Lemma~\ref{LemmaGrToHatG}, Example~\ref{ExampleHniceGr}, and Theorem~\ref{TheoremMainHSum},
 $$\PIexp^\mathrm{gr}(L_{G_0})
 =\PIexp^{F\hat G_0}(L_{G_0})=\max_{1\leqslant i\leqslant q}  \PIexp^{F\hat G_0}((L_i)_{G_0})
=$$ $$\max_{1\leqslant i\leqslant q}  \PIexp^\mathrm{gr}((L_i)_{G_0}) \leqslant \max_{1\leqslant i\leqslant q}  \PIexp^\mathrm{gr}(L_i).$$
\end{proof}

Also we obtain the following upper bound:
\begin{lemma}\label{LemmaTrivialGrUpper}
Let $L$ be a finite dimensional $G$-graded Lie algebra
over a field $F$ of characteristic $0$ and let $G$ be any group.
If $G$ is finitely generated Abelian, then
$$c_n^\mathrm{gr}(L) \leqslant (\dim L)^{n+1}  \text{ for all } n\in\mathbb N.$$
Otherwise, there exists a finitely generated Abelian subgroup $G_0$
 and a constant $C \in \mathbb N$ such that
$$c_n^\mathrm{gr}(L) \leqslant C(\dim L_{G_0})^n  \text{ for all } n\in\mathbb N.$$
\end{lemma}
\begin{proof} Again, graded codimensions do not change upon an extension of the base field,
and without loss of generality we may assume
 $F$ to be algebraically closed.
 
 If $G$ is finitely generated Abelian, we apply Lemmas~\ref{LemmaCodimDim}, \ref{LemmaGrToHatG},
 and get the first assertion.
 
Using Lemma~\ref{LemmaInclExcl} and choosing the maximal $\dim L_{G_0}$,
we obtain the second assertion from the first one.
\end{proof}

\subsection{Applications to $G$-codimensions}\label{SubsectionApplG}

First, we prove that in the case of Lie algebras we can restrict ourselves to the
case when a group acts by automorphisms only.

\begin{lemma}\label{LemmaGtoAut} Let $G$ be a group with a fixed subgroup
$G_0$ of index ${} \leqslant 2$. Let $L$ be a Lie $G$-algebra.
Denote by $\tilde G$ the group isomorphic to $G$
and by $\tilde g \in \tilde G$ the element corresponding to $g\in G$
 under the isomorphism $\tilde G \cong G$.
 Then $L$ is a $\tilde G$-algebra
where $\tilde G$ acts by automorphisms only and the $\tilde G$-action is defined
by the formula
$$a^{\tilde g}=\left\lbrace\begin{array}{rll}
 a^g & \text{if} & g\in G_0, \\
 -a^g & \text{if} & g\in G\backslash G_0 
 \end{array}\right. \text{ for all } a\in L.$$
 Conversely, each Lie $\tilde G$-algebra
 where the $\tilde G$ acts by automorphisms
 is a Lie $G$-algebra where $G_0$ acts by automorphisms
 and $G\backslash G_0$ acts by anti-automorphisms.
  Moreover, $c_n^{\tilde G}(L)=c_n^G(L)$ for all $n \in \mathbb N$. If $G$
  is an affine algebraic group acting on $L$ rationally,
  then $\tilde G$ acts on $L$ rationally too.
  (We assume that the structure of an affine algebraic variety on $\tilde G$ is
  the same as on $G$.)
\end{lemma}
\begin{proof} 
Note that
$$[a,b]^{\tilde g}=[a^{\tilde g}, b^{\tilde g}]\text{
for all }g \in G.$$ Hence $L$ is a $\tilde G$-algebra. The converse implication is evident.
In this case we define
$$a^g=\left\lbrace\begin{array}{rll}
 a^{\tilde g} & \text{if} & g\in G_0, \\
 -a^{\tilde g} & \text{if} & g\in G\backslash G_0 
 \end{array}\right. \text{ for all } a\in L.$$

Therefore, we have an isomorphism $\psi \colon L(X | G) \to L(X|\tilde G)$ of
$G$- and $\tilde G$-algebras such that $\psi(\Id^G(L))=\Id^{\tilde G}(L)$, $\psi(V^G_n)=V^{\tilde G}_n$,
and therefore $c_n^{\tilde G}(L)=c_n^G(L)$ for all $n \in \mathbb N$.

If $G$ is an affine algebraic group, then $G_0$ is a closed subgroup of index $\leqslant 2$ and $G$ is the disjoint union of closed subsets $G_0$ and $G\backslash G_0$.
Denote by $\mathcal O(G)$ the algebra of polynomial functions on $G$.
Let $$I_1 = \lbrace f \in \mathcal O(G) \mid f(g)=0 \text{ for all } g \in G_0 \rbrace$$
and $$I_2 = \lbrace f \in \mathcal O(G) \mid f(g)=0 \text{ for all } g \in G\backslash G_0 \rbrace.$$
Since $G_0 \cap (G\backslash G_0) = \varnothing$, by Hilbert's Nullstellensatz, $I_1 + I_2 = \mathcal O(G)$.
Hence $1 = f_1 + f_2$ where $f_1 \in I_1$, $f_2 \in I_2$. Then $f_2(g) - f_1(g) = 1$ for all $g\in G_0$
and $f_2(g) - f_1(g)= -1$ for all $g\in G\backslash G_0$.
Therefore, if we multiply all operators from $G\backslash G_0$ by $(-1)$, the representation still will be 
rational.
\end{proof}

  The following lemma enables to derive properties of $G$-codimensions
from properties of $H$-codimensions.

\begin{lemma}\label{LemmaFGcodim} Let $L$ be a Lie algebra with the action
of a group $G$ by automorphisms. Then $L$ is an $FG$-module algebra
where the action of the Hopf algebra $FG$ on $L$
is the extension of the $G$-action by linearity.
 Conversely, each Lie $FG$-module algebra
 is a Lie $G$-algebra.
 Moreover, $c_n^{FG}(L)=c_n^G(L)$ for all $n \in \mathbb N$.
\end{lemma}
\begin{proof} 
If we treat $G$ as a subgroup in $FG$, we obtain
$$g[a,b]=[ga, gb] = [g_{(1)}a,g_{(2)}b]\text{
for all }g \in G.$$ Using the linearity, we get~(\ref{EqHmoduleLieAlgebra})
for $H=FG$. Hence $L$ is an $FG$-module algebra. The converse implication is evident.

Therefore, we have an isomorphism $\psi \colon L(X | G) \to L(X|FG)$ of
$FG$-module and $G$-algebras such that $\psi(\Id^G(L))=\Id^{FG}(L)$, $\psi(V^G_n)=V^{FG}_n$,
and therefore $c_n^{FG}(L)=c_n^G(L)$ for all $n \in \mathbb N$.
\end{proof}

\begin{proof}[Proof of Theorem~\ref{TheoremMainGAffAlg}]
We apply Lemmas~\ref{LemmaGtoAut}, \ref{LemmaFGcodim}, Example~\ref{ExampleHniceAffAlg}, and Theorem~\ref{TheoremMainH}.
\end{proof}
\begin{proof}[Proof of Theorem~\ref{TheoremMainGFin}]
We apply Lemmas~\ref{LemmaGtoAut}, \ref{LemmaFGcodim}, and Theorem~\ref{TheoremMainHSS}.
\end{proof}
\begin{proof}[Proof of Theorem~\ref{TheoremMainGAffAlgSum}]
We apply Lemmas~\ref{LemmaGtoAut}, \ref{LemmaFGcodim}, Example~\ref{ExampleHniceAffAlg}, and Theorem
\ref{TheoremMainHSum}.
\end{proof}
\begin{proof}[Proof of Theorem~\ref{TheoremMainGFinSum}]
We apply Lemmas~\ref{LemmaGtoAut}, \ref{LemmaFGcodim}, and Theorem~\ref{TheoremMainHSSSum}.
\end{proof}

Also we obtain the following propositions.

\begin{lemma}
Let $L$ be a finite dimensional Lie algebra with $G$-action
over any field $F$ and let $G$ be any group. Then
$$c_n^G(L) \leqslant (\dim L)^{n+1}  \text{ for all } n\in\mathbb N.$$
\end{lemma}
\begin{proof}
We apply Lemmas~\ref{LemmaCodimDim}, \ref{LemmaGtoAut}, and~\ref{LemmaFGcodim}.
\end{proof}

\begin{lemma}
Let $L$ be a Lie algebra with $G$-action 
over any field $F$ and let $G$ be any group.
Then
$$c_n(L) \leqslant c_n^G(L) \leqslant |G|^n c_n(L) \text{ for all } n\in\mathbb N$$
 where
$c_n(L)$ are ordinary codimensions.
\end{lemma}
\begin{proof}
We apply Lemmas~\ref{LemmaOrdinaryAndHopf}, \ref{LemmaGtoAut}, and~\ref{LemmaFGcodim}.
\end{proof}

\section{Examples and criteria for simplicity}
\label{SectionExamples}

In this section we assume $F$ to be an algebraically closed field
  of characteristic $0$.
  
    \begin{example}\label{ExampleHSimple}
 Let $B$ be a finite dimensional semisimple $H$-module Lie algebra where
  $H$ is a Hopf algebra. If $B$ is $H$-simple, then there exist $C > 0$, $r \in \mathbb R$
  such that $$C n^r (\dim B)^n \leqslant c_n^{H}(B)
  \leqslant (\dim B)^{n+1} \text{ for all } n\in\mathbb N.$$
 \end{example}
 \begin{proof}
 The upper bound follows from Lemma~\ref{LemmaCodimDim}.
 The lower bound is a consequence of Theorem~\ref{TheoremMainH}
 and Example~\ref{ExamplePIexpBSS}. 
 \end{proof}
 
 By~\cite[Theorem~6]{ASGordienko4}, every finite dimensional 
 semisimple $H$-module Lie algebra is the direct sum of $H$-simple
 Lie algebras. Here we calculate its Hopf PI-exponent.

  \begin{example}\label{ExampleHSemiSimple}
 Let $L=B_1 \oplus B_2 \oplus \ldots \oplus B_q$
  be a finite dimensional semisimple $H$-module Lie algebra where
  $H$ is a Hopf algebra and $B_i$ are $H$-simple Lie
  algebras. Let $d := \max_{1 \leqslant k
  \leqslant q} \dim B_k$. Then there exist $C_1, C_2 > 0$, $r_1, r_2 \in \mathbb R$
  such that $$C_1 n^{r_1} d^n \leqslant c_n^{H}(L)
  \leqslant C_2 n^{r_2} d^n \text{ for all } n\in\mathbb N.$$
 \end{example}
 \begin{proof}
 This is a consequence of Theorem~\ref{TheoremMainH} and Example~\ref{ExamplePIexpBSS}.
 \end{proof}
 
 Now we obtain a criterion for $H$-simplicity:

\begin{theorem}\label{TheoremHCrSimple}
Let $L$ be an $H$-nice Lie algebra where $H$ is a Hopf algebra over $F$.
Then $\PIexp^H(L)=\dim L$ if and only if $L$ is semisimple and $H$-simple.
\end{theorem}
\begin{proof} If $L$ is an $H$-simple semisimple algebra, then $\PIexp^H(L)=\dim L$
by Example~\ref{ExampleHSimple}.

Suppose $\PIexp^H(L)=\dim L$.
Let $N$ be the nilpotent radical of $L$. Let $I_1, \ldots, I_r$, $J_1, \ldots, J_r$
are $H$-invariant ideals of $L$ satisfying Conditions 1--2 (see Subsection~\ref{SubsectionHopfPIexp}).
By Lemma~\ref{LemmaIrrAnnBS}, $N \subseteq \Ann(I_1/J_1) \cap \dots \cap \Ann(I_r/J_r)$.
Hence $\PIexp^H(L) \leqslant (\dim L) - (\dim N)$. Therefore $N=0$ and by~\cite[Proposition 2.1.7]{GotoGrosshans}, $[L, R] \subseteq N = 0$ where $R$ is the solvable radical of $L$. Hence $R = Z(L)\subseteq N=0$ and $L$ is semisimple. Now we apply Example~\ref{ExampleHSemiSimple}.
\end{proof}

In particular, Theorem~\ref{TheoremHCrSimple}
holds for finite dimensional $H$-module Lie algebras in the case when $H$ is finite dimensional semisimple (see Example~\ref{ExampleHniceHSS}).

Applying Lemmas~\ref{LemmaGtoAut} and~\ref{LemmaFGcodim}, we get
\begin{theorem}\label{TheoremGFinCrSimple}
Let $L$ be a finite dimensional Lie algebra over $F$. Suppose a finite
  group $G$ acts on $L$ by automorphisms and anti-automorphisms.
Then $\PIexp^G(L)=\dim L$ if and only if $L$ is a $G$-simple algebra.
\end{theorem}

Applying Lemmas~\ref{LemmaGtoAut}, \ref{LemmaFGcodim} and Example~\ref{ExampleHniceAffAlg}, we get
\begin{theorem}\label{TheoremGAffAlgCrSimple}
Let $L$ be a finite dimensional Lie algebra over $F$. Suppose a reductive affine algebraic group
  group $G$ acts on $L$ rationally by automorphisms and anti-automorphisms.
Then $\PIexp^G(L)=\dim L$ if and only if $L$ is a $G$-simple algebra.
\end{theorem}

Moreover, we can apply the results above to graded Lie algebras.

  \begin{example}\label{ExampleGrSimple}
 Let $B=\bigoplus_{g\in G} B^{(g)}$ be a
  finite dimensional graded simple Lie algebra where $G$ is an arbitrary group.
  Then there exist $C > 0$, $r \in \mathbb R$
  such that $$C n^r (\dim B)^n \leqslant c_n^\mathrm{gr}(B)
  \leqslant (\dim B)^{n+1} \text{ for all } n\in\mathbb N.$$
 \end{example}
 \begin{proof} Suppose $g_1, g_2 \in G$, $g_1g_2 \ne g_2 g_1$.
 Denote by $I_j$ the (graded) ideal generated by $B^{(g_j)}$, $j=1,2$.
 Using Lemma~\ref{LemmaGradedNonZero} and the Jacobi identity, we get $[I_1, I_2]=0$.
 Since $B$ is graded simple, the finite set $\lbrace g \in G \mid  B^{(g)} \ne 0\rbrace$
 consists of commuting elements, and we may assume that $G$ is finitely generated Abelian.
 Note that the solvable radical of $L$ is $\hat G$-invariant and hence graded.
 Thus $B$ is an $F\hat G$-simple semisimple Lie algebra,
 and we get the bounds from Lemma~\ref{LemmaGrToHatG} and Example~\ref{ExampleHSimple}.
 \end{proof}
 
 Note that if $L$ is a finite dimensional semisimple graded Lie algebra, then
 by~\cite[Theorem~7]{ASGordienko4}, $L$ is isomorphic to the direct sum
 of simple graded algebras.

  \begin{example}\label{ExampleGrSemiSimple}
 Let $L=B_1 \oplus B_2 \oplus \ldots \oplus B_q$
  be a finite dimensional semisimple Lie algebra graded by any group, where $B_i$ are graded simple
  algebras. Let $d := \max_{1 \leqslant k
  \leqslant q} \dim B_k$. Then there exist $C_1, C_2 > 0$, $r_1, r_2 \in \mathbb R$
  such that $C_1 n^{r_1} d^n \leqslant c_n^{\mathrm{gr}}(L)
  \leqslant C_2 n^{r_2} d^n$ for all $n\in\mathbb N$.
 \end{example}
 \begin{proof}
 This follows immediately from Theorems~\ref{TheoremMainGr}, \ref{TheoremMainGrSum}  and Example~\ref{ExampleGrSimple}.
 \end{proof}

\begin{theorem}\label{TheoremGrCrSimple}
Let $L$ be a finite dimensional Lie algebra over $F$ graded by an arbitrary group.
Then $\PIexp^\mathrm{gr}(L)=\dim L$ if and only if $L$ is a graded simple algebra.
\end{theorem}
\begin{proof}
If $L$ is a graded simple algebra, then $\PIexp^\mathrm{gr}(L)=\dim L$
by Example~\ref{ExampleHSimple}. 

Suppose $\PIexp^\mathrm{gr}(L)=\dim L$. By Lemma~\ref{LemmaTrivialGrUpper},
we may assume $G$ to be finitely generated Abelian. Now we use Lemma~\ref{LemmaGrToHatG},
Example~\ref{ExampleHniceGr}, and Theorem~\ref{TheoremHCrSimple}.
\end{proof}

We conclude the section with an example of a non-semisimple algebra graded by a non-Abelian group.
\begin{example}\label{Example2gl2S3}
Let $G=S_3$ and $L=\mathfrak{gl}_2(F)\oplus \mathfrak{gl}_2(F)$. Consider the following $G$-grading
on~$L$: $$L^{(e)} = \left\lbrace\left(\begin{array}{rr}
\alpha & 0 \\
 0 & \beta 
\end{array} \right)\right\rbrace \oplus  \left\lbrace\left(\begin{array}{rr}
\gamma & 0 \\
 0 & \mu 
\end{array} \right)\right\rbrace,$$ $$L^{\bigl((12)\bigr)} = \left\lbrace\left(\begin{array}{rr}
0 & \alpha \\
 \beta & 0  
\end{array} \right)\right\rbrace \oplus  0,\qquad L^{\bigl((23)\bigr)} = 0 \oplus \left\lbrace\left(\begin{array}{rr}
0 & \alpha \\
 \beta & 0  
\end{array} \right)\right\rbrace,$$ the other components are zero.
Then there exist $C_1, C_2 > 0$, $r_1, r_2 \in \mathbb R$
  such that $$C_1 n^{r_1} 3^n \leqslant c_n^{\mathrm{gr}}(L)
  \leqslant C_2 n^{r_2} 3^n \text{ for all } n \in\mathbb N.$$
\end{example}
\begin{proof}
Note that $L=\mathfrak{sl}_2(F)\oplus \mathfrak{sl}_2(F)\oplus Z(L)$ where
the center $Z(L)$ consists of scalar matrices of both copies of $\mathfrak{gl}_2(F)$.
Hence $V^{\mathrm{gr}}_n\cap\Id^{\mathrm{gr}}(L)=V^{\mathrm{gr}}_n\cap\Id^{\mathrm{gr}}(\mathfrak{sl}_2(F) \oplus \mathfrak{sl}_2(F))$ for all $n\in \mathbb N$.
Now we notice that both copies of $\mathfrak{sl}_2(F)$ are simple graded ideals of $L$ and apply Example~\ref{ExampleGrSemiSimple}.
\end{proof}

Many examples of Lie algebras with a $G$-grading and $G$-action are considered in~\cite[Subsection 1.5]{ASGordienko2}.

\section*{Acknowledgements}

I am grateful to Yuri Bahturin who suggested that I study polynomial $H$-identities.
In addition, I appreciate Mikhail Kotchetov for helpful
discussions.

\end{document}